\documentclass[11pt]{article}

\usepackage{bm}
\usepackage{amsmath}
\usepackage{amsthm}
\usepackage{longtable}
\usepackage{booktabs}
\usepackage{graphicx}
\usepackage{ulem}
\usepackage{bm}
\usepackage{amsopn}
\numberwithin{equation}{section}
\usepackage{color,times}

\usepackage{amsmath}
\usepackage{amssymb}
\usepackage{amsfonts}
\usepackage{mathrsfs}
\usepackage{cite}
\usepackage{cleveref}
\usepackage{float}
\usepackage{graphicx}
\newtheorem{theorem}{Theorem}[section]

\newtheorem{lemma}[theorem]{Lemma}
\newtheorem{remark}[theorem]{Remark}
\newtheorem{definition}[theorem]{Definition}

\newtheorem{proposition}[theorem]{Proposition}
\numberwithin{equation}{section}

\newenvironment{proof2.3}{{\noindent \it Proof of Theorem \ref{thm2.3}.}}{\hfill $\Box$}
\newenvironment{proof2.4}{{\noindent \it Proof of Theorem \ref{thm2.4}.}}{\hfill $\Box$}
\newenvironment{proof2.6}{{\noindent \it Proof of Theorem \ref{thm2.6}.}}{\hfill $\Box$}
\newenvironment{proof2.7}{{\noindent \it Proof of Theorem \ref{thm2.7}.}}{\hfill $\Box$}
\allowdisplaybreaks[4]

\begin{document}

\setlength{\baselineskip}{16pt}{\setlength\arraycolsep{2pt}}

\title{The horizontal magnetic primitive equations approximation of the anisotropic MHD equations in a thin 3D domain}

\author{Jie Zhang$^1$\ \ \ and   \ Wenjun Liu$^{1,2,}$\footnote{Corresponding author. \ \   Email address: wjliu@nuist.edu.cn (W. J. Liu); jiezhang@nuist.edu.cn (J. Zhang)}   \medskip\\ 1. School of Mathematics and Statistics, Nanjing University of Information  \\
	Science and Technology, Nanjing 210044, China
	\medskip\\
	2. Jiangsu International Joint Laboratory on System Modeling and Data Analysis, \\ Center for Applied Mathematics of Jiangsu Province, Nanjing University \\ of Information Science and Technology, Nanjing 210044, China
}



\date{}
\maketitle

\begin{abstract}
In this paper, we give a rigorous justification of the deviation of the primitive equations with only horizontal viscosity and magnetic diffusivity (PEHM) as the small aspect ratio limit of the incompressible three-dimensional scaled horizontal viscous MHD (SHMHD) equations. Choosing an aspect ratio parameter $\varepsilon \in(0,\infty)$, we consider the case that if the horizontal and vertical viscous coefficients are of $\mu = O(1)$ and $\nu = O({\varepsilon ^\alpha })$, and the orders of magnetic diffusion coefficients $k$ and $\sigma$ are $k = O(1)$ and $\sigma = O({\varepsilon ^\alpha })$, with $\alpha > 2$, then the limiting system is the PEHM as $\varepsilon$ goes to zero. For ${H^1}$-initial data, we prove that the global weak solutions of the SHMHD equations converge strongly to the local-in-time strong solutions of the PEHM, as $\varepsilon$ tends to zero. For ${H^1}$-initial data with additional regularity $({\partial _z}{\tilde A_0},{\partial _z}{\tilde B_0}) \in {L^p}(\Omega )(2<p<\infty)$, we slightly improve the well-posed result in \cite{2017-Cao-Li-Titi-Global} to extend the local-in-time strong convergences to the global-in-time one. For ${H^2}$-initial data, we show that the local-in-time strong solutions of the SHMHD equations converge strongly to the global-in-time strong solutions of the PEHM, as $\varepsilon$ goes to zero. Moreover, the rate of convergence is of the order $O({\varepsilon ^{\gamma /2}})$, where $\gamma  = \min \{ 2,\alpha  - 2\}$ with $\alpha \in (2,\infty )$. It should be noted that in contrast to the case $\alpha > 2$, the case $\alpha =2$ has been investigated by Du and Li in \cite{2023-Du-Li}, in which they consider the PEM and the rate of global-in-time convergences is of the order $O(\varepsilon)$.
\end{abstract}

\maketitle

\section{Introduction }
Magnetohydrodynamic (MHD), which was initiated in 1970 by Nobel Prize-winning physicist Hannes Alfvén, is a discipline dedicated to the study of the interaction between conductive fluids and magnetic fields. The physical concept behind MHD implies that magnetic fields can induce electric currents in a moving conductive fluid, which in turn exerts the forces acting on the fluid and also influences the magnetic field itself. MHD is widely used in a variety of fields, such as space exploration, geophysics and space weather \cite{2000-Groth-De-Gombosi,2004-Manchester-Gombosi}, astrophysics \cite{1991-Priest-Hood,1988-Zirin}, metallurgy and materials processing, and so on.

A comprehensive set of equations describing MHD phenomena can be generated by combining the Navier-Stokes equations of fluid dynamics with Maxwell’s equations of electromagnetic dynamics. The MHD equations are regarded as the fundamental model for describing the behavior and dynamics of plasma (see \cite{1972-Duvaut-Lions-1,1983-Sermange-Temam}). Similar to the three-dimensional (3D) incompressible Navier-Stokes equations
, there exists at least one global weak solution with finite energy to the three-dimensional incompressible MHD equations. In particular, a pioneering work is given by Duvaut and Lions in \cite{1972-Duvaut-Lions-1} (see also \cite{1983-Sermange-Temam}), in which the global-in-time weak solutions to the 3D MHD equations was constructed for any incompressible initial pair $({u_0},{b_0}) \in {L^2}(R^3 )$ and the existence of the local-in-time strong solutions to 3D MHD equations was proved for any incompressible initial pair $({u_0},{b_0}) \in {H^k}(R^3)$ with $k>5/2$. But until now, for 3D case, the subject of the regularity and uniqueness of the weak solutions or even that of the global existence of the local-in-time strong solutions remains an outstanding challenging open problem in the field of mathematical fluid dynamics (see \cite{1972-Duvaut-Lions-1,1983-Sermange-Temam,2009-Xiao-Xin-Wu,2022-Duan-Xiao-Xin}). Meanwhile, for the 2D MHD equations with partial dissipation and magnetic diffusion, it is a difficult issue whether or not local classical solutions can develop finite-time singularities \cite{2011-Cao-Wu}. On the other hand, the regularity criteria for weak solutions to the MHD equations have been studied extensively (see \cite{2005-He-xin,2005-Zhou,2007-He-Wang,2008-Chen-Miao-Zhang,2022-Pan,2023-Wang}). Among these references, the work in \cite{2005-He-xin} by He and Xin (and the work in \cite{2005-Zhou} by Zhou, independently) should be noted, in which they established two well-known Ladyzhenskaya-Prodi-Serrin (LPS, for short) regularity criteria by only the information of the fluid velocity. Concerning the small perturbation solutions in axially symmetric case, the 3D ideal resistive MHD equations was investigated by Lei \cite{2015-Leizhen} and the 3D non-homogeneous incompressible MHD equations was considered by Su et al.\cite{2017-Su-Guo-Yang}. For the 3D non-resistive viscous MHD fluids, it is evident that a strong enough magnetic field can effectively prevent the formation of singularities, and even can also inhibit flow instabilities from occurring \cite{2014-Jiangfei-Jiangsong-wangyanjin,2015-Jiangfei-Jiangsong,2018-Jiangfei-Jiangsong}. Interesting readers can refer to \cite{2018-Tanzhong-Wangyanjin} and \cite{2019-Jangfei-Songjiang} for the global stability solution and the existence of large perturbation solutions, respectively.

The hydrostatic approximation can be found both in the geophysics and in the study of the global large-scale atmospheric and oceanic dynamics (see \cite{1987-Pedlosky,1986-Washington-Parkinson,2003-Majda,2006-Vallis}). It is well-known that, owing to the fact that the planetary horizontal scale is significantly larger than the vertical scale of the ocean and atmosphere, the hydrostatic approximation can be utilized to simulate the motion of the ocean and atmosphere in the vertical direction. On basis of this fact and its high accuracy property, we apply either the scale analysis technique or take the small aspect ratio limit to the 3D anisotropic Navier-Stokes equations of incompressible flows to derive the 3D primitive equations (PEs) which is the most simplest form. Similar to the derivation of the 3D primitive equations, the 3D primitive equations coupled to the magnetic field (PEM) is obtained by considering the 3D anisotropic MHD equations of incompressible flows, which is regarded as the extension of the simplest form of the PEs. Until now, in the case from the Navier-Stokes equations to the PEs, the weak convergences were first obtained by Azérad and Guillén \cite{2001-Azérad-Guillén}, then the strong convergences with error estimates were given by Li and Titi \cite{2019-Li-Titi}, and finally the strong convergences were proved by Furukawa et al. \cite{2020-Furukawa-Giga-Hieber-Hussein-Kashiwabara-Wrona} under relaxing the regularity on the initial condition. Later, the case from the scaled Navier-Stokes equations to the PEs with only horizontal viscosity was considered by Li, Titi and Yuan \cite{2022-Li-Titi}. Moreover, the rigorous justification of the small aspect ratio limit from the scaled MHD equations to the PEM was investigated by Du and Li \cite{2023-Du-Li}.

In the 1990s, the PEs with full viscosities and full diffusivity were first introduced and systematically studies by Lions, Temam and Wang in \cite{1992-Lions-Temam-Wang-1,1992-Lions-Temam-Wang-2,1995-Lions-Temam-Wang}. They established the global existence of weak solutions to the PEs for initial data in $L^2$, while the issue of uniqueness of weak solutions is still unclear until today, even for 2D case. For initial data in $H^1$, the local existence of strong solutions was first proved by Guillén-González et al. \cite{2001-Guillén-González-Masmoudi-Rodríguez-Bellido}. In a breakthrough paper, a significant progress on the global well-posedness of strong solutions of the three-dimensional atmospheric and oceanic PEs for arbitrary large initial data in $H^1$ was made by Cao and Titi \cite{2007-Cao-Titi}, where their models include a linear equation of state and consider the Neumann boundary conditions. Successively, considering the mixed Neumann and Dirichlet case, the strong well-posedness result was obtained by Kukavica and Ziane \cite{2007-Kukavica-Ziane}. The different approaches were also utilized in \cite{2006-Kobelkov,2014-Kukavica-Pei-Rusin-Ziane}. Under the mixed Neumann and Dirichlet case, the global existence of strong solutions to the PEs without temperature was proved by Hieber and Kashiwabara \cite{2016-Hieber-Kashiwabara} for initial data in $L^p$, and subsequently by Giga et al. \cite{2020-Giga-Gries-Hieber-Hussein} for initial data belonging to the anisotropic $L^p$-spaces. Moreover, the global well-posedness result corresponding to the PEs without temperature can be further generalized to the oceanic PEs, see Hieber et al. \cite{2016-Hieber-Hussein-Kashiwabara}. Recently, Korn \cite{2021-PeterKorn} considered two different boundary conditions and a nonlinear equation of state, and established the existence and uniqueness of global strong solutions of the oceanic PEs. For more results on the PEs coupled with atmospheric dynamics, one can refer to \cite{2006-Guo-Huang,2013-Zelati-Frémond-Temam-Tribbia,2017-Hittmeir-Klein-Li-Titi,2020-Hittmeir-Klein-Li-Titi}. On the other hand, concerning the subject of the partial anisotropic viscosity and diffusivity, the global well-posedness of strong solutions to the PEs was established by Cao, Li and Titi in \cite{2014-Cao-Li-Titi-Global,2014-Cao-Li-Titi-Local-Global,2016-Cao-Li-Titi-Global,2017-Cao-Li-Titi-Global,2020-Cao-Li-Titi-Global}. For the inviscid PEs (or hydrostatic Euler equations), the short time existence of its solutions has been proved by Brenier et al.\cite{1999-Brenier,2011-Kukavica-Temam-Vicol-Ziane,2012-Masmoudi-Wong} in the analytic function space and in $H^s$ space, and for the inviscid 3D PEs in the absence of rotation, the finite-time blowup has been constructed in \cite{2014-Wong,2015-Cao-Ibrahim-Nakanishi-Titi}.

When $\alpha = 2$, the limiting system of the scaled MHD equations is the primitive equations with magnetic field,  which has been studied in \cite{2023-Du-Li}. However, when $\alpha > 2$, the limiting system of the scaled horizontal viscous MHD (SHMHD) equations is the primitive equations with only horizontal viscosity and magnetic diffusivity (PEHM). Motivated by the excellent work of 
Li, Titi and Yuan \cite{2022-Li-Titi},
the main purpose of this paper is to prove that the SHMHD equations strongly converge to the PEHM, in the cases of $H^1$-initial data, $H^1$-initial data with additional regularity and $H^2$-initial data, respectively, as the aspect ratio parameter goes to zero. Our convergence results are briefly stated as follows.


(i) In the case of ${H^1}$-initial data, there exists a unique local-in-time strong solution $(\tilde A,\tilde B)$ of the PEHM subject to conditions \eqref{2.4}-\eqref{2.6} below (see \cite{2017-Cao-Li-Titi-Global}). According to Proposition  \ref{prop3.2} and Remark \ref{rem2.2} below, we obtain the local-in-time strong convergence theorem (see Theorem \ref{thm2.3} below).

(ii) In the case of ${H^1}$-initial data with additional regularity $({\partial _z}{\tilde A_0},{\partial _z}{\tilde B_0}) \in {L^p}(\Omega )$ for some $p \in (2,\infty )$, the global existence of strong solutions of the PEHM subject to conditions \eqref{2.4}-\eqref{2.6} is proved (see Theorem \ref{thm2.4} below). Note that compared with the global well-posed result in \cite{2017-Cao-Li-Titi-Global}, we do not require the condition $({{\tilde A}_0},{{\tilde B}_0}) \in {L^\infty }(\Omega )$. As a result, combining this improved result and Remark \ref{rem2.2}, we establish the corresponding global-in-time strong convergence Theorem (see Theorem \ref{thm2.6} below).

(iii) In the case of ${H^2}$-initial data, the PEHM subject to conditions \eqref{2.4}-\eqref{2.6} admits a unique global-in-time strong solution $(\tilde A,\tilde B)$ (see \cite{2016-Cao-Li-Titi-Global}). On the basis of Proposition 4.1 and Remark \ref{rem2.2}, we obtain the corresponding global-in-time strong convergence Theorem (see Theorem \ref{thm2.7} below).

The structure of our paper is arranged as follows. In section 2, we derive two important mathematical models and state main results which will be proved in the subsequent sections. In section 3, we prove first the local-in-time strong convergence of Theorem \ref{thm2.3}, then establish the global well-posedness of Theorem \ref{thm2.4}, and finally prove the global-in-time strong convergence of Theorem \ref{thm2.6}, where we will impose additional regularity $({\partial _z}{\tilde A_0},{\partial _z}{\tilde B_0}) \in {L^p}(\Omega )$ with $p \in (2,\infty )$ on the initial data $({{\tilde A}_0},{{\tilde B}_0}) \in {H^1}(\Omega )$. Finally, in section 4, we show the global-in-time strong convergence of Theorem \ref{thm2.7}, where we will consider the initial data $({{\tilde A}_0},{{\tilde B}_0}) \in {H^2}(\Omega )$.

\section{Mathematical models and main results }
\subsection{The derivation of the scaled horizontal viscous MHD equations (SHMHD)}
Consider a bounded domain $M: = (0,{L_1}) \times (0,{L_2}) \subseteq {R^2}$, where ${L_1}$ and ${L_2}$ are both positive constants of order $O(1)$ about $\varepsilon$. Let ${\Omega _\varepsilon }: = M \times ( - \varepsilon ,\varepsilon ) \subseteq {R^3}$ be a thin domain depending on $\varepsilon$, where $\varepsilon = H/L > 0$ is a small enough parameter. The anisotropic incompressible MHD equations in ${\Omega _\varepsilon }$ are listed as
\begin{equation} \label{2.1}
 \begin{cases}
  \begin{aligned}
                   &{\partial _t}u + (u \cdot \nabla )u + \nabla p - \mu {\Delta _H}u - \nu {\partial _{zz}}u = (b \cdot \nabla )b,\\
                   &{\partial _t}b + (u \cdot \nabla )b - k{\Delta _H}b - \sigma {\partial _{zz}}b = (b \cdot \nabla )u,\\
                   &\nabla  \cdot u = 0,\;\;\nabla  \cdot b = 0,
  \end{aligned}
 \end{cases}
\end{equation}
where the velocity field $u = ({\tilde u},{u_3})$, the magnetic field $b = ({\tilde b},{b_3})$ and the scalar pressure function $p$ are all the unknowns. Here the notations ${\nabla _H} = ({\partial _x},{\partial _y})$ and ${\Delta _H} = \partial _x^2 + \partial _y^2$ denote the horizontal gradient and the horizontal Laplacian, respectively; the notation $\nabla = ({\partial _x},{\partial _y},{\partial _z})$ denotes the three-dimensional gradient operator, $\Delta = \partial _x^2 + \partial _y^2 + \partial _z^2$ and $\nabla \cdot $ the three-dimensional Laplacian and the three-dimensional divergence operator; the vector fields ${\tilde u} = ({u_1},{u_2})$  and ${\tilde b} = ({b_1},{b_2})$ denote the horizontal velocity and the horizontal magnetic field, respectively, while ${u_3}$ and ${b_3}$ the vertical one.

Let us introduce the following scaling transformation
\begin{equation*}
      \begin{aligned}
               &{{\tilde u}_\varepsilon }(x,y,z,t) = \tilde u(x,y,\varepsilon z,t) = \left( {{u_1}(x,y,\varepsilon z,t),{u_2}(x,y,\varepsilon z,t)} \right),\\
               &{{\tilde b}_\varepsilon }(x,y,z,t) = \tilde b(x,y,\varepsilon z,t) = \left( {{b_1}(x,y,\varepsilon z,t),{b_2}(x,y,\varepsilon z,t)} \right),\\
               &{u_{3,\varepsilon }}(x,y,z,t) = \frac{1}{\varepsilon }{u_3}(x,y,\varepsilon z,t),\;{b_{3,\varepsilon }}(x,y,z,t) = \frac{1}{\varepsilon }{b_3}(x,y,\varepsilon z,t),\\
               &{p_\varepsilon }(x,y,z,t) = p(x,y,\varepsilon z,t),\;{u_\varepsilon } = ({{\tilde u}_\varepsilon },{u_{3,\varepsilon }}),\;{b_\varepsilon } = ({{\tilde b}_\varepsilon },{b_{3,\varepsilon }}),\;\forall (x,y,z) \in M \times ( - 1,1),
  \end{aligned}
\end{equation*}
and assume that $\mu = 1$, $\nu = {\varepsilon ^\alpha }$, $k = 1$, $\sigma = {\varepsilon ^\alpha }$, respectively. Then system \eqref{2.1} can be rewritten as the scaled horizontal viscous (or diffusive) MHD (SHMHD) equations
\begin{equation} \label{2.2}
 \begin{cases}
 \begin{aligned}
               &{\partial _t}{{\tilde u}_\varepsilon } + ({u_\varepsilon } \cdot \nabla ){{\tilde u}_\varepsilon } - ({b_\varepsilon } \cdot \nabla ){{\tilde b}_\varepsilon } + {\nabla _H}{p_\varepsilon } - {\Delta _H}{{\tilde u}_\varepsilon } - {\varepsilon ^{\alpha  - 2}}{\partial _{zz}}{{\tilde u}_\varepsilon } = 0,\\
               &{\varepsilon ^2}\left( {{\partial _t}{u_{3,\varepsilon }} + {u_\varepsilon } \cdot \nabla {u_{3,\varepsilon }} - {b_\varepsilon } \cdot \nabla {b_{3,\varepsilon }} - {\Delta _H}{u_{3,\varepsilon }} - {\varepsilon ^{\alpha  - 2}}{\partial _{zz}}{u_{3,\varepsilon }}} \right) + {\partial _z}{p_\varepsilon } = 0,\\
               &{\partial _t}{{\tilde b}_\varepsilon } + ({u_\varepsilon } \cdot \nabla ){{\tilde b}_\varepsilon } - ({b_\varepsilon } \cdot \nabla ){{\tilde u}_\varepsilon } - {\Delta _H}{{\tilde b}_\varepsilon } - {\varepsilon ^{\alpha  - 2}}{\partial _{zz}}{{\tilde b}_\varepsilon } = 0,\\
               &{\varepsilon ^2}\left( {{\partial _t}{b_{3,\varepsilon }} + {u_\varepsilon } \cdot \nabla {b_{3,\varepsilon }} - {b_\varepsilon } \cdot \nabla {u_{3,\varepsilon }} - {\Delta _H}{b_{3,\varepsilon }} - {\varepsilon ^{\alpha  - 2}}{\partial _{zz}}{b_{3,\varepsilon }}} \right) = 0,\\
               &\nabla \cdot {u_\varepsilon } = 0,\;\;\nabla  \cdot {b_\varepsilon } = 0,
  \end{aligned}
 \end{cases}
\end{equation}
in the fixed domain $\Omega : = M \times ( - 1,1)$, for any $t \in (0,\infty )$.

Next, in order to simplify our calculations in the following sections, we will recast the SHMHD equations in terms of new variables which we call ${A}$ and ${B}$, in such a way as to symmetrize system \eqref{2.2}. The Elsässer variables are defined as
\begin{align*}
            {{A}_\varepsilon } =:& ({{\tilde A}_\varepsilon },{A_{3,\varepsilon }}) = ({{\tilde u}_\varepsilon } + {{\tilde b}_\varepsilon },{u_{3,\varepsilon }} + {b_{3,\varepsilon }}),\\
            {{B}_\varepsilon } =:& ({{\tilde B}_\varepsilon },{B_{3,\varepsilon }}) = ({{\tilde u}_\varepsilon } - {{\tilde b}_\varepsilon },{u_{3,\varepsilon }} - {b_{3,\varepsilon }}),
\end{align*}
and therefore the new formulation for system \eqref{2.2} in the fixed domain $\Omega$ can be written as
\begin{equation} \label{2.3}
 \begin{cases}
  \begin{aligned}
              &{\partial _t}{{\tilde A}_\varepsilon } + ({B_\varepsilon } \cdot \nabla ){{\tilde A}_\varepsilon } + {\nabla _H}{p_\varepsilon } - {\Delta _H}{{\tilde A}_\varepsilon } - {\varepsilon ^{\alpha  - 2}}{\partial _{zz}}{{\tilde A}_\varepsilon } = 0,\\
              &{\varepsilon ^2}\left( {{\partial _t}{A_{3,\varepsilon }} + {B_\varepsilon } \cdot \nabla {A_{3,\varepsilon }} - {\Delta _H}{A_{3,\varepsilon }} - {\varepsilon ^{\alpha  - 2}}{\partial _{zz}}{A_{3,\varepsilon }}} \right) + {\partial _z}{p_\varepsilon } = 0,\\
              &{\partial _t}{{\tilde B}_\varepsilon } + ({A_\varepsilon } \cdot \nabla ){{\tilde B}_\varepsilon } + {\nabla _H}{p_\varepsilon } - {\Delta _H}{{\tilde B}_\varepsilon } - {\varepsilon ^{\alpha  - 2}}{\partial _{zz}}{{\tilde B}_\varepsilon } = 0,\\
              &{\varepsilon ^2}\left( {{\partial _t}{B_{3,\varepsilon }} + {A_\varepsilon } \cdot \nabla {B_{3,\varepsilon }} - {\Delta _H}{B_{3,\varepsilon }} - {\varepsilon ^{\alpha  - 2}}{\partial _{zz}}{B_{3,\varepsilon }}} \right) + {\partial _z}{p_\varepsilon } = 0,\\
              &{\nabla _H} \cdot {{\tilde A}_\varepsilon } + {\partial _z}{A_{3,\varepsilon }} = 0,\;\;{\nabla _H} \cdot {{\tilde B}_\varepsilon } + {\partial _z}{B_{3,\varepsilon }} = 0.
  \end{aligned}
 \end{cases}
\end{equation}
For system \eqref{2.3}, we consider the periodic initial-boundary value problem with the periodic boundary
\begin{equation} \label{2.4}
\begin{aligned}
                  &{{\tilde A}_\varepsilon },\;{A_{3,\varepsilon} },\;{{\tilde B}_\varepsilon },\;{B_{3,\varepsilon} }\; {\rm{and}} \;{p_{\varepsilon} }\;{\rm{are\;all\;periodic\;functions\;in}}\;x\;{\rm{with\;period}}\;{L_1},\\
                  &{{\tilde A}_\varepsilon },\;{A_{3,\varepsilon} },\;{{\tilde B}_\varepsilon },\;{B_{3,\varepsilon} }\; {\rm{and}} \;{p_{\varepsilon} }\;{\rm{are\;all\;periodic\;functions\;in}}\;y\;{\rm{with\;period}}\;{L_2},\\
                  &{{\tilde A}_\varepsilon },\;{A_{3,\varepsilon} },\;{{\tilde B}_\varepsilon },\;{B_{3,\varepsilon} }\; {\rm{and}} \;{p_{\varepsilon} }\;{\rm{are\;all\;periodic\;functions\;in}}\;z\;{\rm{with\;period}}\;2,
\end{aligned}
\end{equation}
and the initial condition
\begin{equation} \label{2.5}
              { {( {{{\tilde A}_\varepsilon },{A_{3,\varepsilon }}} )} \vert_{t = 0}} = ( {{\tilde A_{\varepsilon ,0}},{A_{3,\varepsilon ,0}}})\;{\rm{and}}\;{ {( {{{\tilde B}_\varepsilon },{B_{3,\varepsilon }}})} \vert_{t = 0}} = ( {{{\tilde B}_{\varepsilon ,0}},{B_{3,\varepsilon ,0}}}),
\end{equation}
and we also take into account the following symmetric condition
\begin{equation} \label{2.6}
\begin{aligned}
              &{{\tilde A}_\varepsilon },\;{A_{3,\varepsilon} }\;{\rm{and}}\;{p_\varepsilon }\;{\rm{are\;even,\;odd\;and\; even\;in}}\;z,\;{\rm{respectively}},\\
              &{{\tilde B}_\varepsilon },\;{B_{3,\varepsilon} }\;{\rm{are\;even\;and\;odd\;in}}\;z,\;{\rm{respectively}}.
\end{aligned}
\end{equation}
It should be noted that the dynamics of the SHMHD equations inherits the above symmetry condition, which means that as long as we impose the symmetric condition on the initial velocity and magnetic field, the solution $({{\tilde A}_\varepsilon },{A_{3,\varepsilon} },{{\tilde B}_\varepsilon },{B_{3,\varepsilon} },{p_{\varepsilon} })$ of system \eqref{2.3} will fulfill \eqref{2.6} automatically. Owing to this fact, throughout our paper, we always suppose that
\begin{equation} \label{2.7}
\begin{aligned}
        &{{\tilde A}_0},\;{A_{3,0}},{{\tilde B}_0}\;{\rm{and}}\;{B_{3,0}}\;{\rm{are}}\;{\rm{all}}\;{\rm{periodic}}\;{\rm{in}}\;x,\;y,\;z,\\
        &{{\tilde A}_0},\;{A_{3,0}}\;{\rm{are\;even\;and\;odd\;in}}\;z,\;{\rm{respectively}},\\
        &{{\tilde B}_0},\;{B_{3,0}}\;{\rm{are\;even\;and\;odd\;in}}\;z,\;{\rm{respectively}}.
\end{aligned}
\end{equation}

The weak solutions of the SHMHD equations are defined as follows.
\begin{definition}\label{def2.1}
(Leray-Hopf weak solution) Assume that $({A_0},{B_0}) = ({{\tilde A}_0},{A_{3,0}},{{\tilde B}_0},{B_{3,0}}) \in L_\sigma ^2(\Omega )$ with ${\nabla \cdot {A_0}} = 0$ and ${\nabla \cdot {B_0}} = 0$, satisfies the symmetric condition \eqref{2.7}. Then a quadruple of functions $({{\tilde A}_\varepsilon },{A_{3,\varepsilon} },{{\tilde B}_\varepsilon },{B_{3,\varepsilon} })$ will be called a Leray-Hopf weak solution to system \eqref{2.3}, subject to the periodic boundary and initial conditions \eqref{2.4}-\eqref{2.5} and symmetry condition \eqref{2.6}, if

(i) $(A_\varepsilon,B_\varepsilon) = ({{\tilde A}_\varepsilon },{A_{3,\varepsilon} },{{\tilde B}_\varepsilon },{B_{3,\varepsilon} })$ has the following regularity properties
\begin{equation*}
            ({{\tilde A}_\varepsilon },{A_{3,\varepsilon} },{{\tilde B}_\varepsilon },{B_{3,\varepsilon} }) \in {C_w}([0,\infty );L_\sigma ^2(\Omega )) \cap L_{loc}^2([0,\infty );{H^1}(\Omega )),
\end{equation*}
where ${C_w}$ means weakly continuous;

(ii) $({{\tilde A}_\varepsilon },{A_{3,\varepsilon} },{{\tilde B}_\varepsilon },{B_{3,\varepsilon} })$ is spatially periodic and satisfies the symmetric condition \eqref{2.6};

(iii) for any $\bm{\varphi} = ({\tilde{\varphi}},{\varphi _3}) \in C_c^\infty (\bar \Omega \times [0,\infty ))$ and $\bm{\phi} =({\tilde{\phi}},{\phi _3}) \in C_c^\infty (\bar \Omega \times [0,\infty ))$, the integral identity
\begin{equation*}
  \begin{aligned}
                &\int_0^\infty  {\int_\Omega  {\left[ { - ({{\tilde A}_\varepsilon } \cdot {\partial _t}\tilde \varphi  + {\varepsilon ^2}{A_{3,\varepsilon }}{\partial _t}{\varphi _3}) + ({B_\varepsilon } \cdot \nabla ){{\tilde A}_\varepsilon } \cdot \tilde \varphi  + {\varepsilon ^2}({B_\varepsilon } \cdot \nabla {A_{3,\varepsilon }}){\varphi _3}} \right.} } \\
                &+ \left. {\nabla {{\tilde A}_\varepsilon }:{\nabla _H}\tilde \varphi  + {\varepsilon ^{\alpha  - 2}}{\partial _z}{{\tilde A}_\varepsilon } \cdot {\partial _z}\tilde \varphi  + {\varepsilon ^2}{\nabla _H}{A_{3,\varepsilon }} \cdot {\nabla _H}{\varphi _3} + {\varepsilon ^\alpha }{\partial _z}{A_{3,\varepsilon }}{\partial _z}{\varphi _3}} \right]d\Omega dt\\
                &+ \int_0^\infty  {\int_\Omega  {\left[ { - ({{\tilde B}_\varepsilon } \cdot {\partial _t}\tilde \phi  + {\varepsilon ^2}{B_{3,\varepsilon }}{\partial _t}{\phi _3}) + ({A_\varepsilon } \cdot \nabla ){{\tilde B}_\varepsilon } \cdot \tilde \phi  + {\varepsilon ^2}({A_\varepsilon } \cdot \nabla {B_{3,\varepsilon }}){\phi _3}} \right.} } \\
                &+ \left. {\nabla {{\tilde B}_\varepsilon }:{\nabla _H}\tilde \phi  + {\varepsilon ^{\alpha  - 2}}{\partial _z}{{\tilde B}_\varepsilon } \cdot {\partial _z}\tilde \phi  + {\varepsilon ^2}{\nabla _H}{B_{3,\varepsilon }} \cdot {\nabla _H}{\phi _3} + {\varepsilon ^\alpha }{\partial _z}{B_{3,\varepsilon }}{\partial _z}{\phi _3}} \right]d\Omega dt\\
                =& \int_\Omega  {\left( {{{\tilde A}_0} \cdot \tilde \varphi ( \cdot ,0) + {\varepsilon ^2}{A_{3,0}}{\varphi _3}( \cdot ,0)} \right)d\Omega }  + \int_\Omega  {\left( {{{\tilde B}_0} \cdot \tilde \phi ( \cdot ,0) + {\varepsilon ^2}{B_{3,0}}{\phi _3}( \cdot ,0)} \right)d\Omega }
   \end{aligned}
\end{equation*}
holds, where the spatially periodic functions $\bm{\varphi}$ and $\bm{\phi}$ satisfy the incompressible conditions ${\nabla \cdot \bm{\varphi}} = 0$ and ${\nabla \cdot \bm{\phi}} = 0$, respectively, and the symmetric condition \eqref{2.6}, and ${\tilde{\varphi}} = ({\varphi _1},{\varphi _2})$ and ${\tilde{\phi}} = ({\phi _1},{\phi _2})$. Here we write $d\Omega = dxdydz$.
\end{definition}

\begin{remark}\label{rem2.2}
(i) Following the similar strategy of Duvaut and Lions \cite{1972-Duvaut-Lions-1} (see also \cite{1983-Sermange-Temam}), the proof of the existence of weak (strong) solutions to SHMHD equations \eqref{2.3} can be obtained. More precisely, when the initial data $({{\tilde A}_0},{A_{3,0}},{{\tilde B}_0},{B_{3,0}}) \in L_\sigma ^2(\Omega )$, with $\nabla \cdot {A_0} = 0$ and $\nabla \cdot {B_0} = 0$, we can prove that there exists a global-in-time weak solution $({{\tilde A}_\varepsilon },{A_{3,\varepsilon} },{{\tilde B}_\varepsilon },{B_{3,\varepsilon} })$ to system \eqref{2.3}, supplemented with the periodic boundary and initial conditions \eqref{2.4}-\eqref{2.5} and symmetric condition \eqref{2.6}. On the other hand, when the initial data $({{\tilde A}_0},{A_{3,0}},{{\tilde B}_0},{B_{3,0}}) \in H^1(\Omega )$, with $\nabla \cdot {A_0} = 0$ and $\nabla \cdot {B_0} = 0$, we can also prove that it possesses a unique local-in-time strong solution.

(ii) Applying some classical theories on the three-dimensional Navier-Stokes equations (see Ch.III, Remark 4.1 in \cite{1979-Temam}; Theorem 4.6 in \cite{2016-Robinson-Rodrigo-Sadowski}) to the three-dimensional MHD equations, we can prove that there exists a weak solution $({{\tilde A}_\varepsilon },{A_{3,\varepsilon} },{{\tilde B}_\varepsilon },{B_{3,\varepsilon} })$ of SHMHD equations \eqref{2.3} by Galerkin method such that the following energy inequality holds:
\begin{equation} \label{2.8}
\begin{aligned}
                &\Vert {{{\tilde A}_\varepsilon }(t)} \Vert_2^2 + \Vert {{{\tilde B}_\varepsilon }(t)} \Vert_2^2 + {\varepsilon ^2}\Vert {{A_{3,\varepsilon }}(t)} \Vert_2^2 + {\varepsilon ^2}\Vert {{B_{3,\varepsilon }}(t)} \Vert_2^2\\
                &+ 2\int_0^{t} {\left( {\Vert {{\nabla _H}{{\tilde A}_\varepsilon }} \Vert_2^2 + {\varepsilon ^{\alpha  - 2}}\Vert {{\partial _z}{{\tilde A}_\varepsilon }} \Vert_2^2 + \Vert {{\nabla _H}{{\tilde B}_\varepsilon }} \Vert_2^2 + {\varepsilon ^{\alpha  - 2}}\Vert {{\partial _z}{{\tilde B}_\varepsilon }} \Vert_2^2} \right)ds} \\
                &+ 2\int_0^{t} {\left( {{\varepsilon ^2}\Vert {{\nabla _H}{A_{3,\varepsilon }}} \Vert_2^2 + {\varepsilon ^\alpha }\Vert {{\partial _z}{A_{3,\varepsilon }}} \Vert_2^2 + {\varepsilon ^2}\Vert {{\nabla _H}{B_{3,\varepsilon }}} \Vert_2^2 + {\varepsilon ^\alpha }\Vert {{\partial _z}{B_{3,\varepsilon }}} \Vert_2^2} \right)ds}\\
                \le& \Vert {{{\tilde A}_0}} \Vert_2^2 + \Vert {{{\tilde B}_0}} \Vert_2^2 + {\varepsilon ^2}\Vert {{A_{3,0}}} \Vert + {\varepsilon ^2}\Vert {{B_{3,0}}} \Vert_2^2,
\end{aligned}
\end{equation}
for a.e. $t \in [0,\infty )$.
\end{remark}

\subsection{Primitive equations with only horizontal viscosity and magnetic diffusivity (PEHM)}
In this subsection, we discuss the full diffusivity case and partial diffusivity case, respectively. In the case $\alpha = 2$, by setting $\mu = k = 1$, $\nu = \sigma ={\varepsilon ^2}$, and taking the limit as $\varepsilon \to 0$ in system \eqref{2.2}, we arrive at the following primitive equations with magnetic field (PEM, for short) (more details, see \cite{2023-Du-Li}):
\begin{equation} \label{2.9}
 \begin{cases}
  \begin{aligned}
                &{\partial _t}\tilde u + (u \cdot \nabla )\tilde u - {\Delta _H}\tilde u - {\partial _{zz}}\tilde u + {\nabla _H}p = (b \cdot \nabla )\tilde b,\\
                &{\partial _z}p = 0,\\
                &{\partial _t}\tilde b + (u \cdot \nabla )\tilde b - {\Delta _H}\tilde b - {\partial _{zz}}\tilde b = (b \cdot \nabla )\tilde u,\\
                &{\nabla _H} \cdot \tilde u + {\partial _z}{u_3} = 0,\\
                &{\nabla _H} \cdot \tilde b + {\partial _z}{b_3} = 0.
  \end{aligned}
 \end{cases}
\end{equation}
In the case $2 < \alpha < \infty$, by setting $\mu = k = 1$, $\nu = \sigma = {\varepsilon ^\alpha }$, and taking the limit as $\varepsilon \to 0$ in system \eqref{2.3}, it is straightforward to derive the following primitive equations with only horizontal viscosity and magnetic diffusivity (PEHM, for short)
\begin{equation} \label{2.10}
 \begin{cases}
  \begin{aligned}
                &{\partial _t}\tilde A + (B \cdot \nabla )\tilde A - {\Delta _H}\tilde A + {\nabla _H}p = 0,\\
                &{\partial _z}p = 0,\\
                &{\partial _t}\tilde B + (A \cdot \nabla )\tilde B - {\Delta _H}\tilde B + {\nabla _H}p = 0,\\
                &{\nabla _H} \cdot \tilde A + {\partial _z}{A_3} = 0,\\
                &{\nabla _H} \cdot \tilde B + {\partial _z}{B_3} = 0.
  \end{aligned}
 \end{cases}
\end{equation}
Compared \eqref{2.9} with \eqref{2.10}, it is clear that they are both systems that the primitive equations are coupled to the diffusion equations for the magnetic field, which is the extension of the simplest form of the PEs. In addition, we observe that the dissipation of system \eqref{2.9} in all directions does not vanish  when the case $\alpha = 2$, while the dissipation of system \eqref{2.10} in the vertical directions vanishes when the case $2 < \alpha < \infty$.

As in SHMHD equations \eqref{2.3}, we also need to consider the periodic initial-boundary value problem of system \eqref{2.10}, so we should impose the same boundary and initial conditions \eqref{2.4}-\eqref{2.5} and symmetric condition \eqref{2.6} on system \eqref{2.10}. In the study of the global (local) well-posedness of system \eqref{2.10}, we see that it is not necessary to consider the initial conditions for vertical velocity ${A_{3}}$ and magnetic field ${B_{3}}$. Thus, for convenience, we say that system \eqref{2.10} satisfies the initial conditions \eqref{2.5}. We also note that, by \eqref{2.7}, one deduces
\begin{equation} \label{2.11}
          {A_{3,0}}(x,y,0) = {B_{3,0}}(x,y,0) = 0,
\end{equation}
and then ${A_{3,0}}$ and ${B_{3,0}}$ can be uniquely determined by the incompressible constraints as
\begin{equation} \label{2.12}
             {A_{3,0}}(x,y,z) = - \int_0^z {{\nabla _H} \cdot {{\tilde A}_0}(x,y,\xi )d\xi } ,\;{\rm{for}}\;\forall (x,y,z) \in \Omega ,
\end{equation}
and
\begin{equation} \label{2.13}
            {B_{3,0}}(x,y,z) = - \int_0^z {{\nabla _H} \cdot {{\tilde B}_0}(x,y,\xi )d\xi } ,\;{\rm{for}}\;\forall (x,y,z) \in \Omega,
\end{equation}
respectively. Analogously, ${A_{3}}$ and ${B_{3}}$ can be uniquely determined by the incompressible properties as
\begin{equation} \label{2.14}
               {A_3}(x,y,z,t) = - \int_0^z {{\nabla _H} \cdot \tilde A(x,y,\xi ,t)d\xi } ,\;{\rm{for}}\;\forall (x,y,z) \in \Omega,
\end{equation}
and
\begin{equation} \label{2.15}
                {B_3}(x,y,z,t) = - \int_0^z {{\nabla _H} \cdot \tilde B(x,y,\xi ,t)d\xi } ,\;{\rm{for}}\;\forall (x,y,z) \in \Omega,
\end{equation}
respectively. In order to obtain the solutions of system \eqref{2.10}, throughout our paper, we only need to solve the horizontal components ${\tilde{A}}$ and ${\tilde{B}}$, and then the vertical components $A_{3}$ and $B_{3}$ are uniquely determined by \eqref{2.14} and \eqref{2.15}, respectively.

\subsection{Main results}
In this subsection, our main results are concerning the local-in-time and global-in-time strong convergences from the SHMHD equations to the PEHM and the global well-posedness of the PEHM for the initial data $({{\tilde A}_0},{{\tilde B}_0}) \in {H^1}(\Omega )$ with additional regularity conditions, as the aspect ratio parameter $\varepsilon $ tends to zero, as stated in the following.

\begin{theorem}\label{thm2.3}
Let $({{\tilde A}_0},{{\tilde B}_0}) \in {H^1}(\Omega )$ be a pair of periodic function fulfilling
\begin{equation*}
                {\nabla _H} \cdot \left( {\int_{ - 1}^1 {{{\tilde A}_0}(x,y,z)dz} } \right) = 0\;and\;{\nabla _H} \cdot \left( {\int_{ - 1}^1 {{{\tilde B}_0}(x,y,z)dz} } \right) = 0
\end{equation*}
on $M$, and satisfy the symmetry condition \eqref{2.7}, and let ${A_{3,0}}$ and ${B_{3,0}}$ be uniquely determined by \eqref{2.12} and \eqref{2.13}, respectively. Assume that $({{\tilde A}_\varepsilon },{A_{3,\varepsilon }},{{\tilde B}_\varepsilon },{B_{3,\varepsilon }})$ is an arbitrary Leray-Hopf weak solution to system \eqref{2.3}, satisfying the energy inequality \eqref{2.8}, and that $(\tilde A,{A_3},\tilde B,{B_3})$ is the unique local-in-time strong solution to system \eqref{2.10}, supplemented with the same periodic boundary and initial conditions \eqref{2.4}-\eqref{2.5} and symmetry condition \eqref{2.6}. Let ${t_1^ * }$ be the existence time of $(\tilde A,{A_3},\tilde B,{B_3})$ and set
\begin{equation*}
                ({\tilde U_\varepsilon },{U_3},{\tilde V_\varepsilon },{V_3}) = ({\tilde A_\varepsilon } - \tilde A,{A_{3,\varepsilon }} - {A_3},{\tilde B_\varepsilon } - \tilde B,{B_{3,\varepsilon }} - {B_3}).
\end{equation*}
Then, the following estimate holds for any $\varepsilon \in (0,\infty )$, that is,
\begin{equation*}
\begin{aligned}
              \mathop {\sup }\limits_{0 \le t < t_1^ * } &\left( {\Vert {({{\tilde U}_\varepsilon },\varepsilon {U_3},{{\tilde V}_\varepsilon },\varepsilon {V_3})} \Vert_2^2} \right)(t)\\
              &+ \int_0^{t_1^ * } {\left( {\Vert {{\nabla _H}{{\tilde U}_\varepsilon }} \Vert_2^2 + {\varepsilon ^{\alpha  - 2}}\Vert {{\partial _z}{{\tilde U}_\varepsilon }} \Vert_2^2 + \Vert {{\nabla _H}{{\tilde V}_\varepsilon }} \Vert_2^2 + {\varepsilon ^{\alpha  - 2}}\Vert {{\partial _z}{{\tilde V}_\varepsilon }} \Vert_2^2} \right)dt} \\
              &+ \int_0^{t_1^ * } {\left( {{\varepsilon ^2}\Vert {{\nabla _H}{U_3}} \Vert_2^2 + {\varepsilon ^\alpha }\Vert {{\partial _z}{U_3}} \Vert_2^2 + {\varepsilon ^2}\Vert {{\nabla _H}{V_3}} \Vert_2^2 + {\varepsilon ^\alpha }\Vert {{\partial _z}{V_3}} \Vert_2^2} \right)dt} \\
              \le C{\varepsilon ^\gamma }&(t_1^ *  + 1){e^{C(t_1^ *  + 1)}}\left[ {1 + {\left( {\Vert {{{\tilde A}_0}} \Vert_2^2 + \Vert {{A_{3,0}}} \Vert_2^2 + \Vert {{{\tilde B}_0}} \Vert_2^2 + \Vert {{B_{3,0}}} \Vert_2^2} \right)^2}} \right],
\end{aligned}
\end{equation*}
where $\gamma = \min \{ 2,\alpha - 2\}$ with $\alpha \in (2,\infty )$, and $C$ is a positive constant independent of $\varepsilon$. Therefore, we have the following local-in-time strong convergences
\begin{equation*}
    \begin{aligned}
              &({{\tilde A}_\varepsilon },\varepsilon {A_{3,\varepsilon }},{{\tilde B}_\varepsilon },\varepsilon {B_{3,\varepsilon }}) \to (\tilde A,0,\tilde B,0)\;\;{\rm{in}}\;{L^\infty }([0,t_1^*);{L^2}(\Omega )),\\
              &({\nabla _H}{{\tilde A}_\varepsilon },{\varepsilon ^{(\alpha  - 2)/2}}{\partial _z}{{\tilde A}_\varepsilon },\varepsilon {\nabla _H}{A_{3,\varepsilon }},{\varepsilon ^{\alpha /2}}{\partial _z}{A_{3,\varepsilon }},{A_{3,\varepsilon }}) \to ({\nabla _H}\tilde A,0,0,0,{A_3})\;\;{\rm{in}}\;{L^2}([0,t_1^*);{L^2}(\Omega )),\\
              &({\nabla _H}{{\tilde B}_\varepsilon },{\varepsilon ^{(\alpha  - 2)/2}}{\partial _z}{{\tilde B}_\varepsilon },\varepsilon {\nabla _H}{B_{3,\varepsilon }},{\varepsilon ^{\alpha /2}}{\partial _z}{B_{3,\varepsilon }},{B_{3,\varepsilon }}) \to ({\nabla _H}\tilde B,0,0,0,{B_3})\;\;{\rm{in}}\;{L^2}([0,t_1^*);{L^2}(\Omega )),
    \end{aligned}
\end{equation*}
and the rate of the convergence is of the order $O({\varepsilon ^{\gamma /2}})$.
\end{theorem}

\begin{theorem}\label{thm2.4}
Suppose that a periodic pair $({\tilde A_0},{\tilde B_0}) \in {H^1}(\Omega )$ with
\begin{equation*}
                {\nabla _H} \cdot \left( {\int_{ - 1}^1 {{{\tilde A}_0}(x,y,z)dz} } \right) = 0\;and\;{\nabla _H} \cdot \left( {\int_{ - 1}^1 {{{\tilde B}_0}(x,y,z)dz} } \right) = 0,
\end{equation*}
and that $({\partial _z}{\tilde A_0},{\partial _z}{\tilde B_0}) \in {L^p}(\Omega )$ for some $p \in (2,\infty )$. Then there exists a unique global-in-time strong solution $({\tilde A},{\tilde B})$ to system \eqref{2.10}, supplemented with the periodic boundary and initial conditions \eqref{2.4}-\eqref{2.5} and symmetry condition \eqref{2.6}. Moreover, we have the following estimate
\begin{equation*}
               \mathop {\sup }\limits_{0 \le s \le t} \left( {\Vert {(\tilde A,\tilde B)} \Vert_{{H^1}(\Omega )}^2} \right)(s) + \int_0^t {\left( {\Vert {{\nabla _H}\tilde A} \Vert_{{H^1}(\Omega )}^2 + \Vert {{\nabla _H}\tilde B} \Vert_{{H^1}(\Omega )}^2 + \Vert {({\partial _t}\tilde A,{\partial _t}\tilde B)} \Vert_2^2} \right)}  \le {N_2}(t),
\end{equation*}
for any $t\in [0,\infty)$, where ${N_2}(t)$ is a non-negative continuous increasing function defined on $[0,\infty)$.
\end{theorem}

\begin{remark}\label{rem2.5}
Through the similar argument as in \cite{2017-Cao-Li-Titi-Global}, we can obtain a global-in-time strong solution to system \eqref{2.10} under two additional regularity conditions $({\partial _z}{{\tilde A}_0},{\partial _z}{{\tilde B}_0}) \in {L^p}(\Omega )$ with $p \in (2,\infty )$ and $({{\tilde A}_0},{{\tilde B}_0}) \in {L^\infty }(\Omega )$, while in Theorem \ref{thm2.4} we also obtain the global well-posed result to system $(2.10)$ under only one regularity condition $({\partial _z}{{\tilde A}_0},{\partial _z}{{\tilde B}_0}) \in {L^p}(\Omega )$ with $p \in (2,\infty )$, which will slightly improve the global well-posed result in \cite{2017-Cao-Li-Titi-Global}.
\end{remark}

On account of the global well-posedness of system \eqref{2.10} in Theorem 2.4, we give the corresponding global-in-time strong convergence conclusion as follows.
\begin{theorem}\label{thm2.6}
Given a pair of periodic function $({{\tilde A}_0},{{\tilde B}_0}) \in {H^1}(\Omega )$ such that for any $(x,y) \in M$,
\begin{equation*}
                {\nabla _H} \cdot \left( {\int_{ - 1}^1 {{{\tilde A}_0}(x,y,z)dz} } \right) = 0\;and\;{\nabla _H} \cdot \left( {\int_{ - 1}^1 {{{\tilde B}_0}(x,y,z)dz} } \right) = 0,
\end{equation*}
and $({\partial _z}{\tilde A_0},{\partial _z}{\tilde B_0}) \in {L^p}(\Omega )$, for some $p \in (2,\infty )$. Assume that $({{\tilde A}_\varepsilon },{A_{3,\varepsilon }},{{\tilde B}_\varepsilon },{B_{3,\varepsilon }})$ is an arbitrary Leray-Hopf weak solution to system \eqref{2.3}, satisfying the energy inequality \eqref{2.8}, and that $(\tilde A,{A_3},\tilde B,{B_3})$ is the unique global-in-time strong solution to system \eqref{2.10}, supplemented with the same periodic boundary and initial conditions \eqref{2.4}-\eqref{2.5} and symmetric condition \eqref{2.6}. Set
\begin{equation*}
                ({\tilde U_\varepsilon },{U_3},{\tilde V_\varepsilon },{V_3}) = ({\tilde A_\varepsilon } - \tilde A,{A_{3,\varepsilon }} - {A_3},{\tilde B_\varepsilon } - \tilde B,{B_{3,\varepsilon }} - {B_3}).
\end{equation*}
Then, the following estimate holds for any $\mathcal{T}\in(0,\infty)$, that is,
\begin{equation*}
\begin{aligned}
              \mathop {\sup }\limits_{0 \le t \le \mathcal{T} } &\left( {\Vert {({{\tilde U}_\varepsilon },\varepsilon {U_3},{{\tilde V}_\varepsilon },\varepsilon {V_3})} \Vert_2^2} \right)(t)\\
              &+ \int_0^{\mathcal{T}} {\left( {\Vert {{\nabla _H}{{\tilde U}_\varepsilon }} \Vert_2^2 + {\varepsilon ^{\alpha  - 2}}\Vert {{\partial _z}{{\tilde U}_\varepsilon }} \Vert_2^2 + \Vert {{\nabla _H}{{\tilde V}_\varepsilon }} \Vert_2^2 + {\varepsilon ^{\alpha  - 2}}\Vert {{\partial _z}{{\tilde V}_\varepsilon }} \Vert_2^2} \right)dt} \\
              &+ \int_0^{\mathcal{T}} {\left( {{\varepsilon ^2}\Vert {{\nabla _H}{U_3}} \Vert_2^2 + {\varepsilon ^\alpha }\Vert {{\partial _z}{U_3}} \Vert_2^2 + {\varepsilon ^2}\Vert {{\nabla _H}{V_3}} \Vert_2^2 + {\varepsilon ^\alpha }\Vert {{\partial _z}{V_3}} \Vert_2^2} \right)dt} \\
              \le {\varepsilon ^\gamma }&{N_3(\mathcal{T})},
\end{aligned}
\end{equation*}
for a non-negative continuous increasing function $N_3(t)$ defined on $[0,\infty)$, depending only on $\mathcal{T}$, $L_1$, $L_2$, ${\Vert {({{\tilde A}_0},{{\tilde B}_0})} \Vert_{{H^1}}}$ and ${\Vert {({\partial _z}{{\tilde A}_0},{\partial _z}{{\tilde B}_0})} \Vert_p}$, but independent of $\varepsilon$, where $\gamma = \min \{ 2,\alpha - 2\}$ with $\alpha \in (2,\infty )$. As a consequence, the local-in-time strong convergences in Theorem \ref{thm2.3} can be extended to the global-in-time strong convergences.
\end{theorem}

\begin{theorem}\label{thm2.7}
Given a pair of periodic function $({{\tilde A}_0},{{\tilde B}_0}) \in {H^2}(\Omega )$ such that for any $(x,y) \in M$,
\begin{equation*}
                {\nabla _H} \cdot \left( {\int_{ - 1}^1 {{{\tilde A}_0}(x,y,z)dz} } \right) = 0\;and\;{\nabla _H} \cdot \left( {\int_{ - 1}^1 {{{\tilde B}_0}(x,y,z)dz} } \right) = 0,
\end{equation*}
Assume that $({{\tilde A}_\varepsilon },{A_{3,\varepsilon }},{{\tilde B}_\varepsilon },{B_{3,\varepsilon }})$ is the unique local-in-time strong solution to system \eqref{2.3}, and that $({\tilde A},{A_{3}},{\tilde B},{B_{3}})$ is the unique global-in-time strong solution to system \eqref{2.10}, supplemented with the same periodic boundary and initial conditions \eqref{2.4}-\eqref{2.5}and symmetric condition \eqref{2.6}. Set
\begin{equation*}
                ({\tilde U_\varepsilon },{U_3},{\tilde V_\varepsilon },{V_3}) = ({\tilde A_\varepsilon } - \tilde A,{A_{3,\varepsilon }} - {A_3},{\tilde B_\varepsilon } - \tilde B,{B_{3,\varepsilon }} - {B_3}).
\end{equation*}
Then for any finite time $\mathcal{T}\in(0,\infty)$, we can choose a small positive number $\varepsilon (\mathcal{T}) = {\left( {\frac{{5\delta_0^2}}{{8{N_6}(\mathcal{T})}}} \right)^{1/\gamma }}$ such that there exists a unique strong-in-time solution $({{\tilde A}_\varepsilon },{A_{3,\varepsilon }},{{\tilde B}_\varepsilon },{B_{3,\varepsilon }})$ of system \eqref{2.3} on $[0,\mathcal{T}]$, and such that the following estimate holds for system \eqref{4.2} (more details, see Section 4), that is,
\begin{equation*}
 \begin{aligned}
             \mathop {\sup }\limits_{0 \le t \le \mathcal{T}} &\left( {\Vert {({{\tilde U}_\varepsilon },\varepsilon {U_3},{{\tilde V}_\varepsilon },\varepsilon {V_3})} \Vert_{{H^1}(\Omega )}^2} \right)(t)\\
             &+ \int_0^\mathcal{T} {\left( {\Vert {{\nabla _H}{{\tilde U}_\varepsilon }} \Vert_{{H^1}(\Omega )}^2 + {\varepsilon ^{\alpha  - 2}}\Vert {{\partial _z}{{\tilde U}_\varepsilon }} \Vert_{{H^1}(\Omega )}^2 + \Vert {{\nabla _H}{{\tilde V}_\varepsilon }} \Vert_{{H^1}(\Omega )}^2 + {\varepsilon ^{\alpha  - 2}}\Vert {{\partial _z}{{\tilde V}_\varepsilon }} \Vert_{{H^1}(\Omega )}^2} \right)dt} \\
             &+ \int_0^\mathcal{T} {\left( {{\varepsilon ^2}\Vert {{\nabla _H}{U_3}} \Vert_{{H^1}(\Omega )}^2 + {\varepsilon ^\alpha }\Vert {{\nabla _H}{U_3}} \Vert_{{H^1}(\Omega )}^2 + {\varepsilon ^2}\Vert {{\nabla _H}{V_3}} \Vert_{{H^1}(\Omega )}^2 + {\varepsilon ^\alpha }\Vert {{\nabla _H}{V_3}} \Vert_{{H^1}(\Omega )}^2} \right)dt} \\
             \le {\varepsilon ^\gamma }&{N_7}(\mathcal{T}),
\end{aligned}
\end{equation*}
provided that $\varepsilon \in (0,\varepsilon (\mathcal{T}))$, where $\gamma = \min \{ 2,\alpha - 2\}$ with $\alpha \in (2,\infty )$ and $N_7(t)$, a non-negative continuous increasing function defined on $[0,\infty)$, is independent of $\varepsilon$. Thus, we have the following global-in-time strong convergences
\begin{equation*}
    \begin{aligned}
                   &({{\tilde A}_\varepsilon },\varepsilon {A_{3,\varepsilon }},{{\tilde B}_\varepsilon },\varepsilon {B_{3,\varepsilon }}) \to (\tilde A,0,\tilde B,0)\;\;{\rm{in}}\;{L^\infty }([0,{\cal T}];{H^1}(\Omega )),\\
                   &({\nabla _H}{{\tilde A}_\varepsilon },{\varepsilon ^{(\alpha  - 2)/2}}{\partial _z}{{\tilde A}_{\varepsilon }},\varepsilon {\nabla _H}{A_{3,\varepsilon }},{\varepsilon ^{\alpha /2}}{\partial _z}{{\tilde A}_\varepsilon },{A_{3,\varepsilon }}) \to ({\nabla _H}\tilde A,0,0,0,{A_3})\;\;{\rm{in}}\;{L^2}([0,{\cal T}];{H^1}(\Omega )),\\
                   &({\nabla _H}{{\tilde B}_\varepsilon },{\varepsilon ^{(\alpha  - 2)/2}}{\partial _z}{{\tilde B}_\varepsilon },\varepsilon {\nabla _H}{B_{3,\varepsilon }},{\varepsilon ^{\alpha /2}}{\partial _z}{{\tilde B}_\varepsilon },{B_{3,\varepsilon }}) \to ({\nabla _H}\tilde B,0,0,0,{B_3})\;\;{\rm{in}}\;{L^2}([0,{\cal T}];{H^1}(\Omega )),\\
                   &({A_{3,\varepsilon }},{B_{3,\varepsilon }}) \to ({A_3},{B_3})\;\;{\rm{in}}\;{L^\infty }([0,{\cal T}];{L^2}(\Omega )),
    \end{aligned}
\end{equation*}
and the rate of the convergence is of the order $O({\varepsilon ^{\gamma /2}})$.
\end{theorem}

\begin{remark}\label{rem2.8}
As stated in Theorem \ref{thm2.7}, when the initial pair $({{\tilde A}_0},{{\tilde B}_0})$ belongs to ${H^2}(\Omega )$, system \eqref{2.10} admits a unique global-in-time strong solution $({\tilde A},{\tilde B})$, which is proved by the similar method as in \cite{2016-Cao-Li-Titi-Global}. From \eqref{2.12} and \eqref{2.13}, we deduce that $({{\tilde A}_0},{A_{3,0}},{{\tilde B}_0},{B_{3,0}}) \in {H^1}(\Omega )$. In addition, we know from Remark \ref{rem2.2} that system \eqref{2.3}, supplemented with \eqref{2.4}-\eqref{2.6}, has a unique local-in-time strong solution $({{\tilde A}_\varepsilon },{A_{3,\varepsilon }},{{\tilde B}_\varepsilon },{B_{3,\varepsilon }})$.
\end{remark}

\section{Strong convergence for the $H^1$-initial data}
\subsection{The $H^1$-initial data without additional regularity}
Our aim of this subsection is to establish the proof of Theorem \ref{thm2.3}. More specifically, suppose that the initial data $({{\tilde A}_0},{{\tilde B}_0}) \in {H^1}(\Omega )$ with
\begin{equation*}
                 {\nabla _H} \cdot \left( {\int_{ - 1}^1 {{{\tilde A}_0}(x,y,z )dz} } \right) = 0\;{\rm{and}}\;{\nabla _H} \cdot \left( {\int_{ - 1}^1 {{{\tilde B}_0}(x,y,z )dz} } \right) = 0
\end{equation*}
for any $(x,y)\in M$, we prove that SHMHD equations \eqref{2.3} strongly converge to PEHM \eqref{2.10} as the aspect ratio parameter $\varepsilon$ tends zero, in which the convergences are local-in-time. Note that for the initial data $({{\tilde A}_0},{{\tilde B}_0}) \in {H^1}(\Omega )$, system \eqref{2.3} subject to \eqref{2.4}-\eqref{2.6} has a global-in-time weak solution, while system \eqref{2.10} corresponding to \eqref{2.4}-\eqref{2.6} possesses a unique local-in-time strong solution.

The next Lemma will be used in the proof of strong convergences.
\begin{lemma}\label{lem3.1}
(see, Lemma 2.1 in \cite{2017-Cao-Li-Titi-Global}) The following tri-linear estimates hold true:
\begin{equation*}
    \begin{aligned}
                  &\int_M {\left( {\int_{ - 1}^1 {\vert {f(x,y,z)} \vert dz} } \right)\left( {\int_{ - 1}^1 {\vert {g(x,y,z)h(x,y,z)} \vert dz} } \right)dxdy} \\
                  \le& \min \left\{ {C\Vert f \Vert_2^{\frac{1}{2}}\left( {\Vert f \Vert_2^{\frac{1}{2}} + \Vert {{\nabla _H}f} \Vert_2^{\frac{1}{2}}} \right){{\Vert h \Vert}_2}\Vert g \Vert_2^{\frac{1}{2}}\left( {\Vert g \Vert_2^{\frac{1}{2}} + \Vert {{\nabla _H}g} \Vert_2^{\frac{1}{2}}} \right),} \right.\\
                  &\left. {C{{\Vert f \Vert}_2}\Vert g \Vert_2^{\frac{1}{2}}\left( {\Vert g \Vert_2^{\frac{1}{2}} + \Vert {{\nabla _H}g} \Vert_2^{\frac{1}{2}}} \right)\Vert h \Vert_2^{\frac{1}{2}}\left( {\Vert h \Vert_2^{\frac{1}{2}} + \Vert {{\nabla _H}h} \Vert_2^{\frac{1}{2}}} \right)} \right\},
    \end{aligned}
\end{equation*}
for every $f$, $g$, $h$ such that the right-hand sides make sense and are finite, where $C$ is a positive constant.
\end{lemma}

Adopting the similar method as in \cite{2017-Cao-Li-Titi-Global}, for the initial data $({{\tilde A}_0},{{\tilde B}_0}) \in {H^1}(\Omega )$, we obtain the local well-posedness of PEHM \eqref{2.10}. The local well-posedness of PEHM \eqref{2.10} is stated as follows.

\begin{proposition}\label{prop3.2}
Assume that an initial pair $({{\tilde A}_0},{{\tilde B}_0}) \in {H^1}(\Omega )$ is two periodic functions with
\begin{equation*}
                  {\nabla _H} \cdot \left( {\int_{ - 1}^1 {{{\tilde A}_0}(x,y,z)dz} } \right) = 0\;{\rm{and}}\;{\nabla _H} \cdot \left( {\int_{ - 1}^1 {{{\tilde B}_0}(x,y,z)dz} } \right) = 0,
\end{equation*}
for any $(x,y)\in M$ and fulfils the symmetry condition \eqref{2.7}. Then the following conclusions are true:

(i) System \eqref{2.10}, corresponding to the periodic boundary and initial conditions \eqref{2.4}-\eqref{2.5} and symmetric condition \eqref{2.6}, admits a unique local-in-time strong solution $({\tilde A},{\tilde B})$ enjoying the following regularity properties
\begin{equation*}
              \begin{aligned}
              &(\tilde A,\tilde B) \in {L^\infty }([0,t_1^ * );{H^1}(\Omega )) \cap C([0,t_1^ * );{L^2}(\Omega )),\\
              &({\nabla _H}\tilde A,{\nabla _H}\tilde B) \in {L^2}([0,t_1^ * );{H^1}(\Omega )),\;({\partial _t}\tilde A,{\partial _t}\tilde B) \in {L^2}([0,t_1^ * );{L^2}(\Omega )),
              \end{aligned}
\end{equation*}
where $t_1^ *$ denotes the maximum local existence time of this strong solution;

(ii) The local-in-time strong solution $({\tilde A},{\tilde B})$ of system \eqref{2.10} satisfies the following estimate
\begin{equation} \label{3.1}
               \mathop {\sup }\limits_{0 \le s \le t} \left( {\Vert {(\tilde A,\tilde B)} \Vert_{{H^1}(\Omega )}^2} \right)(s) + \int_0^t {\left( {\Vert {{\nabla _H}\tilde A} \Vert_{{H^1}(\Omega )}^2 + \Vert {{\nabla _H}\tilde B} \Vert_{{H^1}(\Omega )}^2 + \Vert {({\partial _t}\tilde A,{\partial _t}\tilde B)} \Vert_2^2} \right)ds}  \le C,
\end{equation}
for any $t \in [0,t_1^ * )$, where $C$ is a positive constant.
\end{proposition}

The following proposition is critical for proving strong convergences of SHMHD equations \eqref{2.3} to PEHM \eqref{2.10}, which is formally obtained by testing SHMHD equations \eqref{2.3} with $({\tilde A },{A_{3}},{\tilde B },{B_{3}})$. More information on the rigorous justification for this proposition may be found in the references \cite{2019-Li-Titi} and \cite{2013-Bardos}.

\begin{proposition}\label{prop3.3}
Given a pair of periodic function $({{\tilde A}_0},{{\tilde B}_0}) \in {H^1}(\Omega )$, such that
\begin{equation*}
                 {\nabla _H} \cdot \left( {\int_{ - 1}^1 {{{\tilde A}_0}(x,y,z)dz} } \right) = 0,\;{A_{3,0}}(x,y,z) =  - {\nabla _H} \cdot \left( {\int_0^z {{{\tilde A}_0}(x,y,\xi )d\xi } } \right),
\end{equation*}
and
\begin{equation*}
                 {\nabla _H} \cdot \left( {\int_{ - 1}^1 {{{\tilde B}_0}(x,y,z)dz} } \right) = 0,\;{B_{3,0}}(x,y,z) =  - {\nabla _H} \cdot \left( {\int_0^z {{{\tilde B}_0}(x,y,\xi )d\xi } } \right).
\end{equation*}
Assume that $({{\tilde A}_\varepsilon },{A_{3,\varepsilon }},{{\tilde B}_\varepsilon },{B_{3,\varepsilon }})$ is an arbitrary Leray-Hopf weak solution of system \eqref{2.3}, satisfying the energy inequality \eqref{2.8}, and that $({\tilde A},{A_{3}},{\tilde B},{B_{3}})$ is the unique local-in-time strong solution of system \eqref{2.10}. Then, the following integral equality holds, that is,
\begin{equation} \label{3.2}
        \begin{aligned}
                 &\left[ {\int_\Omega  {({{\tilde A}_\varepsilon } \cdot \tilde A + {\varepsilon ^2}{A_{3,\varepsilon }}{A_3} + {{\tilde B}_\varepsilon } \cdot \tilde B + {\varepsilon ^2}{B_{3,\varepsilon }}{B_3})d\Omega} } \right]({r_0})\\
                 &+ \int_0^{r_0} {\int_\Omega  {({\nabla _H}{{\tilde A}_\varepsilon }:{\nabla _H}\tilde A + {\varepsilon ^{\alpha  - 2}}{\partial _z}{{\tilde A}_\varepsilon } \cdot {\partial _z}\tilde A + {\varepsilon ^2}{\nabla _H}{A_{3,\varepsilon }} \cdot {\nabla _H}{A_3} + {\varepsilon ^\alpha }{\partial _z}{A_{3,\varepsilon }} {\partial _z}{A_3} )d\Omega} dt} \\
                 &+ \int_0^{r_0} {\int_\Omega  {({\nabla _H}{{\tilde B}_\varepsilon }:{\nabla _H}\tilde B + {\varepsilon ^{\alpha  - 2}}{\partial _z}{{\tilde B}_\varepsilon } \cdot {\partial _z}\tilde B + {\varepsilon ^2}{\nabla _H}{B_{3,\varepsilon }} \cdot {\nabla _H}{B_3} + {\varepsilon ^\alpha }{\partial _z}{B_{3,\varepsilon }} {\partial _z}{B_3})d\Omega} dt} \\
                 =& \Vert {{{\tilde A}_0}} \Vert_2^2 + \frac{{{\varepsilon ^2}}}{2}\Vert {{A_{3,0}}} \Vert_2^2 + \frac{{{\varepsilon ^2}}}{2}\Vert {{A_3}({r_0})} \Vert_2^2 + {\varepsilon ^2}\int_0^{r_0} {\int_\Omega  {\left( {\int_0^z {{\partial _t}\tilde A(x,y,\xi ,t)d\xi } } \right) \cdot {\nabla _H}{U_{3,\varepsilon }}d\Omega} dt} \\
                 &+ \Vert {{{\tilde B}_0}} \Vert_2^2 + \frac{{{\varepsilon ^2}}}{2}\Vert {{B_{3,0}}} \Vert_2^2 + \frac{{{\varepsilon ^2}}}{2}\Vert {{B_3}({r_0})} \Vert_2^2 + {\varepsilon ^2}\int_0^{r_0} {\int_\Omega  {\left( {\int_0^z {{\partial _t}\tilde B(x,y,\xi ,t)d\xi } } \right) \cdot {\nabla _H}{V_{3,\varepsilon }}d\Omega} dt} \\
                 &+ \int_0^{r_0} {\int_\Omega  {\left[ { - ({B_\varepsilon } \cdot \nabla ){{\tilde A}_\varepsilon } \cdot \tilde A - {\varepsilon ^2}({B_\varepsilon } \cdot \nabla {A_{3,\varepsilon }}){A_3}- ({A_\varepsilon } \cdot \nabla ){{\tilde B}_\varepsilon } \cdot \tilde B - {\varepsilon ^2}({A_\varepsilon } \cdot \nabla {B_{3,\varepsilon }}){B_3}} \right]d\Omega} dt} \\
                 &+ \int_0^{r_0} {\int_\Omega  {({{\tilde A}_\varepsilon } \cdot {\partial _t}\tilde A + {{\tilde B}_\varepsilon } \cdot {\partial _t}\tilde B)d\Omega} dt},
        \end{aligned}
\end{equation}
for any ${r_0} \in [0,t_1^ * )$.
\end{proposition}

\begin{proposition}\label{prop3.4}
Denote $({{\tilde U}_\varepsilon },{U_{3,\varepsilon }},{{\tilde V}_\varepsilon },{V_{3,\varepsilon }}) = ({{\tilde A}_\varepsilon } - \tilde A,{A_{3,\varepsilon }} - {A_3},{{\tilde B}_\varepsilon } - \tilde B,{B_{3,\varepsilon }} - {B_3})$. Suppose that all assumptions in Proposition \ref{prop3.3} hold. Then, for any ${t} \in [0,t_1^ * )$ the following estimate
\begin{equation} \label{3.3}
\begin{aligned}
                \mathop {\sup }\limits_{0 \le s \le t}& \left( {\Vert {({{\tilde U}_\varepsilon },\varepsilon {U_{3,\varepsilon }},{{\tilde V}_\varepsilon },\varepsilon {V_{3,\varepsilon }})} \Vert_2^2} \right)(s)\\
                &+ \int_0^t {\left( {\Vert {{\nabla _H}{{\tilde U}_\varepsilon }} \Vert_2^2 + {\varepsilon ^{\alpha  - 2}}\Vert {{\partial _z}{{\tilde U}_\varepsilon }} \Vert_2^2 + \Vert {{\nabla _H}{{\tilde V}_\varepsilon }} \Vert_2^2 + {\varepsilon ^{\alpha  - 2}}\Vert {{\partial _z}{{\tilde V}_\varepsilon }} \Vert_2^2} \right)ds} \\
                &+ \int_0^t {\left( {{\varepsilon ^2}\Vert {{\nabla _H}{U_{3,\varepsilon }}} \Vert_2^2 + {\varepsilon ^\alpha }\Vert {{\partial _z}{U_{3,\varepsilon }}} \Vert_2^2 + {\varepsilon ^2}\Vert {{\nabla _H}{V_{3,\varepsilon }}} \Vert_2^2 + {\varepsilon ^\alpha }\Vert {{\partial _z}{V_{3,\varepsilon }}} \Vert_2^2} \right)ds} \\
                \le C&(t + 1){e^{C(t + 1)}}\left[ {{\varepsilon ^2} + {\varepsilon ^{\alpha  - 2}} + {\varepsilon ^2}{\left( {\Vert {{{\tilde A}_0}} \Vert_2^2 + {\varepsilon ^2}\Vert {{A_{3,0}}} \Vert_2^2 + \Vert {{{\tilde B}_0}} \Vert_2^2 + {\varepsilon ^2}\Vert {{B_{3,0}}} \Vert_2^2} \right)^2}} \right]
    \end{aligned}
\end{equation}
holds, which $C$ is a positive constant independent of $\varepsilon$.
\end{proposition}

\begin{proof}
Multiplying $(2.10)_1$ and $(2.10)_3$ by ${{\tilde A}_\varepsilon }$ and ${{\tilde B}_\varepsilon }$, respectively, and integrating the resulted equations over $\Omega \times (0,{r_0})$, after integration by parts we can deduce that
\begin{equation} \label{3.4}
\begin{aligned}
               &\int_0^{r_0} {\int_\Omega  {({{\tilde A}_\varepsilon } \cdot {\partial _t}\tilde A + {{\tilde B}_\varepsilon } \cdot {\partial _t}\tilde B + {\nabla _H}{{\tilde A}_\varepsilon }:{\nabla _H}\tilde A + {\nabla _H}{{\tilde B}_\varepsilon }:{\nabla _H}\tilde B)d\Omega} dt} \\
               =& \int_0^{r_0} {\int_\Omega  {\left[ { - (B \cdot \nabla )\tilde A \cdot {{\tilde A}_\varepsilon } - (A \cdot \nabla )\tilde B \cdot {{\tilde B}_\varepsilon }} \right]d\Omega} dt}.
\end{aligned}
\end{equation}
By replacing $({{\tilde A}_\varepsilon },{A_{3,\varepsilon }},{{\tilde B}_\varepsilon },{B_{3,\varepsilon }})$ with $(\tilde A,{A_3},\tilde B,{B_3})$, we infer from \eqref{3.4} and integration by parts that
\begin{equation} \label{3.5}
                \frac{1}{2}\left( {\Vert {\tilde A({r_0})} \Vert_2^2 + \Vert {\tilde B({r_0})} \Vert_2^2} \right) + \int_0^{{r_0}} {\left( {\Vert {{\nabla _H}\tilde A} \Vert_2^2 + \Vert {{\nabla _H}\tilde B} \Vert_2^2} \right)dt} = \frac{1}{2}\left( {\Vert {{{\tilde A}_0}} \Vert_2^2 + \Vert {{{\tilde B}_0}} \Vert_2^2} \right).
\end{equation}
From Remark \ref{rem2.2}, we see that for any ${r_0} \in [0,t_1^ * )$, the global weak solution $({{\tilde A}_\varepsilon },{A_{3,\varepsilon} },{{\tilde B}_\varepsilon },{B_{3,\varepsilon} })$ to system \eqref{2.3} satisfies the following energy inequality
\begin{equation} \label{3.6}
\begin{aligned}
                \Vert {{{\tilde A}_\varepsilon }({r_0})} \Vert_2^2 &+ \Vert {{{\tilde B}_\varepsilon }({r_0})} \Vert_2^2 + {\varepsilon ^2}\Vert {{A_{3,\varepsilon }}({r_0})} \Vert_2^2 + {\varepsilon ^2}\Vert {{B_{3,\varepsilon }}({r_0})} \Vert_2^2\\
                &+ 2\int_0^{{r_0}} {\left( {\Vert {{\nabla _H}{\tilde A}_\varepsilon} \Vert_2^2 + {\varepsilon ^{\alpha  - 2}}\Vert {{\partial _z}{\tilde A}_\varepsilon} \Vert_2^2 + \Vert {{\nabla _H}{\tilde B}_\varepsilon} \Vert_2^2 + {\varepsilon ^{\alpha  - 2}}\Vert {{\partial _z}{\tilde B}_\varepsilon} \Vert_2^2} \right)dt} \\
                &+ 2\int_0^{{r_0}} {\left( {{\varepsilon ^2}\Vert {{\nabla _H}{A_{3,\varepsilon }}} \Vert_2^2 + {\varepsilon ^\alpha }\Vert {{\partial _z}{A_{3,\varepsilon }}} \Vert_2^2 + {\varepsilon ^2}\Vert {{\nabla _H}{B_{3,\varepsilon }}} \Vert_2^2 + {\varepsilon ^\alpha }\Vert {{\partial _z}{B_{3,\varepsilon }}} \Vert_2^2} \right)dt}\\
                \le \Vert {{{\tilde A}_0}} \Vert_2^2 &+ \Vert {{{\tilde B}_0}} \Vert_2^2 + {\varepsilon ^2}\Vert {{A_{3,0}}} \Vert + {\varepsilon ^2}\Vert {{B_{3,0}}} \Vert_2^2.
\end{aligned}
\end{equation}
We start by adding \eqref{3.2} and \eqref{3.4}, then subtract the sum of \eqref{3.5} and \eqref{3.6} to obtain
\begin{equation} \label{3.7}
    \begin{aligned}
                 \frac{1}{2}&\left( {\Vert {{{\tilde U}_\varepsilon }({r_0})} \Vert_2^2 + {\varepsilon ^2}\Vert {{U_{3,\varepsilon }}({r_0})} \Vert_2^2 + \Vert {{{\tilde V}_\varepsilon }({r_0})} \Vert_2^2 + {\varepsilon ^2}\Vert {{V_{3,\varepsilon }}({r_0})} \Vert_2^2} \right)\\
                 &+ \int_0^{{r_0}} {\left( {\Vert {{\nabla _H}{{\tilde U}_\varepsilon }} \Vert_2^2 + {\varepsilon ^{\alpha  - 2}}\Vert {{\partial _z}{{\tilde U}_\varepsilon }} \Vert_2^2 + \Vert {{\nabla _H}{{\tilde V}_\varepsilon }} \Vert_2^2 + {\varepsilon ^{\alpha  - 2}}\Vert {{\partial _z}{{\tilde V}_\varepsilon }} \Vert_2^2} \right)dt} \\
                 &+ \int_0^{{r_0}} {\left( {{\varepsilon ^2}\Vert {{\nabla _H}{U_{3,\varepsilon }}} \Vert_2^2 + {\varepsilon ^\alpha }\Vert {{\partial _z}{U_{3,\varepsilon }}} \Vert_2^2 + {\varepsilon ^2}\Vert {{\nabla _H}{V_{3,\varepsilon }}} \Vert_2^2 + {\varepsilon ^\alpha }\Vert {{\partial _z}{V_{3,\varepsilon }}} \Vert_2^2} \right)dt}\\
                 \le& \int_0^{{r_0}} {\int_\Omega  {\left[ {({B_\varepsilon } \cdot \nabla ){{\tilde A}_\varepsilon } \cdot \tilde A + (B \cdot \nabla )\tilde A \cdot {{\tilde A}_\varepsilon } + ({A_\varepsilon } \cdot \nabla ){{\tilde B}_\varepsilon } \cdot \tilde B + (A \cdot \nabla )\tilde B \cdot {{\tilde B}_\varepsilon }} \right]d\Omega } dt} \\
                 &+ {\varepsilon ^2}\int_0^{{r_0}} {\int_\Omega  {\left[ {\left( { - \int_0^z {{\partial _t}\tilde A(x,y,\xi ,t)d\xi } } \right) \cdot {\nabla _H}{U_{3,\varepsilon }}} \right]d\Omega } dt} \\
                 &+ {\varepsilon ^2}\int_0^{{r_0}} {\int_\Omega  {\left[ {\left( { - \int_0^z {{\partial _t}\tilde B(x,y,\xi ,t)d\xi } } \right) \cdot {\nabla _H}{V_{3,\varepsilon }}} \right]d\Omega } dt} \\
                 &+ \int_0^{{r_0}} {\int_\Omega  {( - {\varepsilon ^2}{\nabla _H}{U_{3,\varepsilon }} \cdot {\nabla _H}{A_3} - {\varepsilon ^2}{\nabla _H}{V_{3,\varepsilon }} \cdot {\nabla _H}{B_3})d\Omega } dt} \\
                 &+ \int_0^{{r_0}} {\int_\Omega  {( - {\varepsilon ^{\alpha  - 2}}{\partial _z}{{\tilde U}_\varepsilon } \cdot {\partial _z}\tilde A - {\varepsilon ^\alpha }{\partial _z}{U_{3,\varepsilon }}{\partial _z}{A_3} - {\varepsilon ^{\alpha  - 2}}{\partial _z}{{\tilde V}_\varepsilon } \cdot {\partial _z}\tilde B - {\varepsilon ^\alpha }{\partial _z}{V_{3,\varepsilon }}{\partial _z}{B_3})d\Omega } dt} \\
                 &+ {\varepsilon ^2}\int_0^{{r_0}} {\int_\Omega  {\left[ {({B_\varepsilon } \cdot \nabla {A_{3,\varepsilon }}){A_3} + ({A_\varepsilon } \cdot \nabla {B_{3,\varepsilon }}){B_3}} \right]d\Omega } dt} \\
                 :=& {I_1} + {I_2} + {I_3} + {I_4} + {I_5} + {I_6}.
\end{aligned}
\end{equation}
In what follows, let us estimate each term on the right-hand side of \eqref{3.7} as follows. In fact, for the first integral term $I_1$, we utilize the incompressibility conditions and integration by parts to recast it as
\begin{equation} \label{3.8}
\begin{aligned}
               {I_1} =& \int_0^{{r_0}} {\int_\Omega  {\left[ {({B_\varepsilon } \cdot \nabla ){{\tilde A}_\varepsilon } \cdot \tilde A + (B \cdot \nabla )\tilde A \cdot {{\tilde A}_\varepsilon } + ({A_\varepsilon } \cdot \nabla ){{\tilde B}_\varepsilon } \cdot \tilde B + (A \cdot \nabla )\tilde B \cdot {{\tilde B}_\varepsilon }} \right]d\Omega } dt} \\
               =& \int_0^{{r_0}} {\int_\Omega  {\left[ {({{\tilde V}_\varepsilon } \cdot {\nabla _H}){{\tilde U}_\varepsilon } \cdot \tilde A + ({\nabla _H} \cdot {{\tilde V}_\varepsilon }){{\tilde U}_\varepsilon } \cdot \tilde A + ({{\tilde U}_\varepsilon } \cdot {\partial _z}\tilde A)\int_0^z {({\nabla _H} \cdot {{\tilde V}_\varepsilon })d\xi } } \right]d\Omega } dt} \\
               &+ \int_0^{{r_0}} {\int_\Omega  {\left[ {({{\tilde U}_\varepsilon } \cdot {\nabla _H}){{\tilde V}_\varepsilon } \cdot \tilde B + ({\nabla _H} \cdot {{\tilde U}_\varepsilon }){{\tilde V}_\varepsilon } \cdot \tilde B + ({{\tilde V}_\varepsilon } \cdot {\partial _z}\tilde B)\int_0^z {({\nabla _H} \cdot {{\tilde U}_\varepsilon }) d\xi } } \right]d\Omega } dt} \\
               := & {I_{11}} + {I_{12}} + {I_{13}} + {I_{14}} + {I_{15}} + {I_{16}}.
\end{aligned}
\end{equation}
Noticing that
\begin{equation*}
                \vert {{\tilde A}(z)} \vert \le \frac{1}{2}\int_{ - 1}^1 {\vert {\tilde A} \vert dz}  + \int_{ - 1}^1 {\vert {{\partial _z}{\tilde A}} \vert dz},
\end{equation*}
and invoking Lemma \ref{lem3.1} and the Young inequality, the terms ${I_{11}}$, ${I_{12}}$ and ${I_{13}}$ on the right-hand side of \eqref{3.8} can be estimated as, respectively,
\begin{equation} \label{3.9}
\begin{aligned}
              {I_{11}} =& \int_0^{{r_0}} {\int_\Omega  { {({{\tilde V}_\varepsilon } \cdot {\nabla _H}){{\tilde U}_\varepsilon } \cdot \tilde A} d\Omega } dt} \\
              \le& \int_0^{{r_0}} {\int_M {\left( {\int_{ - 1}^1 {(\vert {\tilde A} \vert + \vert {{\partial _z}\tilde A} \vert)dz} } \right)\left( {\int_{ - 1}^1 {\vert {{{\tilde V}_\varepsilon }} \vert\vert {{\nabla _H}{{\tilde U}_\varepsilon }} \vert dz} } \right)dxdy} dt} \\
              \le& C\int_0^{{r_0}} {\left( {{{\Vert {\tilde A} \Vert}_2} + \Vert {\tilde A} \Vert_2^{\frac{1}{2}}\Vert {{\nabla _H}\tilde A} \Vert_2^{\frac{1}{2}}} \right)\left( {{{\Vert {{{\tilde V}_\varepsilon }} \Vert}_2} + \Vert {{{\tilde V}_\varepsilon }} \Vert_2^{\frac{1}{2}}\Vert {{\nabla _H}{{\tilde V}_\varepsilon }} \Vert_2^{\frac{1}{2}}} \right){{\Vert {{\nabla _H}{{\tilde U}_\varepsilon }} \Vert}_2}dt} \\
              &+ C\int_0^{{r_0}} {\left( {{{\Vert {{\partial _z}\tilde A} \Vert}_2} + \Vert {{\partial _z}\tilde A} \Vert_2^{\frac{1}{2}}\Vert {{\nabla _H}{\partial _z}\tilde A} \Vert_2^{\frac{1}{2}}} \right)\Vert {{{\tilde V}_\varepsilon }} \Vert_2^{\frac{1}{2}}\Vert {{\nabla _H}{{\tilde V}_\varepsilon }} \Vert_2^{\frac{1}{2}}{{\Vert {{\nabla _H}{{\tilde U}_\varepsilon }} \Vert}_2}dt} \\
              &+ C\int_0^{{r_0}} {{{\Vert {{{\tilde V}_\varepsilon }} \Vert}_2}{{\Vert {{\nabla _H}{{\tilde U}_\varepsilon }} \Vert}_2} \left( {{{\Vert {{\partial _z}\tilde A} \Vert}_2} + \Vert {{\partial _z}\tilde A} \Vert_2^{\frac{1}{2}}\Vert {{\nabla _H}{\partial _z}\tilde A} \Vert_2^{\frac{1}{2}}} \right)dt} \\
              \le& C\int_0^{{r_0}} {\left( {\Vert {\tilde A} \Vert_2^2 + {{\Vert {\tilde A} \Vert}_2}{{\Vert {{\nabla _H}\tilde A} \Vert}_2} + \Vert {\tilde A} \Vert_2^4 + \Vert {\tilde A} \Vert_2^2\Vert {{\nabla _H}\tilde A} \Vert_2^2} \right)\Vert {{{\tilde V}_\varepsilon }} \Vert_2^2dt} \\
              &+ C\int_0^{{r_0}} {\left( {\Vert {{\partial _z}\tilde A} \Vert_2^2 + {{\Vert {{\partial _z}\tilde A} \Vert}_2}{{\Vert {{\nabla _H}{\partial _z}\tilde A} \Vert}_2} + \Vert {{\partial _z}\tilde A} \Vert_2^4 + \Vert {{\partial _z}\tilde A} \Vert_2^2\Vert {{\nabla _H}{\partial _z}\tilde A} \Vert_2^2} \right)\Vert {{{\tilde V}_\varepsilon }} \Vert_2^2dt} \\
              &+ \frac{1}{{144}}\int_0^{{r_0}} {\left( {\Vert {{\nabla _H}{{\tilde U}_\varepsilon }} \Vert_2^2 + \Vert {{\nabla _H}{{\tilde V}_\varepsilon }} \Vert_2^2} \right)dt},
\end{aligned}
\end{equation}

\begin{equation} \label{3.10}
\begin{aligned}
               {I_{12}} =& \int_0^{{r_0}} {\int_\Omega  { {({\nabla _H} \cdot {{\tilde V}_\varepsilon }){{\tilde U}_\varepsilon } \cdot \tilde A} d\Omega} dt} \\
               \le& \int_0^{{r_0}} {\int_M {\left( {\int_{ - 1}^1 {(\vert {\tilde A} \vert + \vert {{\partial _z}\tilde A} \vert)dz} } \right)\left( {\int_{ - 1}^1 {\vert {{{\tilde U}_\varepsilon }} \vert\vert {{\nabla _H}{{\tilde V}_\varepsilon }} \vert dz} } \right)dxdy} dt} \\
               \le& C\int_0^{{r_0}} {\left( {\Vert {\tilde A} \Vert_2^2 + {{\Vert {\tilde A} \Vert}_2}{{\Vert {{\nabla _H}\tilde A} \Vert}_2} + \Vert {\tilde A} \Vert_2^4 + \Vert {\tilde A} \Vert_2^2\Vert {{\nabla _H}\tilde A} \Vert_2^2} \right)\Vert {{{\tilde U}_\varepsilon }} \Vert_2^2dt} \\
               &+ C\int_0^{{r_0}} {\left( {\Vert {{\partial _z}\tilde A} \Vert_2^2 + {{\Vert {{\partial _z}\tilde A} \Vert}_2}{{\Vert {{\nabla _H}{\partial _z}\tilde A} \Vert}_2} + \Vert {{\partial _z}\tilde A} \Vert_2^4 + \Vert {{\partial _z}\tilde A} \Vert_2^2\Vert {{\nabla _H}{\partial _z}\tilde A} \Vert_2^2} \right)\Vert {{{\tilde U}_\varepsilon }} \Vert_2^2dt} \\
               &+ \frac{1}{{144}}\int_0^{{r_0}} {\left( {\Vert {{\nabla _H}{{\tilde U}_\varepsilon }} \Vert_2^2 + \Vert {{\nabla _H}{{\tilde V}_\varepsilon }} \Vert_2^2} \right)dt},
\end{aligned}
\end{equation}
and
\begin{equation} \label{3.11}
\begin{aligned}
             {I_{13}} =& \int_0^{{r_0}} {\int_\Omega  {\left[ {({{\tilde U}_\varepsilon } \cdot {\partial _z}\tilde A)\int_0^z {({\nabla _H} \cdot {{\tilde V}_\varepsilon })d\xi } } \right]d\Omega} dt} \\
             \le& \int_0^{{r_0}} {\int_M {\left( {\int_{ - 1}^1 {\vert {{\nabla _H}{{\tilde V}_\varepsilon }} \vert dz} } \right)\left( {\int_{ - 1}^1 {\vert {{{\tilde U}_\varepsilon }} \vert\vert {{\partial _z}\tilde A} \vert dz} } \right)dxdy} dt} \\
             \le& C\int_0^{{r_0}} {\left( {{{\Vert {{\partial _z}\tilde A} \Vert}_2} + \Vert {{\partial _z}\tilde A} \Vert_2^{\frac{1}{2}}\Vert {{\nabla _H}{\partial _z}\tilde A} \Vert_2^{\frac{1}{2}}} \right)\Vert {{{\tilde U}_\varepsilon }} \Vert_2^{\frac{1}{2}}\Vert {{\nabla _H}{{\tilde U}_\varepsilon }} \Vert_2^{\frac{1}{2}}{{\Vert {{\nabla _H}{{\tilde V}_\varepsilon }} \Vert}_2}dt} \\
             &+ C\int_0^{{r_0}} {{{\Vert {{{\tilde U}_\varepsilon }} \Vert}_2}{{\Vert {{\nabla _H}{{\tilde V}_\varepsilon }} \Vert}_2}\left( {{{\Vert {{\partial _z}\tilde A} \Vert}_2} + \Vert {{\partial _z}\tilde A} \Vert_2^{\frac{1}{2}}\Vert {{\nabla _H}{\partial _z}\tilde A} \Vert_2^{\frac{1}{2}}} \right)dt} \\
             \le& C\int_0^{{r_0}} {\left( {\Vert {{\partial _z}\tilde A} \Vert_2^2 + {{\Vert {{\partial _z}\tilde A} \Vert}_2}{{\Vert {{\nabla _H}{\partial _z}\tilde A} \Vert}_2} + \Vert {{\partial _z}\tilde A} \Vert_2^4 + \Vert {{\partial _z}\tilde A} \Vert_2^2\Vert {{\nabla _H}{\partial _z}\tilde A} \Vert_2^2} \right)\Vert {{{\tilde U}_\varepsilon }} \Vert_2^2dt} \\
             &+ \frac{1}{{144}}\int_0^{{r_0}} {\left( {\Vert {{\nabla _H}{{\tilde U}_\varepsilon }} \Vert_2^2 + \Vert {{\nabla _H}{{\tilde V}_\varepsilon }} \Vert_2^2} \right)dt}.
\end{aligned}
\end{equation}
Repeating the similar processes as the terms ${I_{11}}$, ${I_{12}}$ and ${I_{13}}$ on the right-hand side of \eqref{3.8}, we deduce from the terms ${I_{14}}$, ${I_{15}}$ and ${I_{16}}$ that
\begin{equation} \label{3.12}
\begin{aligned}
              {I_{14}} =& \int_0^{{r_0}} {\int_\Omega  { {({{\tilde U}_\varepsilon } \cdot {\nabla _H}){{\tilde V}_\varepsilon } \cdot \tilde B} d\Omega } dt} \\
              \le& C\int_0^{{r_0}} {\left( {\Vert {\tilde B} \Vert_2^2 + {{\Vert {\tilde B} \Vert}_2}{{\Vert {{\nabla _H}\tilde B} \Vert}_2} + \Vert {\tilde B} \Vert_2^4 + \Vert {\tilde B} \Vert_2^2\Vert {{\nabla _H}\tilde B} \Vert_2^2} \right)\Vert {{{\tilde U}_\varepsilon }} \Vert_2^2dt} \\
              &+ C\int_0^{{r_0}} {\left( {\Vert {{\partial _z}\tilde B} \Vert_2^2 + {{\Vert {{\partial _z}\tilde B} \Vert}_2}{{\Vert {{\nabla _H}{\partial _z}\tilde B} \Vert}_2} + \Vert {{\partial _z}\tilde B} \Vert_2^4 + \Vert {{\partial _z}\tilde B} \Vert_2^2\Vert {{\nabla _H}{\partial _z}\tilde B} \Vert_2^2} \right)\Vert {{{\tilde U}_\varepsilon }} \Vert_2^2dt} \\
              &+ \frac{1}{{144}}\int_0^{{r_0}} {\left( {\Vert {{\nabla _H}{{\tilde U}_\varepsilon }} \Vert_2^2 + \Vert {{\nabla _H}{{\tilde V}_\varepsilon }} \Vert_2^2} \right)dt},
\end{aligned}
\end{equation}
\begin{equation} \label{3.13}
\begin{aligned}
               {I_{15}} =& \int_0^{{r_0}} {\int_\Omega  { {({\nabla _H} \cdot {{\tilde U}_\varepsilon }){{\tilde V}_\varepsilon } \cdot \tilde B} d\Omega} dt} \\
               \le& C\int_0^{{r_0}} {\left( {\Vert {\tilde B} \Vert_2^2 + {{\Vert {\tilde B} \Vert}_2}{{\Vert {{\nabla _H}\tilde B} \Vert}_2} + \Vert {\tilde B} \Vert_2^4 + \Vert {\tilde B} \Vert_2^2\Vert {{\nabla _H}\tilde B} \Vert_2^2} \right)\Vert {{{\tilde V}_\varepsilon }} \Vert_2^2dt} \\
               &+ C\int_0^{{r_0}} {\left( {\Vert {{\partial _z}\tilde B} \Vert_2^2 + {{\Vert {{\partial _z}\tilde B} \Vert}_2}{{\Vert {{\nabla _H}{\partial _z}\tilde B} \Vert}_2} + \Vert {{\partial _z}\tilde B} \Vert_2^4 + \Vert {{\partial _z}\tilde B} \Vert_2^2\Vert {{\nabla _H}{\partial _z}\tilde B} \Vert_2^2} \right)\Vert {{{\tilde V}_\varepsilon }} \Vert_2^2dt} \\
               &+ \frac{1}{{144}}\int_0^{{r_0}} {\left( {\Vert {{\nabla _H}{{\tilde U}_\varepsilon }} \Vert_2^2 + \Vert {{\nabla _H}{{\tilde V}_\varepsilon }} \Vert_2^2} \right)dt},
\end{aligned}
\end{equation}
and
\begin{equation} \label{3.14}
\begin{aligned}
             {I_{16}} =& \int_0^{{r_0}} {\int_\Omega  {\left[ {({{\tilde V}_\varepsilon } \cdot {\partial _z}\tilde B)\int_0^z {({\nabla _H} \cdot {{\tilde U}_\varepsilon })d\xi } } \right]d\Omega} dt} \\
             \le& C\int_0^{{r_0}} {\left( {\Vert {{\partial _z}\tilde B} \Vert_2^2 + {{\Vert {{\partial _z}\tilde B} \Vert}_2}{{\Vert {{\nabla _H}{\partial _z}\tilde B} \Vert}_2} + \Vert {{\partial _z}\tilde B} \Vert_2^4 + \Vert {{\partial _z}\tilde B} \Vert_2^2\Vert {{\nabla _H}{\partial _z}\tilde B} \Vert_2^2} \right)\Vert {{{\tilde V}_\varepsilon }} \Vert_2^2dt} \\
             &+ \frac{1}{{144}}\int_0^{{r_0}} {\left( {\Vert {{\nabla _H}{{\tilde U}_\varepsilon }} \Vert_2^2 + \Vert {{\nabla _H}{{\tilde V}_\varepsilon }} \Vert_2^2} \right)dt}.
\end{aligned}
\end{equation}
Hence, collecting the estimates \eqref{3.9}-\eqref{3.14}, we conclude from the first integral term $I_1$ that
\begin{equation} \label{3.15}
   \begin{aligned}
              {I_1}\le& C\int_0^{{r_0}} {\left( {\Vert {\tilde A} \Vert_2^2 + {{\Vert {\tilde A} \Vert}_2}{{\Vert {{\nabla _H}\tilde A} \Vert}_2} + \Vert {\tilde A} \Vert_2^4 + \Vert {\tilde A} \Vert_2^2\Vert {{\nabla _H}\tilde A} \Vert_2^2} \right)\left( {\Vert {{{\tilde U}_\varepsilon }} \Vert_2^2 + \Vert {{{\tilde V}_\varepsilon }} \Vert_2^2} \right)dt} \\
              &+ C\int_0^{{r_0}} {\left( {\Vert {{\partial _z}\tilde A} \Vert_2^2 + {{\Vert {{\partial _z}\tilde A} \Vert}_2}{{\Vert {{\nabla _H}{\partial _z}\tilde A} \Vert}_2} + \Vert {{\partial _z}\tilde A} \Vert_2^4 + \Vert {{\partial _z}\tilde A} \Vert_2^2\Vert {{\nabla _H}{\partial _z}\tilde A} \Vert_2^2} \right)\left( {\Vert {{{\tilde U}_\varepsilon }} \Vert_2^2 + \Vert {{{\tilde V}_\varepsilon }} \Vert_2^2} \right)dt}\\
              &+ C\int_0^{{r_0}} {\left( {\Vert {\tilde B} \Vert_2^2 + {{\Vert {\tilde B} \Vert}_2}{{\Vert {{\nabla _H}\tilde B} \Vert}_2} + \Vert {\tilde B} \Vert_2^4 + \Vert {\tilde B} \Vert_2^2\Vert {{\nabla _H}\tilde B} \Vert_2^2} \right)\left( {\Vert {{{\tilde U}_\varepsilon }} \Vert_2^2 + \Vert {{{\tilde V}_\varepsilon }} \Vert_2^2} \right)dt} \\
              &+ C\int_0^{{r_0}} {\left( {\Vert {{\partial _z}\tilde B} \Vert_2^2 + {{\Vert {{\partial _z}\tilde B} \Vert}_2}{{\Vert {{\nabla _H}{\partial _z}\tilde B} \Vert}_2} + \Vert {{\partial _z}\tilde B} \Vert_2^4 + \Vert {{\partial _z}\tilde B} \Vert_2^2\Vert {{\nabla _H}{\partial _z}\tilde B} \Vert_2^2} \right)\left( {\Vert {{{\tilde U}_\varepsilon }} \Vert_2^2 + \Vert {{{\tilde V}_\varepsilon }} \Vert_2^2} \right)dt}\\
              &+ \frac{1}{{24}}\int_0^{{r_0}} {\left( {\Vert {{\nabla _H}{{\tilde U}_\varepsilon }} \Vert_2^2 + \Vert {{\nabla _H}{{\tilde V}_\varepsilon }} \Vert_2^2} \right)dt}.
\end{aligned}
\end{equation}
With the aid of H\"older's  inequality and Young's inequality, the sum of the integral terms ${I_2}$ and ${I_3}$ can be bounded as
\begin{equation} \label{3.16}
\begin{aligned}
                {I_2} + {I_3} =& {\varepsilon ^2}\int_0^{{r_0}} {\int_\Omega  {\left[ {\left( { - \int_0^z {{\partial _t}\tilde A(x,y,\xi ,t)d\xi } } \right) \cdot {\nabla _H}{U_{3,\varepsilon }}} \right]d\Omega } dt} \\
                &+ {\varepsilon ^2}\int_0^{{r_0}} {\int_\Omega  {\left[ {\left( { - \int_0^z {{\partial _t}\tilde B(x,y,\xi ,t)d\xi } } \right) \cdot {\nabla _H}{V_{3,\varepsilon }}} \right]d\Omega } dt} \\
                \le& {\varepsilon ^2}\int_0^{{r_0}} {\int_M {\left( {\int_{ - 1}^1 {\vert {{\partial _t}\tilde A} \vert dz} } \right)\left( {\int_{ - 1}^1 {\left| {{\nabla _H}{U_{3,\varepsilon }}} \right|dz} } \right)} dt} \\
                &+ {\varepsilon ^2}\int_0^{{r_0}} {\int_M {\left( {\int_{ - 1}^1 {\vert {{\partial _t}\tilde B} \vert dz} } \right)\left( {\int_{ - 1}^1 {\vert {{\nabla _H}{V_{3,\varepsilon }}} \vert dz} } \right)} dt} \\
                \le& C{\varepsilon ^2}\int_0^{{r_0}} {\left( {\Vert {{\partial _t}\tilde A} \Vert_2^2 + \Vert {{\partial _t}\tilde B} \Vert_2^2} \right)dt}  + \frac{1}{{24}}\int_0^{{r_0}} {{\varepsilon ^2}\left( {\Vert {{\nabla _H}{U_{3,\varepsilon }}} \Vert_2^2 + \Vert {{\nabla _H}{V_{3,\varepsilon }}} \Vert_2^2} \right)dt}.
\end{aligned}
\end{equation}
Employing the incompressible conditions, and using H\"older's inequality and Young's inequality, the fourth integral term ${I_4}$ can be estimated as
\begin{equation} \label{3.17}
\begin{aligned}
              {I_4} =& \int_0^{{r_0}} {\int_\Omega  {( - {\varepsilon ^2}{\nabla _H}{U_{3,\varepsilon }} \cdot {\nabla _H}{A_3} - {\varepsilon ^2}{\nabla _H}{V_{3,\varepsilon }} \cdot {\nabla _H}{B_3})d\Omega} dt} \\
              \le& {\varepsilon ^2}\int_0^{{r_0}} {\left( {\Vert {{\nabla _H}{U_{3,\varepsilon }}} \Vert_2\Vert {{\nabla _H}{A_3}} \Vert_2 + \Vert {{\nabla _H}{V_{3,\varepsilon }}} \Vert_2\Vert {{\nabla _H}{B_3}} \Vert_2} \right)dt} \\
              \le& C{\varepsilon ^2}\int_0^{{r_0}} {\left( {\Vert {{\nabla _H}{A_3}} \Vert_2^2 + \Vert {{\nabla _H}{B_3}} \Vert_2^2} \right)dt}  + \frac{1}{{24}}\int_0^{{r_0}} {{\varepsilon ^2}\left( {\Vert {{\nabla _H}{U_{3,\varepsilon }}} \Vert_2^2 + \Vert {{\nabla _H}{V_{3,\varepsilon }}} \Vert_2^2} \right)dt} \\
              \le& C{\varepsilon ^2}\int_0^{{r_0}} {{\int_\Omega  {\vert {{\nabla _H}({\nabla _H} \cdot \tilde A)} \vert}^2 d\Omega}dt}  + C{\varepsilon ^2}\int_0^{{r_0}} {{\int_\Omega  {\vert {{\nabla _H}({\nabla _H} \cdot \tilde B)} \vert}^2 d\Omega}dt} \\
              &+ \frac{1}{{24}}\int_0^{{r_0}} {{\varepsilon ^2}\left( {\Vert {{\nabla _H}{U_{3,\varepsilon }}} \Vert_2^2 + \Vert {{\nabla _H}{V_{3,\varepsilon }}} \Vert_2^2} \right)dt} \\
              \le& C{\varepsilon ^2}\int_0^{{r_0}} {\left( {\Vert {\nabla _H^2\tilde A} \Vert_2^2 + \Vert {\nabla _H^2\tilde B} \Vert_2^2} \right)dt}  + \frac{1}{{24}}\int_0^{{r_0}} {{\varepsilon ^2}\left( {\Vert {{\nabla _H}{U_{3,\varepsilon }}} \Vert_2^2 + \Vert {{\nabla _H}{V_{3,\varepsilon }}} \Vert_2^2} \right)dt},
\end{aligned}
\end{equation}
and the fifth integral term ${I_5}$ can be estimated as
\begin{equation} \label{3.18}
\begin{aligned}
                {I_5} =& \int_0^{{r_0}} {\int_\Omega  {( - {\varepsilon ^{\alpha  - 2}}{\partial _z}{{\tilde U}_\varepsilon } \cdot {\partial _z}\tilde A - {\varepsilon ^\alpha }{\partial _z}{U_{3,\varepsilon }}{\partial _z}{A_3} - {\varepsilon ^{\alpha  - 2}}{\partial _z}{{\tilde V}_\varepsilon } \cdot {\partial _z}\tilde B - {\varepsilon ^\alpha }{\partial _z}{V_{3,\varepsilon }}{\partial _z}{B_3})d\Omega } dt}\\
                \le& {\varepsilon ^{\alpha  - 2}}\int_0^{{r_0}} {\left( {{{\Vert {{\partial _z}{{\tilde U}_\varepsilon }} \Vert}_2}{{\Vert {{\partial _z}\tilde A} \Vert}_2} + {{\Vert {{\partial _z}{{\tilde V}_\varepsilon }} \Vert}_2}{{\Vert {{\partial _z}\tilde B} \Vert}_2}} \right)dt} \\
                &+ {\varepsilon ^\alpha }\int_0^{{r_0}} {\left( {{{\Vert {{\partial _z}{U_{3,\varepsilon }}} \Vert}_2}{{\Vert {{\partial _z}{A_3}} \Vert}_2} + {{\Vert {{\partial _z}{V_{3,\varepsilon }}} \Vert}_2}{{\Vert {{\partial _z}{B_3}} \Vert}_2}} \right)dt} \\
                \le& C{\varepsilon ^{\alpha  - 2}}\int_0^{{r_0}} {\left( {\Vert {{\partial _z}\tilde A} \Vert_2^2 + \Vert {{\partial _z}\tilde B} \Vert_2^2} \right)dt}  + \frac{1}{{24}}\int_0^{{r_0}} {{\varepsilon ^{\alpha  - 2}}\left( {\Vert {{\partial _z}{{\tilde U}_\varepsilon }} \Vert_2^2 + \Vert {{\partial _z}{{\tilde V}_\varepsilon }} \Vert_2^2} \right)dt} \\
                &+ C{\varepsilon ^\alpha }\int_0^{{r_0}} {\left( {\Vert {{\partial _z}{A_3}} \Vert_2^2 + \Vert {{\partial _z}{B_3}} \Vert_2^2} \right)dt}  + \frac{1}{{24}}\int_0^{{r_0}} {{\varepsilon ^\alpha }\left( {\Vert {{\partial _z}{U_{3,\varepsilon }}} \Vert_2^2 + \Vert {{\partial _z}{V_{3,\varepsilon }}} \Vert_2^2} \right)dt} \\
                \le& C{\varepsilon ^{\alpha  - 2}}\int_0^{{r_0}} {\left( {\Vert {\nabla \tilde A} \Vert_2^2 + \Vert {\nabla \tilde B} \Vert_2^2} \right)dt}  + C{\varepsilon ^\alpha }\int_0^{{r_0}} {\left( {\Vert {{\nabla _H}\tilde A} \Vert_2^2 + \Vert {{\nabla _H}\tilde B} \Vert_2^2} \right)dt} \\
                &+ \frac{1}{{24}}\int_0^{{r_0}} {{\varepsilon ^{\alpha  - 2}}\left( {\Vert {{\partial _z}{{\tilde U}_\varepsilon }} \Vert_2^2 + \Vert {{\partial _z}{{\tilde V}_\varepsilon }} \Vert_2^2} \right)dt}  + \frac{1}{{24}}\int_0^{{r_0}} {{\varepsilon ^\alpha }\left( {\Vert {{\partial _z}{U_{3,\varepsilon }}} \Vert_2^2 + \Vert {{\partial _z}{V_{3,\varepsilon }}} \Vert_2^2} \right)dt}.
\end{aligned}
\end{equation}
By means of the incompressibility constraints, Lemma \ref{lem3.1} and Young's inequality again, the last integral term ${I_6}$ on the right-hand side of \eqref{3.7} can be controlled by
\begin{equation} \label{3.19}
\begin{aligned}
               {I_6} =& {\varepsilon ^2}\int_0^{{r_0}} {\int_\Omega  {\left[ {({B_\varepsilon } \cdot \nabla {A_{3,\varepsilon }}){A_3} + ({A_\varepsilon } \cdot \nabla {B_{3,\varepsilon }}){B_3}} \right]d\Omega} dt} \\
               =& {\varepsilon ^2}\int_0^{{r_0}} {\int_\Omega  {\left[ {{B_{3,\varepsilon }}({\nabla _H} \cdot {{\tilde U}_\varepsilon })\int_0^z {({\nabla _H} \cdot \tilde A})d\xi- {{\tilde B}_\varepsilon } \cdot {\nabla _H}{U_{3,\varepsilon }} \int_0^z ({{\nabla _H} \cdot \tilde A} )d\xi} \right]d\Omega} dt} \\
               &+ {\varepsilon ^2}\int_0^{{r_0}} {\int_\Omega  {\left[ {{A_{3,\varepsilon }}({\nabla _H} \cdot {{\tilde V}_\varepsilon } )\int_0^z ({{\nabla _H} \cdot \tilde B})d\xi - {{\tilde A}_\varepsilon } \cdot {\nabla _H}{V_{3,\varepsilon }}\int_0^z ({{\nabla _H} \cdot \tilde B})d\xi} \right]d\Omega} dt} \\
               \le& C{\varepsilon ^2}\int_0^{{r_0}} {\left( {\Vert {{\nabla _H}\tilde A} \Vert_2^2\Vert {\nabla _H^2\tilde A} \Vert_2^2 + {\varepsilon ^4}\Vert {{A_{3,\varepsilon }}} \Vert_2^4 + {\varepsilon ^4}\Vert {{A_{3,\varepsilon }}} \Vert_2^2\Vert {{\nabla _H}{A_{3,\varepsilon }}} \Vert_2^2} \right)dt} \\
               &+ C{\varepsilon ^2}\int_0^{{r_0}} {\left( {\Vert {{\nabla _H}\tilde B} \Vert_2^2\Vert {\nabla _H^2\tilde B} \Vert_2^2 + {\varepsilon ^4}\Vert {{B_{3,\varepsilon }}} \Vert_2^4 + {\varepsilon ^4}\Vert {{B_{3,\varepsilon }}} \Vert_2^2\Vert {{\nabla _H}{B_{3,\varepsilon }}} \Vert_2^2} \right)dt} \\
               &+ C{\varepsilon ^2}\int_0^{{r_0}} {\left( {\Vert {{{\tilde A}_\varepsilon }} \Vert_2^4 + \Vert {{{\tilde A}_\varepsilon }} \Vert_2^2\Vert {{\nabla _H}{{\tilde A}_\varepsilon }} \Vert_2^2 + \Vert {{{\tilde B}_\varepsilon }} \Vert_2^4 + \Vert {{{\tilde B}_\varepsilon }} \Vert_2^2\Vert {{\nabla _H}{{\tilde B}_\varepsilon }} \Vert_2^2} \right)dt} \\
               &+ \frac{1}{{24}}\int_0^{{r_0}} {\left( {\Vert {{\nabla _H}{{\tilde U}_\varepsilon }} \Vert_2^2 + {\varepsilon ^2}\Vert {{\nabla _H}{U_{3,\varepsilon }}} \Vert_2^2 + \Vert {{\nabla _H}{{\tilde V}_\varepsilon }} \Vert_2^2 + {\varepsilon ^2}\Vert {{\nabla _H}{V_{3,\varepsilon }}} \Vert_2^2} \right)dt}.
\end{aligned}
\end{equation}
Plugging the estimates of all integral terms ${I_1}$-${I_6}$ into \eqref{3.7}, we arrive at
\begin{equation*}
      \begin{aligned}
              m(t): =& \left( {\Vert {{{\tilde U}_\varepsilon }(t)} \Vert_2^2 + {\varepsilon ^2}\Vert {{U_{3,\varepsilon }}(t)} \Vert_2^2 + \Vert {{{\tilde V}_\varepsilon }(t)} \Vert_2^2 + {\varepsilon ^2}\Vert {{V_{3,\varepsilon }}(t)} \Vert_2^2}\right) \\
              &+ \int_0^t {\left( {\Vert {{\nabla _H}{{\tilde U}_\varepsilon }} \Vert_2^2 + {\varepsilon ^{\alpha  - 2}}\Vert {{\partial _z}{{\tilde U}_\varepsilon }} \Vert_2^2 + \Vert {{\nabla _H}{{\tilde V}_\varepsilon }} \Vert_2^2 + {\varepsilon ^{\alpha  - 2}}\Vert {{\partial _z}{{\tilde V}_\varepsilon }} \Vert_2^2} \right)ds} \\
              &+ \int_0^t {\left( {{\varepsilon ^2}\Vert {{\nabla _H}{U_{3,\varepsilon }}} \Vert_2^2 + {\varepsilon ^\alpha }\Vert {{\partial _z}{U_{3,\varepsilon }}} \Vert_2^2 + {\varepsilon ^2}\Vert {{\nabla _H}{V_{3,\varepsilon }}} \Vert_2^2 + {\varepsilon ^\alpha }\Vert {{\partial _z}{V_{3,\varepsilon }}} \Vert_2^2} \right)ds}\\
              \le& C\int_0^t {{Y_1}(s)X(s)ds}  + C\int_0^t {{Y_2}(s)X(s)ds}  + C\int_0^t {{Y_3}(s)ds} \\
              & + C\int_0^t {{Y_4}(s)ds} + C\int_0^t {{Y_5}(s)ds} + C\int_0^t {{Y_6}(s)ds}  + C\int_0^t {{Y_7}(s)ds} := M(t),
      \end{aligned}
\end{equation*}
for a.e. $t \in [0,t_1^ * )$, where we define
\begin{equation*}
      \begin{aligned}
                {X}(t):=& {\Vert {{{\tilde U}_\varepsilon }} \Vert_2^2 + \Vert {{{\tilde V}_\varepsilon }} \Vert_2^2},\\
                {Y_1}(t):=&  {\Vert {\tilde A} \Vert_2^2 + {{\Vert {\tilde A} \Vert}_2}{{\Vert {{\nabla _H}\tilde A} \Vert}_2} + \Vert {\tilde A} \Vert_2^4 + \Vert {\tilde A} \Vert_2^2\Vert {{\nabla _H}\tilde A} \Vert_2^2}\\
                &+ {\Vert {\tilde B} \Vert_2^2 + {{\Vert {\tilde B} \Vert}_2}{{\Vert {{\nabla _H}\tilde B} \Vert}_2} + \Vert {\tilde B} \Vert_2^4 + \Vert {\tilde B} \Vert_2^2\Vert {{\nabla _H}\tilde B} \Vert_2^2},\\
                {Y_2}(t):=& {\Vert {{\partial _z}\tilde A} \Vert_2^2 + {{\Vert {{\partial _z}\tilde A} \Vert}_2}{{\Vert {{\nabla _H}{\partial _z}\tilde A} \Vert}_2} + \Vert {{\partial _z}\tilde A} \Vert_2^4 + \Vert {{\partial _z}\tilde A} \Vert_2^2\Vert {{\nabla _H}{\partial _z}\tilde A} \Vert_2^2}\\
                &+{\Vert {{\partial _z}\tilde B} \Vert_2^2 + {{\Vert {{\partial _z}\tilde B} \Vert}_2}{{\Vert {{\nabla _H}{\partial _z}\tilde B} \Vert}_2} + \Vert {{\partial _z}\tilde B} \Vert_2^4 + \Vert {{\partial _z}\tilde B} \Vert_2^2\Vert {{\nabla _H}{\partial _z}\tilde B} \Vert_2^2},\\
                {Y_3}(t):=& {\varepsilon ^2}\left({\Vert {{\partial _t}\tilde A} \Vert_2^2 + \Vert {\nabla _H^2\tilde A} \Vert_2^2 + \Vert {{\partial _t}\tilde B} \Vert_2^2 + \Vert {\nabla _H^2\tilde B} \Vert_2^2}\right),\\
                {Y_4}(t):=& {\varepsilon ^{\alpha  - 2}}\left({\Vert {\nabla \tilde A} \Vert_2^2 + \Vert {\nabla \tilde B} \Vert_2^2}\right),\\
                {Y_5}(t):=& {\varepsilon ^\alpha }\left({\Vert {{\nabla _H}\tilde A} \Vert_2^2 + \Vert {{\nabla _H}\tilde B} \Vert_2^2}\right),\\
                {Y_6}(t):=& {\varepsilon ^2}\left( {\Vert {{\nabla _H}\tilde A} \Vert_2^2\Vert {\nabla _H^2\tilde A} \Vert_2^2 + {\varepsilon ^4}\Vert {{A_{3,\varepsilon }}} \Vert_2^4 + {\varepsilon ^4}\Vert {{A_{3,\varepsilon }}} \Vert_2^2\Vert {{\nabla _H}{A_{3,\varepsilon }}} \Vert_2^2} \right.\\
                &\left. {+ {\Vert {{\nabla _H}\tilde B} \Vert_2^2\Vert {\nabla _H^2\tilde B} \Vert_2^2 + {\varepsilon ^4}\Vert {{B_{3,\varepsilon }}} \Vert_2^4 + {\varepsilon ^4}\Vert {{B_{3,\varepsilon }}} \Vert_2^2\Vert {{\nabla _H}{B_{3,\varepsilon }}} \Vert_2^2}} \right),\\
                {Y_7}(t):=& {\varepsilon ^2}\left({\Vert {{{\tilde A}_\varepsilon }} \Vert_2^4 + \Vert {{{\tilde A}_\varepsilon }} \Vert_2^2\Vert {{\nabla _H}{{\tilde A}_\varepsilon }} \Vert_2^2 + \Vert {{{\tilde B}_\varepsilon }} \Vert_2^4 + \Vert {{{\tilde B}_\varepsilon }} \Vert_2^2\Vert {{\nabla _H}{{\tilde B}_\varepsilon }} \Vert_2^2}\right).
      \end{aligned}
\end{equation*}
After applying ${\partial _t}$ to $M(t)$ and using the inequality $X(t) < m(t) \le M(t)$, we have that
\begin{equation*}
\begin{aligned}
               M'(t) =& C\left[ {{Y_1}(t) + {Y_2}(t)} \right]X(t) + C\left[ {{Y_3}(t) + {Y_4}(t) + {Y_5}(t) + {Y_6}(t) + {Y_7}(t)} \right]\\
               \le& C\left[ {{Y_1}(t) + {Y_2}(t)} \right]M(t) + C\left[ {{Y_3}(t) + {Y_4}(t) + {Y_5}(t) + {Y_6}(t) + {Y_7}(t)} \right],
    \end{aligned}
\end{equation*}
and thus it follows from Grönwall's inequality that
\begin{equation} \label{3.20}
                   m(t) \le C\exp \left\{ {C\int_0^t {\left[ {{Y_1}(s) + {Y_2}(s)} \right]ds} } \right\}\int_0^t {\left[ {{Y_3}(s) + {Y_4}(s) + {Y_5}(s) + {Y_6}(s) + {Y_7}(s)} \right]ds},
\end{equation}
where we made use of the fact that ${\left. M \right|_{t = 0}} = 0$. Summarily, combining \eqref{3.1} and \eqref{3.6}, we obtain the desired estimate \eqref{3.3}. This completes the proof of Proposition \ref{prop3.4}.
\end{proof}

\begin{proof2.3}
Let ${ t_1^ * }$ be the maximal existence time of the local-in-time strong solution to system \eqref{2.10}. On the basis of Proposition \ref{prop3.4}, we can choose $\gamma = \min \{ 2,\alpha - 2\}$ with $\alpha \in (2,\infty )$ such that
\begin{equation*}
    \begin{aligned}
                \mathop {\sup }\limits_{0 \le t < t_1^ * }& \left( {\Vert {({{\tilde U}_\varepsilon },\varepsilon {U_{3,\varepsilon }},{{\tilde V}_\varepsilon },\varepsilon {V_{3,\varepsilon }})} \Vert_2^2} \right)(t)\\
                &+ \int_0^{t_1^ *} {\left( {\Vert {{\nabla _H}{{\tilde U}_\varepsilon }} \Vert_2^2 + {\varepsilon ^{\alpha  - 2}}\Vert {{\partial _z}{{\tilde U}_\varepsilon }} \Vert_2^2 + \Vert {{\nabla _H}{{\tilde V}_\varepsilon }} \Vert_2^2 + {\varepsilon ^{\alpha  - 2}}\Vert {{\partial _z}{{\tilde V}_\varepsilon }} \Vert_2^2} \right)ds} \\
                &+ \int_0^{t_1^ *} {\left( {{\varepsilon ^2}\Vert {{\nabla _H}{U_{3,\varepsilon }}} \Vert_2^2 + {\varepsilon ^\alpha }\Vert {{\partial _z}{U_{3,\varepsilon }}} \Vert_2^2 + {\varepsilon ^2}\Vert {{\nabla _H}{V_{3,\varepsilon }}} \Vert_2^2 + {\varepsilon ^\alpha }\Vert {{\partial _z}{V_{3,\varepsilon }}} \Vert_2^2} \right)ds} \\
                \le C{\varepsilon ^\gamma }&(t_1^ *  + 1){e^{C(t_1^ *  + 1)}}\left[ {1 + {\left( {\Vert {{{\tilde A}_0}} \Vert_2^2 + \Vert {{A_{3,0}}} \Vert_2^2 + \Vert {{{\tilde B}_0}} \Vert_2^2 + \Vert {{B_{3,0}}} \Vert_2^2} \right)^2}} \right],
    \end{aligned}
\end{equation*}
for some positive constant $C$ independent of ${\varepsilon }$, which implies that
\begin{equation*}
    \begin{aligned}
               &({{\tilde A}_\varepsilon },\varepsilon {A_{3,\varepsilon }},{{\tilde B}_\varepsilon },\varepsilon {B_{3,\varepsilon }}) \to (\tilde A,0,\tilde B,0)\;\;{\rm{in}}\;{L^\infty }([0,t_1^*);{L^2}(\Omega )),\\
               &({\nabla _H}{{\tilde A}_\varepsilon },{\varepsilon ^{(\alpha  - 2)/2}}{\partial _z}{{\tilde A}_\varepsilon },\varepsilon {\nabla _H}{A_{3,\varepsilon }},{\varepsilon ^{\alpha /2}}{\partial _z}{A_{3,\varepsilon }}) \to ({\nabla _H}\tilde A,0,0,0)\;\;{\rm{in}}\;{L^2}([0,t_1^*);{L^2}(\Omega )),\\
               &({\nabla _H}{{\tilde B}_\varepsilon },{\varepsilon ^{(\alpha  - 2)/2}}{\partial _z}{{\tilde B}_\varepsilon },\varepsilon {\nabla _H}{B_{3,\varepsilon }},{\varepsilon ^{\alpha /2}}{\partial _z}{B_{3,\varepsilon }}) \to ({\nabla _H}\tilde B,0,0,0)\;\;{\rm{in}}\;{L^2}([0,t_1^*);{L^2}(\Omega )).
    \end{aligned}
\end{equation*}
By virtue of the strong convergence
\begin{equation*}
               ({\nabla _H}{{\tilde A}_\varepsilon },{\nabla _H}{{\tilde B}_\varepsilon }) \to ({\nabla _H}\tilde A,{\nabla _H}\tilde B)\;\;{\rm{in}}\;{L^2}([0,t_1^*);{L^2}(\Omega ))
\end{equation*}
and the divergence-free properties, we can conclude that
\begin{equation*}
              ({A_{3,\varepsilon }},{B_{3,\varepsilon }}) \to ({A_3},{B_3})\;\;{\rm{in}}\;{L^2}([0,t_1^*);{L^2}(\Omega )).
\end{equation*}
In addition, it is clear that the rate of convergence is of the order $O({\varepsilon ^{\gamma /2}})$.
\end{proof2.3}

\subsection{The $H^1$-initial data with additional regularity}
In this subsection, we aim at proving that if the initial pair $({{\tilde A}_0},{{\tilde B}_0}) \in {H^1}(\Omega )$ has additional regularity $({\partial _z}{{\tilde A}_0},{\partial _z}{{\tilde B}_0}) \in {L^p}(\Omega )$, for some $p \in (2,\infty )$, then the local-in-time strong solution $({\tilde A},{\tilde B})$ to system \eqref{2.10} can be extended to be a global-in-time one, and studying the global-in-time strong convergence of the SHMHD equations to the PEHM as $\varepsilon$ tends to zero.

\begin{proposition}\label{prop3.5}
Suppose that a pair of periodic function $({{\tilde A}_0},{{\tilde B}_0}) \in {H^1}(\Omega ) \cap {L^\infty }(\Omega )$ with
\begin{equation*}
                 {\nabla _H} \cdot \left( {\int_{ - 1}^1 {{{\tilde A}_0}(x,y,z )dz} } \right) = 0\;{\rm{and}}\;{\nabla _H} \cdot \left( {\int_{ - 1}^1 {{{\tilde B}_0}(x,y,z )dz} } \right) = 0,
\end{equation*}
and that $({\partial _z}{{\tilde A}_0},{\partial _z}{{\tilde B}_0}) \in {L^p}(\Omega )$, for some $p \in (2,\infty )$. Then the local-in-time strong solution $({\tilde A},{\tilde B})$ of system \eqref{2.10} corresponding to \eqref{2.4}-\eqref{2.6} can be extended uniquely to be a global-in-time one which satisfies the following energy estimate
\begin{equation} \label{3.21}
 \begin{aligned}
           \mathop {\sup }\limits_{0 \le s \le t} &\left( {\Vert {(\tilde A,\tilde B)} \Vert_{{H^1}(\Omega )}^2} \right)(s)\\
           &+ \int_0^t {\left( {\Vert {{\nabla _H}\tilde A} \Vert_{{H^1}(\Omega )}^2 + \Vert {{\nabla _H}\tilde B} \Vert_{{H^1}(\Omega )}^2 + \Vert {({\partial _t}\tilde A,{\partial _t}\tilde B)} \Vert_2^2} \right)ds}  \le {N_1}(t),
 \end{aligned}
\end{equation}
for any $t\in [0,\infty)$, where $N(t)$ is a no-negative continuous increasing function defined on $[0,\infty)$.
\end{proposition}

\begin{proof}
The proof of the existence of global-in-time strong solutions to system \eqref{2.10} follows from the similar argument in \cite{2017-Cao-Li-Titi-Global}. This completes the proof of Proposition \ref{prop3.5}.
\end{proof}

The next proposition is utilized in the proof of Theorem \ref{thm2.4}.
\begin{proposition}\label{prop3.6}
Let $(\tilde A,\tilde B)$ be the local-in-time strong solution of system \eqref{2.10}, corresponding to \eqref{2.4}-\eqref{2.6}. Then, this pair of solution must satisfy the following inequality
\begin{equation*}
\begin{aligned}
               &\mathop {\sup }\limits_{(x,y,z) \in \bar \Omega } \vert {\tilde A(x,y,z,{t_\eta })} \vert + \mathop {\sup }\limits_{(x,y,z) \in \bar \Omega } \vert {\tilde B(x,y,z,{t_\eta })} \vert\\
               &\le {C}\left( {{{\Vert {{\nabla _H}\tilde A} \Vert}_6} + {{\Vert {{\partial _z}\tilde A} \Vert}_2} + {{\Vert {{\nabla _H}\tilde B} \Vert}_6} + {{\Vert {{\partial _z}\tilde B} \Vert}_2}} \right)({t_\eta }),
  \end{aligned}
\end{equation*}
for some fixed time ${t_\eta } \in [0,t_1^ * )$.
\end{proposition}

\begin{proof}
         Taking $p_1 = 6$, $p_2 = 6$ and $p_3 = 2$, a direct application of Lemma \ref{rem2.2} in \cite{2022-Li-Titi} leads to Proposition \ref{prop3.6}.
\end{proof}

\begin{proof2.4} Recalling \eqref{3.1} in Proposition \ref{prop3.2}, we obtain that
\begin{equation*}
               \int_{t_1^*/8}^{t_1^ * } {\left( {\Vert {\nabla {\nabla _H}\tilde A} \Vert_2^2 + \Vert {\nabla {\nabla _H}\tilde B} \Vert_2^2} \right)ds}  \le C,
\end{equation*}
and we further choose a fixed time ${t_1} \in (t_1^*/8,t_1^ * )$ such that
\begin{equation} \label{3.22}
               {\Vert {\nabla {\nabla _H}\tilde A} \Vert_2^2}(t_1^*) + {\Vert {\nabla {\nabla _H}\tilde B} \Vert_2^2}(t_1^*) \le \frac{C}{{t_1^*}}.
\end{equation}
In addition, the following energy estimate holds for system \eqref{2.10}, that is,
\begin{equation} \label{3.23}
            \mathop {\sup }\limits_{0 \le s \le {t_1}} \left( {\Vert {(\tilde A,\tilde B)} \Vert_{{H^1}(\Omega )}^2} \right)(s) + \int_0^{t_1} {\left( {\Vert {{\nabla _H}\tilde A} \Vert_{{H^1}(\Omega )}^2 + \Vert {{\nabla _H}\tilde B} \Vert_{{H^1}(\Omega )}^2 + \Vert {({\partial _t}\tilde A,{\partial _t}\tilde B)} \Vert_2^2} \right)ds} \le C,
\end{equation}
With help of Proposition \ref{prop3.6}, Sobolev embedding theorem, \eqref{3.22} and \eqref{3.23}, one deduces
\begin{equation*}
    \begin{aligned}
                &\mathop {\sup }\limits_{(x,y,z) \in \bar \Omega } \vert {\tilde A(x,y,z,{t_1})} \vert + \mathop {\sup }\limits_{(x,y,z) \in \bar \Omega } \vert {\tilde B(x,y,z,{t_1})} \vert\\
\le &C\left( {{{\Vert {\nabla {\nabla _H}\tilde A} \Vert}_2}({t_1}) + {{\Vert {{\nabla _H}\tilde A} \Vert}_2}({t_1}) + {{\Vert {{\partial _z}\tilde A} \Vert}_2}({t_1})} \right)\\
                &+ C\left( {{{\Vert {\nabla {\nabla _H}\tilde B} \Vert}_2}({t_1}) + {{\Vert {{\nabla _H}\tilde B} \Vert}_2}({t_1}) + {{\Vert {{\partial _z}\tilde B} \Vert}_2}({t_1})} \right)\\
\le &C\left( {1 + \Vert {\nabla {\nabla _H}\tilde A} \Vert_2^2({t_1}) + \Vert {\nabla \tilde A} \Vert_2^2({t_1}) + \Vert {\nabla {\nabla _H}\tilde B} \Vert_2^2({t_1}) + \Vert {\nabla \tilde B} \Vert_2^2}({t_1}) \right)\\
\le &C(1 + 1/t_1^*),
\end{aligned}
\end{equation*}
and in particular we have $(\tilde A(x,y,z,{t_1}),\tilde B(x,y,z,{t_1})) \in {L^\infty }(\Omega )$.

In what follows, we want to prove that $({\partial _z}\tilde A,{\partial _z}\tilde B) \in {L^\infty }([0,t_1^ * );{L^p}(\Omega ))$ for some $p \in (2,\infty )$. On the account of the incompressible constraints, differentiating system \eqref{2.10} with respect to $z$, one can easily derive that
\begin{equation} \label{3.24}
\begin{aligned}
                 {\partial _t}({\partial _z}\tilde A) &+ ({\partial _z}\tilde u \cdot {\nabla _H})\tilde A + (\tilde u \cdot {\nabla _H}){\partial _z}\tilde A - ({\nabla _H} \cdot \tilde u){\partial _z}\tilde A\\
                 &- \left( {\int_0^z {{\nabla _H} \cdot \tilde u(x,y,\xi ,t)d\xi } } \right){\partial _{zz}}\tilde A - {\Delta _H}{\partial _z}\tilde A - ({\partial _z}\tilde b \cdot {\nabla _H})\tilde A\\
                 &- (\tilde b \cdot {\nabla _H}){\partial _z}\tilde A + ({\nabla _H} \cdot \tilde b){\partial _z}\tilde A + \left( {\int_0^z {{\nabla _H} \cdot \tilde b(x,y,\xi ,t)d\xi } } \right){\partial _{zz}}\tilde A = 0,
\end{aligned}
\end{equation}
and
\begin{equation} \label{3.25}
\begin{aligned}
                {\partial _t}({\partial _z}\tilde B) &+ ({\partial _z}\tilde u \cdot {\nabla _H})\tilde B + (\tilde u \cdot {\nabla _H}){\partial _z}\tilde B - ({\nabla _H} \cdot \tilde u){\partial _z}\tilde B\\
                &- \left( {\int_0^z {{\nabla _H} \cdot \tilde u(x,y,\xi ,t)d\xi } } \right){\partial _{zz}}\tilde B - {\Delta _H}{\partial _z}\tilde B + ({\partial _z}\tilde b \cdot {\nabla _H})\tilde B\\
                &+ (\tilde b \cdot {\nabla _H}){\partial _z}\tilde B - ({\nabla _H} \cdot \tilde b){\partial _z}\tilde B - \left( {\int_0^z {{\nabla _H} \cdot \tilde b(x,y,\xi ,t)d\xi } } \right){\partial _{zz}}\tilde B = 0.
\end{aligned}
\end{equation}
We begin with multiplying \eqref{3.24} and \eqref{3.25} by ${\vert {{\partial _z}\tilde A} \vert^{p - 2}}{\partial _z}\tilde A$ and ${\vert {{\partial _z}\tilde B} \vert^{p - 2}}{\partial _z}\tilde B$, respectively, and integrate the resultants over $\Omega$, and then by the help of the integration by parts, we can obtain that
\begin{equation} \label{3.26}
\begin{aligned}
                 \frac{1}{p}&\frac{d}{{dt}}\Vert {{\partial _z}\tilde A} \Vert_p^p + \frac{1}{p}\frac{d}{{dt}}\Vert {{\partial _z}\tilde B} \Vert_p^p + \int_\Omega  {{{\vert {{\partial _z}\tilde A} \vert}^{p - 2}}\left( {{{\vert {{\nabla _H}{\partial _z}\tilde A} \vert}^2} + (p - 2){\left| {{\nabla _H}\vert {{\partial _z}\tilde A} \vert} \right|^2}} \right)d\Omega } \\
                 &+ \int_\Omega  {{{\vert {{\partial _z}\tilde B} \vert}^{p - 2}}\left( {{{\vert {{\nabla _H}{\partial _z}\tilde B} \vert}^2} + (p - 2){\left| {{\nabla _H}\vert {{\partial _z}\tilde B} \vert} \right|^2}} \right)d\Omega } \\
                 =& - \int_\Omega  {\left[ {({\partial _z}\tilde u \cdot {\nabla _H})\tilde A - ({\nabla _H} \cdot \tilde u){\partial _z}\tilde A - ({\partial _z}\tilde b \cdot {\nabla _H})\tilde A + ({\nabla _H} \cdot \tilde b){\partial _z}\tilde A} \right] \cdot {{\vert {{\partial _z}\tilde A} \vert}^{p - 2}}{\partial _z}\tilde Ad\Omega } \\
                 &- \int_\Omega  {\left[ {({\partial _z}\tilde u \cdot {\nabla _H})\tilde B - ({\nabla _H} \cdot \tilde u){\partial _z}\tilde B + ({\partial _z}\tilde b \cdot {\nabla _H})\tilde B - ({\nabla _H} \cdot \tilde b){\partial _z}\tilde B} \right] \cdot {{\vert {{\partial _z}\tilde B} \vert}^{p - 2}}{\partial _z}\tilde Bd\Omega }\\
                 \le& \int_\Omega  {\left( {\vert {{\partial _z}\tilde u} \vert\vert {{\nabla _H}\tilde A} \vert{{\vert {{\partial _z}\tilde A} \vert}^{p - 1}} + \vert {{\nabla _H}\tilde u} \vert{{\vert {{\partial _z}\tilde A} \vert}^p} + \vert {{\partial _z}\tilde b} \vert\vert {{\nabla _H}\tilde A} \vert{{\vert {{\partial _z}\tilde A} \vert}^{p - 1}} + \vert {{\nabla _H}\tilde b} \vert{{\vert {{\partial _z}\tilde A} \vert}^p}} \right)d\Omega } \\
                 &+ \int_\Omega  {\left( {\vert {{\partial _z}\tilde u} \vert\vert {{\nabla _H}\tilde B} \vert{{\vert {{\partial _z}\tilde B} \vert}^{p - 1}} + \vert {{\nabla _H}\tilde u} \vert{{\vert {{\partial _z}\tilde B} \vert}^p} + \vert {{\partial _z}\tilde b} \vert\vert {{\nabla _H}\tilde B} \vert{{\vert {{\partial _z}\tilde B} \vert}^{p - 1}} + \vert {{\nabla _H}\tilde b} \vert{{\vert {{\partial _z}\tilde B} \vert}^p}} \right)d\Omega } \\
                 \le& \int_\Omega  {\left( {4\vert {{\nabla _H}\tilde A} \vert{{\vert {{\partial _z}\tilde A} \vert}^p} + 4\vert {{\nabla _H}\tilde B} \vert{{\vert {{\partial _z}\tilde B} \vert}^p}} \right)d\Omega }:= {K_1} + {K_2},
\end{aligned}
\end{equation}
where the periodic boundary condition \eqref{2.4} and symmetric condition \eqref{2.6} have been applied. To deal with the integral terms ${K_1}$ and ${K_2}$, we invoke Lemma \ref{lem3.1} and the Young inequality to derive
\begin{equation} \label{3.27}
\begin{aligned}
               {K_1} =& 4\int_\Omega  {\vert {{\nabla _H}\tilde A} \vert{{\vert {{\partial _z}\tilde A} \vert}^p}d\Omega } \\
               =& 4\int_M {\left( {\int_{ - 1}^1 {(\vert {{\nabla _H}\tilde A} \vert + \vert {{\nabla _H}{\partial _z}\tilde A} \vert)dz} } \right)\left( {\int_{ - 1}^1 {{{\vert {{\partial _z}\tilde A} \vert}^{\frac{p}{2}}}{{\vert {{\partial _z}\tilde A} \vert}^{\frac{p}{2}}}dz} } \right)dxdy} \\
               \le& C\left( {{{\Vert {{\nabla _H}\tilde A} \Vert}_2} + {{\Vert {{\nabla _H}{\partial _z}\tilde A} \Vert}_2}} \right){\left\| {{{\vert {{\partial _z}\tilde A} \vert}^{\frac{p}{2}}}} \right\|_2}\left( {{\left\| {{{\vert {{\partial _z}\tilde A} \vert}^{\frac{p}{2}}}} \right\|_2} + {{\left\| {{\nabla _H}{{\vert {{\partial _z}\tilde A} \vert}^{\frac{p}{2}}}} \right\|}_2}} \right)\\
               \le& C\left( {1 + \Vert {{\nabla _H}\tilde A} \Vert_2^2 + \Vert {{\nabla _H}{\partial _z}\tilde A} \Vert_2^2} \right)\Vert {{\partial _z}\tilde A} \Vert_p^p + \frac{1}{16}\int_\Omega  {{{\vert {{\nabla _H}{\partial _z}\tilde A} \vert}^2}{{\vert {{\partial _z}\tilde A} \vert}^{p - 2}}d\Omega } \\
               \le& C\left( {1 + \Vert {{\nabla _H}\tilde A} \Vert_{{H^1}(\Omega )}^2} \right)\Vert {{\partial _z}\tilde A} \Vert_p^p + \frac{1}{16}\int_\Omega  {{{\vert {{\nabla _H}{\partial _z}\tilde A} \vert}^2}{{\vert {{\partial _z}\tilde A} \vert}^{p - 2}}d\Omega },
\end{aligned}
\end{equation}
and
\begin{equation} \label{3.28}
\begin{aligned}
               {K_2} =& 4\int_\Omega  {\vert {{\nabla _H}\tilde B} \vert{{\vert {{\partial _z}\tilde B} \vert}^p}d\Omega }\\
               \le& C\left( {1 + \Vert {{\nabla _H}\tilde B} \Vert_{{H^1}(\Omega )}^2} \right)\Vert {{\partial _z}\tilde B} \Vert_p^p + \frac{1}{16}\int_\Omega  {{{\vert {{\nabla _H}{\partial _z}\tilde B} \vert}^2}{{\vert {{\partial _z}\tilde B} \vert}^{p - 2}}d\Omega }.
\end{aligned}
\end{equation}
Summarily, combining \eqref{3.26}, \eqref{3.27} and \eqref{3.28}, we obtain
\begin{equation*}
    \begin{aligned}
                 \frac{d}{{dt}}&\Vert {{\partial _z}\tilde A} \Vert_p^p + \frac{d}{{dt}}\Vert {{\partial _z}\tilde B} \Vert_p^p + \frac{p}{2}\int_\Omega  {{{\vert {{\partial _z}\tilde A} \vert}^{p - 2}}\left( {{{\vert {{\nabla _H}{\partial _z}\tilde A} \vert}^2} + (p - 2){\left| {{\nabla _H}\vert {{\partial _z}\tilde A} \vert} \right|^2}} \right)d\Omega } \\
                 &+ \frac{p}{2}\int_\Omega  {{{\vert {{\partial _z}\tilde B} \vert}^{p - 2}}\left( {{{\vert {{\nabla _H}{\partial _z}\tilde B} \vert}^2} + (p - 2){\left| {{\nabla _H}\vert {{\partial _z}\tilde B} \vert} \right|^2}} \right)d\Omega } \\
                 \le& C\left( {1 + \Vert {{\nabla _H}\tilde A} \Vert_{{H^1}(\Omega )}^2 + \Vert {{\nabla _H}\tilde B} \Vert_{{H^1}(\Omega )}^2} \right)\left( {\Vert {{\partial _z}\tilde A} \Vert_p^p + \Vert {{\partial _z}\tilde B} \Vert_p^p} \right),
\end{aligned}
\end{equation*}
which together with the Grönwall inequality and Proposition \ref{prop3.2} gives that
\begin{equation*}
    \begin{aligned}
              &\mathop {\sup }\limits_{0 \le s < t_1^ * } \left( {\Vert {({\partial _z}\tilde A,{\partial _z}\tilde B)} \Vert_p^p} \right)(s)\\
              &\le \exp\left[ {C\int_0^{t_1^ * } {\left( {1 + \Vert {{\nabla _H}\tilde A} \Vert_{{H^1}(\Omega )}^2 + \Vert {{\nabla _H}\tilde B} \Vert_{{H^1}(\Omega )}^2} \right)ds} } \right]\left( {\Vert {{\partial _z}{{\tilde A}_0}} \Vert_p^p + \Vert {{\partial _z}{{\tilde B}_0}} \Vert_p^p} \right)\\
              &\le C{e^{C(1 + t_1^ * )}}\left( {\Vert {{\partial _z}{{\tilde A}_0}} \Vert_p^p + \Vert {{\partial _z}{{\tilde B}_0}} \Vert_p^p} \right),
    \end{aligned}
\end{equation*}
and in particular we have $({\partial _z}\tilde A(x,y,z,t_1),{\partial _z}\tilde B(x,y,z,t_1)) \in {L^p}(\Omega )$. Let us recall that $(\tilde A(x,y,z,t_1),$ $\tilde B(x,y,z,t_1))\in {L^\infty }(\Omega )$. Then we conclude from Proposition \ref{prop3.5} that
\begin{equation} \label{3.29}
\begin{aligned}
           \mathop {\sup }\limits_{{t_1} \le s \le t} &\left( {\Vert {(\tilde A,\tilde B)} \Vert_{{H^1}(\Omega )}^2} \right)(s)\\
           &+ \int_{t_1}^t {\left( {\Vert {{\nabla _H}\tilde A} \Vert_{{H^1}(\Omega )}^2 + \Vert {{\nabla _H}\tilde B} \Vert_{{H^1}(\Omega )}^2 + \Vert {({\partial _t}\tilde A,{\partial _t}\tilde B)} \Vert_2^2} \right)ds}  \le {N_1}(t),
\end{aligned}
\end{equation}
for any $t\in [t_1,\infty)$. Thus, it follows from \eqref{3.23} and \eqref{3.29} that the energy estimate in Theorem \ref{thm2.4} holds. \end{proof2.4}

\begin{proof2.6} This proof can be obtained arguing exactly as in the proof of Proposition \ref{prop3.4} with only one difference that we now state. We observe that since the local-in-time strong solution $(\tilde A,\tilde B)$ is extended to the global-in-time one, we cannot apply \eqref{3.1} in Proposition \ref{prop3.2}. Instead, plugging \eqref{3.6} and the energy estimate in Theorem \ref{thm2.4} into \eqref{3.20}, for any $\mathcal{T}$ we obtain the following estimate
\begin{equation*}
\begin{aligned}
                \mathop {\sup }\limits_{0 \le t \le \mathcal{T}} &\left( {\Vert {({{\tilde U}_\varepsilon },\varepsilon {U_{3,\varepsilon }},{{\tilde V}_\varepsilon },\varepsilon {V_{3,\varepsilon }})} \Vert_2^2} \right)(t)\\
                &+ \int_0^\mathcal{T} {\left( {\Vert {{\nabla _H}{{\tilde U}_\varepsilon }} \Vert_2^2 + {\varepsilon ^{\alpha  - 2}}\Vert {{\partial _z}{{\tilde U}_\varepsilon }} \Vert_2^2 + \Vert {{\nabla _H}{{\tilde V}_\varepsilon }} \Vert_2^2 + {\varepsilon ^{\alpha  - 2}}\Vert {{\partial _z}{{\tilde V}_\varepsilon }} \Vert_2^2} \right)ds} \\
                &+ \int_0^\mathcal{T} {\left( {{\varepsilon ^2}\Vert {{\nabla _H}{U_{3,\varepsilon }}} \Vert_2^2 + {\varepsilon ^\alpha }\Vert {{\partial _z}{U_{3,\varepsilon }}} \Vert_2^2 + {\varepsilon ^2}\Vert {{\nabla _H}{V_{3,\varepsilon }}} \Vert_2^2 + {\varepsilon ^\alpha }\Vert {{\partial _z}{V_{3,\varepsilon }}} \Vert_2^2} \right)ds} \\
                \le &C\exp \left\{ {C(\mathcal{T} + 1)\left[ {{N_2}(\mathcal{T}) + N_2^2(\mathcal{T})} \right]} \right\} \times \left\{ {\left( {{\varepsilon ^2} + {\varepsilon ^\alpha }} \right){N_2}(\mathcal{T}) + {\varepsilon ^2}N_2^2(\mathcal{T})} \right.\\
                &\left. { + {\varepsilon ^{\alpha  - 2}}\mathcal{T}{N_2}(\mathcal{T}) + {\varepsilon ^2}(\mathcal{T} + 1){\left( {\Vert {{{\tilde A}_0}} \Vert_2^2 + {\varepsilon ^2}\Vert {{A_{3,0}}} \Vert_2^2 + \Vert {{{\tilde B}_0}} \Vert_2^2 + {\varepsilon ^2}\Vert {{B_{3,0}}} \Vert_2^2} \right)^2}} \right\}\\
                \le &C{\varepsilon ^\gamma }(\mathcal{T} + 1)\exp \left\{ {C(\mathcal{T} + 1)\left[ {{N_2}(\mathcal{T}) + N_2^2(\mathcal{T})} \right]} \right\}\\
                &\times \left[ {{N_2}(\mathcal{T}) + N_2^2(\mathcal{T}) + {\left( {\Vert {{{\tilde A}_0}} \Vert_2^2 + \Vert {{A_{3,0}}} \Vert_2^2 + \Vert {{{\tilde B}_0}} \Vert_2^2 + \Vert {{B_{3,0}}} \Vert_2^2} \right)^2}} \right] := {\varepsilon ^\gamma }{N_3}(\mathcal{T}),
    \end{aligned}
\end{equation*}
 for some positive constant $C$ independent of $\varepsilon$, where ${N_3}(t)$ is nonnegative continuous increasing function on $[0,\infty)$ and $\gamma = \min \{ 2,\alpha - 2\}$ with $\alpha \in (2,\infty )$. As a result, the local-in-time strong convergences in Theorem \ref{thm2.3} can be extended to the global-in-time strong convergences, and the convergence rate is of the order $O({\varepsilon ^{\gamma /2}})$. \end{proof2.6}

\section{Strong convergence for the $H^2$-initial data}
In this section, our aim is devoted to the strong convergence of SHMHD equations \eqref{2.3} to PEHM \eqref{2.10}, with the initial data $({{\tilde A}_0},{{\tilde B}_0}) \in {H^2}(\Omega )$, as the aspect ratio parameter $\varepsilon$ tends zero, in other words, we present the proof of Theorem \ref{thm2.7}.

Under the assumption that an initial pair $({{\tilde A}_0},{{\tilde B}_0}) \in {H^2}(\Omega )$ fulfilling
\begin{equation*}
                 {\nabla _H} \cdot \left( {\int_{ - 1}^1 {{{\tilde A}_0}(x,y,z )dz} } \right) = 0\;{\rm{and}}\;{\nabla _H} \cdot \left( {\int_{ - 1}^1 {{{\tilde B}_0}(x,y,z )dz} } \right) = 0,
\end{equation*}
for any $(x,y)\in M$, system \eqref{2.3} supplemented with \eqref{2.4}-\eqref{2.6} admits a unique local-in-time strong solution $({{\tilde A}_\varepsilon },{A_{3,\varepsilon }},{{\tilde B}_\varepsilon },{B_{3,\varepsilon }})$. Denote by ${\mathcal{T}_\varepsilon ^ *}$ the maximal time of the existence of this local-in-time strong solution. Furthermore, since PEHM \eqref{2.10} shares the same velocity field with the oceanic and atmospheric PEs, we can apply the similar argument as in \cite{2016-Cao-Li-Titi-Global} to prove that there is a unique global-in-time strong solution $(\tilde A,\tilde B)$ to system \eqref{2.10}. The global well-posedness of strong solutions of system \eqref{2.10} is stated as follows.

\begin{proposition}\label{prop4.1}
Let an initial pair $({{\tilde A}_0},{{\tilde B}_0}) \in {H^2}(\Omega )$ be two periodic functions and satisfy
\begin{equation*}
                  {\nabla _H} \cdot \left( {\int_{ - 1}^1 {{{\tilde A}_0}(x,y,z)dz} } \right) = 0\;{\rm{and}}\;{\nabla _H} \cdot \left( {\int_{ - 1}^1 {{{\tilde B}_0}(x,y,z)dz} } \right) = 0,
\end{equation*}
for any $(x,y)\in M$. Then the following conclusions are true:

(i) For any $\mathcal{T} > 0$, PEHM \eqref{2.10}, corresponding to \eqref{2.4}-\eqref{2.6}, has a unique global-in-time strong solution $(\tilde A,\tilde B)$ fulfilling the following regularities
\begin{equation*}
    \begin{aligned}
              &(\tilde A,\tilde B) \in {L^\infty }([0,\mathcal{T}];{H^2}(\Omega )) \cap C([0,\mathcal{T}];{H^1}(\Omega )),\\
              &({\nabla _H}\tilde A,{\nabla _H}\tilde B) \in {L^2}([0,\mathcal{T}];{H^2}(\Omega )),\; ({\partial _t}\tilde A,{\partial _t}\tilde B) \in {L^2}([0,\mathcal{T}];{H^1}(\Omega ));
        \end{aligned}
\end{equation*}

(ii) The global-in-time strong solution $(\tilde A,\tilde B)$ to PEHM \eqref{2.10} fulfils the following estimate
\begin{equation} \label{4.1}
        \begin{aligned}
                  \mathop {\sup }\limits_{0 \le s \le t} &\left( {\Vert {(\tilde A,\tilde B)} \Vert_{{H^2}(\Omega )}^2} \right)(s) + \int_0^t {\left( {\Vert {{\nabla _H}\tilde A} \Vert_{{H^2}(\Omega )}^2 + \Vert {{\nabla _H}\tilde B} \Vert_{{H^2}(\Omega )}^2} \right)ds} \\
                  &+ \int_0^t {\left( {\Vert {({\partial _t}\tilde A,{\partial _t}\tilde B)} \Vert_{{H^1}(\Omega )}^2} \right)} ds \le {N_4}(t),
        \end{aligned}
\end{equation}
for $t \in [0,\infty)$, where ${N_4}(t)$ is a non-negative continuous increasing function defined on $[0,\infty)$.
\end{proposition}

In what follows, we denote
\begin{equation*}
    \begin{aligned}
                ({U_\varepsilon },{V_\varepsilon },{P_\varepsilon }) =& ({{\tilde U}_\varepsilon },{U_{3,\varepsilon }},{{\tilde V}_\varepsilon },{V_{3,\varepsilon }},{P_\varepsilon }),\\
                ({{\tilde U}_\varepsilon },{U_{3,\varepsilon }},{{\tilde V}_\varepsilon },{V_{3,\varepsilon }},{P_\varepsilon }) =& ({{\tilde A}_\varepsilon } - \tilde A,{A_{3,\varepsilon }} - {A_3},{{\tilde B}_\varepsilon } - \tilde B,{B_{3,\varepsilon }} - {B_3},{p_\varepsilon } - p),
     \end{aligned}
\end{equation*}
and subtract system \eqref{2.10} from system \eqref{2.3} to lead to the following difference system
\begin{equation} \label{4.2}
\begin{cases}
    \begin{aligned}
                  {\partial _t}{{\tilde U}_\varepsilon } &- {\Delta _H}{{\tilde U}_\varepsilon } - {\varepsilon ^{\alpha  - 2}}{\partial _{zz}}{{\tilde U}_\varepsilon } - {\varepsilon ^{\alpha  - 2}}{\partial _{zz}}\tilde A + (B \cdot \nabla ){{\tilde U}_\varepsilon } + ({V_\varepsilon } \cdot \nabla )\tilde A\\
                  &+ ({V_\varepsilon } \cdot \nabla ){{\tilde U}_\varepsilon } + {\nabla _H}{P_\varepsilon } = 0,\\
                  {\varepsilon ^2}({\partial _t}&{U_{3,\varepsilon }} - {\Delta _H}{U_{3,\varepsilon }} - {\varepsilon ^{\alpha  - 2}}{\partial _{zz}}{U_{3,\varepsilon }} + B \cdot \nabla {A_3} + B \cdot \nabla {U_{3,\varepsilon }} + {V_\varepsilon } \cdot \nabla {A_3})\\
                  &+ {\varepsilon ^2}{V_\varepsilon } \cdot \nabla {U_{3,\varepsilon }} + {\partial _z}{P_\varepsilon } + {\varepsilon ^2}({\partial _t}{A_3} - {\Delta _H}{A_3} + {\varepsilon ^{\alpha  - 2}}{\nabla _H} \cdot {\partial _z}\tilde A) = 0,\\
                  {\partial _t}{{\tilde V}_\varepsilon } &- {\Delta _H}{{\tilde V}_\varepsilon } - {\varepsilon ^{\alpha  - 2}}{\partial _{zz}}{{\tilde V}_\varepsilon } - {\varepsilon ^{\alpha  - 2}}{\partial _{zz}}\tilde B + (A \cdot \nabla ){{\tilde V}_\varepsilon } + ({U_\varepsilon } \cdot \nabla )\tilde B\\
                  &+ ({U_\varepsilon } \cdot \nabla ){{\tilde V}_\varepsilon } + {\nabla _H}{P_\varepsilon } = 0,\\
                  {\varepsilon ^2}({\partial _t}&{V_{3,\varepsilon }} - {\Delta _H}{V_{3,\varepsilon }} - {\varepsilon ^{\alpha  - 2}}{\partial _{zz}}{V_{3,\varepsilon }} + A \cdot \nabla {B_3} + A \cdot \nabla {V_{3,\varepsilon }} + {U_\varepsilon } \cdot \nabla {B_3})\\
                  &+ {\varepsilon ^2}{U_\varepsilon } \cdot \nabla {V_{3,\varepsilon }} + {\partial _z}{P_\varepsilon } + {\varepsilon ^2}({\partial _t}{B_3} - {\Delta _H}{B_3} + {\varepsilon ^{\alpha  - 2}}{\nabla _H} \cdot {\partial _z}\tilde B) = 0,\\
                  {\nabla _H} \cdot& {{\tilde U}_\varepsilon } + {\partial _z}{U_{3,\varepsilon }} = 0,\\
                  {\nabla _H} \cdot& {{\tilde V}_\varepsilon } + {\partial _z}{V_{3,\varepsilon }} = 0,
        \end{aligned}
\end{cases}
\end{equation}
where $A = (\tilde A,{A_3})$ and $B = (\tilde B,{B_3})$.

\begin{proposition}\label{prop4.2}
Assume that $({{\tilde A}_0},{{\tilde B}_0}) \in {H^2}(\Omega )$ satisfying
\begin{equation*}
                  {\nabla _H} \cdot \left( {\int_{ - 1}^1 {{{\tilde A}_0}(x,y,z)dz} } \right) = 0\;{\rm{and}}\;{\nabla _H} \cdot \left( {\int_{ - 1}^1 {{{\tilde B}_0}(x,y,z)dz} } \right) = 0,
\end{equation*}
for any $(x,y)\in M$. Then the following basic energy estimate holds for system \eqref{4.2}, that is,
\begin{equation*}
    \begin{aligned}
                    \mathop {\sup }\limits_{0 \le s \le t}& \left( {\Vert {({{\tilde U}_\varepsilon },\varepsilon {U_{3,\varepsilon }},{{\tilde V}_\varepsilon },\varepsilon {V_{3,\varepsilon }})} \Vert_2^2} \right)(s)\\
                    &+ \int_0^t {\left( {\Vert {\nabla _H}{{\tilde U}_\varepsilon } \Vert_2^2 + {\varepsilon ^{\alpha  - 2}}\Vert  {\partial _z}{{\tilde U}_\varepsilon }} \Vert_2^2 + \Vert {\nabla _H}{{\tilde V}_\varepsilon } \Vert_2^2 + {\varepsilon ^{\alpha  - 2}}\Vert {\partial _z}{{\tilde V}_\varepsilon } \Vert_2^2 \right)ds} \\
                    &+ \int_0^t {\left( {{\varepsilon ^2}\Vert {\nabla _H}{U_{3,\varepsilon }}} \Vert_2^2 + {\varepsilon ^\alpha }\Vert {\partial _z}{U_{3,\varepsilon }} \Vert_2^2 + {\varepsilon ^2}\Vert {\nabla _H}{V_{3,\varepsilon }} \Vert_2^2 + {\varepsilon ^\alpha }\Vert {\partial _z}{V_{3,\varepsilon }} \Vert_2^2 \right)ds} \le {\varepsilon ^\gamma }{N_5}(t),
    \end{aligned}
\end{equation*}
for any $t \in [0,{\mathcal{T}_\varepsilon ^ *} )$, where $\gamma = \min \{ 2,\alpha - 2\}$ with $\alpha \in (2,\infty )$. Notice that for some positive constant $C$ independent of $\varepsilon$, a non-negative continuous increasing function ${N_5}(t)$ defined on $[0,\infty)$, is denoted by
\begin{equation*}
                 {N_5}(t) = C(t + 1){e^{C(t + 1)\left[ {{N_4}(t) + N_4^2(t)} \right]}}\left[ {{N_4}(t) + N_4^2(t) + {\left( {\Vert {{{\tilde A}_0}} \Vert_2^2 + \Vert {{A_{3,0}}} \Vert_2^2 + \Vert {{{\tilde B}_0}} \Vert_2^2 + \Vert {{B_{3,0}}} \Vert_2^2} \right)^2}} \right].
\end{equation*}
\end{proposition}
\begin{proof}
The proof of Proposition \ref{prop4.2} can be obtained arguing exactly as in the proof of Theorem \ref{thm2.6} (see Section 3 for details) with only one difference that we now state. We observe that in the case $H^2$-initial data we can apply the energy estimate \eqref{4.1} instead of the energy estimate in Theorem \ref{thm2.4}.
\end{proof}

Under a suitable smallness assumption, based on Proposition \ref{prop4.1} and Proposition \ref{prop4.2}, we also obtain the first-order energy estimate of the solution $({\tilde U_\varepsilon },{U_{3,\varepsilon }},{\tilde V_\varepsilon },{V_{3,\varepsilon }})$ to system \eqref{4.2} in the following proposition.

\begin{proposition}\label{prop4.3}
Let a pair of periodic function $({{\tilde A}_0},{{\tilde B}_0}) \in {H^2}(\Omega )$ and fulfil
\begin{equation*}
                  {\nabla _H} \cdot \left( {\int_{ - 1}^1 {{{\tilde A}_0}(x,y,z)dz} } \right) = 0\;{\rm{and}}\;{\nabla _H} \cdot \left( {\int_{ - 1}^1 {{{\tilde B}_0}(x,y,z)dz} } \right) = 0,
\end{equation*}
for any $(x,y)\in M$. Suppose that
\begin{equation*}
                \mathop {\sup}\limits_{0 \le s \le t} \left( {\Vert {\nabla ( {{{\tilde U}_\varepsilon },{{\tilde V}_\varepsilon }} )} \Vert_2^2 + {\varepsilon ^2}\Vert {\nabla {U_{3,\varepsilon }}} \Vert_2^2 + {\varepsilon ^2}\Vert {\nabla {V_{3,\varepsilon }}} \Vert_2^2} \right)(s) \le \delta_0^2,
\end{equation*}
then there exists a small constant ${\delta_0} = \sqrt {\frac{1}{{3{C_1}}}} > 0$ such that the first-order energy estimate of the solution $({\tilde U_\varepsilon },{U_{3,\varepsilon }},{\tilde V_\varepsilon },{V_{3,\varepsilon }})$ of system \eqref{4.2} holds, that is,
\begin{equation*}
    \begin{aligned}
                    \mathop {\sup }\limits_{0 \le s \le t} &\left( {\Vert {\nabla ({{\tilde U}_\varepsilon },\varepsilon {U_{3,\varepsilon }},{{\tilde V}_\varepsilon },\varepsilon {V_{3,\varepsilon }})} \Vert_2^2} \right)(s)\\
                    &+ \int_0^t {\left( {\Vert {\nabla {\nabla _H}{{\tilde U}_\varepsilon }} \Vert_2^2 + {\varepsilon ^{\alpha  - 2}}\Vert {\nabla {\partial _z}{{\tilde U}_\varepsilon }} \Vert_2^2 + \Vert {\nabla {\nabla _H}{{\tilde V}_\varepsilon }} \Vert_2^2 + {\varepsilon ^{\alpha  - 2}}\Vert {\nabla {\partial _z}{{\tilde V}_\varepsilon }} \Vert_2^2} \right)ds} \\
                    &+ \int_0^t {\left( {{\varepsilon ^2}\Vert {\nabla {\nabla _H}{U_{3,\varepsilon }}} \Vert_2^2 + {\varepsilon ^\alpha }\Vert {\nabla {\partial _z}{U_{3,\varepsilon }}} \Vert_2^2 + {\varepsilon ^2}\Vert {\nabla {\nabla _H}{V_{3,\varepsilon }}} \Vert_2^2 + {\varepsilon ^\alpha }\Vert {\nabla {\partial _z}{V_{3,\varepsilon }}} \Vert_2^2} \right)ds}  \le {\varepsilon ^\gamma }{N_6}(t),
    \end{aligned}
\end{equation*}
for any $t \in [0,{\mathcal{T}_\varepsilon ^ *} )$, where
\begin{equation*}
\begin{aligned}
               {N_6}(t) \le& C(t + 1)\exp \left\{ {C(t + 1)\left[ {{N_4}(t) + N_4^2(t) + {N_5}(t) + N_5^2(t) + 1} \right]} \right\}\\
               &\times \left[ {{N_4}(t) + N_4^2(t) + {N_5}(t) + N_5^2(t) + {N_4}(t){N_5}(t)} \right]
\end{aligned}
\end{equation*}
is a non-negative continuous increasing function ${N_6}(t)$ defined on $[0,\infty)$. Here $\gamma  = \min \{ 2,\alpha  - 2\}$ with $\alpha \in (2,\infty )$ and some positive constant $C$ dose not depend on $\varepsilon$.
\end{proposition}

\begin{proof}
One can take the inner product of $(4.2)_1$, $(4.2)_2$, $(4.2)_3$ and $(4.2)_4$ in $L^2(\Omega)$ with $- \Delta {{\tilde U}_\varepsilon }$, $- \Delta {U_{3,\varepsilon }}$, $- \Delta {{\tilde V}_\varepsilon }$ and $- \Delta {V_{3,\varepsilon }}$, respectively, and perform integration by parts, which gives rise to the integral equality
\begin{equation*}
    \begin{aligned}
                  \frac{1}{2}&\frac{d}{{dt}}\left( {\Vert {\nabla ({{\tilde U}_\varepsilon },\varepsilon {U_{3,\varepsilon }},{{\tilde V}_\varepsilon },\varepsilon {V_{3,\varepsilon }})} \Vert_2^2} \right)\\
                  &+ \left( {\Vert {\nabla {\nabla _H}{{\tilde U}_\varepsilon }} \Vert_2^2 + {\varepsilon ^{\alpha  - 2}}\Vert {\nabla {\partial _z}{{\tilde U}_\varepsilon }} \Vert_2^2 + {\varepsilon ^2} \Vert{\nabla {\nabla _H}{U_{3,\varepsilon }}} \Vert_2^2 + {\varepsilon ^\alpha }\Vert {\nabla {\partial _z}{U_{3,\varepsilon }}} \Vert_2^2} \right)\\
                  &+ \left( {\Vert {\nabla {\nabla _H}{{\tilde V}_\varepsilon }} \Vert_2^2 + {\varepsilon ^{\alpha  - 2}}\Vert {\nabla {\partial _z}{{\tilde V}_\varepsilon }} \Vert_2^2 + {\varepsilon ^2}\Vert {\nabla {\nabla _H}{V_{3,\varepsilon }}} \Vert_2^2 + {\varepsilon ^\alpha }\Vert {\nabla {\partial _z}{V_{3,\varepsilon }}} \Vert_2^2} \right)\\
    \end{aligned}
\end{equation*}
\begin{equation} \label{4.3}
        \begin{aligned}
                  =&\int_\Omega  {[(B \cdot \nabla ){{\tilde U}_\varepsilon } + ({V_\varepsilon } \cdot \nabla )\tilde A + ({V_\varepsilon } \cdot \nabla ){{\tilde U}_\varepsilon }] \cdot \Delta {{\tilde U}_\varepsilon }d\Omega} \\
                  &+ {\varepsilon ^2}\int_\Omega  {(B \cdot \nabla {A_3} + B \cdot \nabla {U_{3,\varepsilon }} + {V_\varepsilon } \cdot \nabla {A_3} + {V_\varepsilon } \cdot \nabla {U_{3,\varepsilon }})\Delta {U_{3,\varepsilon }}d\Omega} \\
                  &+ \int_\Omega  {[(A \cdot \nabla ){{\tilde V}_\varepsilon } + ({U_\varepsilon } \cdot \nabla )\tilde B + ({U_\varepsilon } \cdot \nabla ){{\tilde V}_\varepsilon }] \cdot \Delta {{\tilde V}_\varepsilon }d\Omega} \\
                  &+ {\varepsilon ^2}\int_\Omega  {(A \cdot \nabla {B_3} + A \cdot \nabla {V_{3,\varepsilon }} + {U_\varepsilon } \cdot \nabla {B_3} + {U_\varepsilon } \cdot \nabla {V_{3,\varepsilon }})\Delta {V_{3,\varepsilon }}d\Omega} \\
                  &+ \int_\Omega  {[{\varepsilon ^2}({\partial _t}{A_3} - {\Delta _H}{A_3})\Delta {U_{3,\varepsilon }}]d\Omega} + \int_\Omega  {[{\varepsilon ^\alpha }({\nabla _H} \cdot {\partial _z}{\tilde A})\Delta {U_{3,\varepsilon }} - {\varepsilon ^{\alpha  - 2}}({\partial _{zz}}{\tilde A}) \cdot \Delta {{\tilde U}_\varepsilon }]d\Omega} \\
                  &+ \int_\Omega  {[{\varepsilon ^2}({\partial _t}{B_3} - {\Delta _H}{B_3})\Delta {V_{3,\varepsilon }}]d\Omega} + \int_\Omega  {[{\varepsilon ^\alpha }({\nabla _H} \cdot {\partial _z}{\tilde B})\Delta {V_{3,\varepsilon }} - {\varepsilon ^{\alpha  - 2}}({\partial _{zz}}{\tilde B}) \cdot \Delta {{\tilde V}_\varepsilon }]d\Omega}\\
=:& {J_1} + {J_2} + {J_3} + {J_4} + {J_5} + {J_6} + {J_7} + {J_8}.
        \end{aligned}
\end{equation}
Now let us estimate each factor on the right-hand side of \eqref{4.3}. To estimate the first integral factor $J_1$, we can divide it into three terms, namely, ${J_{11}}$, ${J_{12}}$ and ${J_{13}}$. On the basis of the incompressible constraints, Lemma \ref{lem3.1} and Young's inequality, these integral terms can be estimated as
\begin{equation} \label{4.4}
        \begin{aligned}
                  {J_{11}} =& \int_\Omega  {(B \cdot \nabla ){{\tilde U}_\varepsilon } \cdot \Delta {{\tilde U}_\varepsilon }d\Omega} \\
                  =& \int_\Omega  {\left[ {(\tilde B \cdot {\nabla _H}){{\tilde U}_\varepsilon } - ({\partial _z}{{\tilde U}_\varepsilon })\int_0^z {({\nabla _H} \cdot \tilde B)d\xi } } \right] \cdot {\Delta _H}{{\tilde U}_\varepsilon }d\Omega} \\
                  &+ \int_\Omega  {[( - {\partial _z}\tilde B \cdot {\nabla _H}){{\tilde U}_\varepsilon } - 2(\tilde B \cdot {\nabla _H}){\partial _z}{{\tilde U}_\varepsilon }] \cdot {\partial _z}{{\tilde U}_\varepsilon }d\Omega}\\
                  \le& \int_M {\left( {\int_{ - 1}^1 {(\vert {\tilde B} \vert + \vert {{\partial _z}\tilde B} \vert)dz} } \right)\left( {\int_{ - 1}^1 {\vert {{\nabla _H}{{\tilde U}_\varepsilon }} \vert\vert {{\Delta _H}{{\tilde U}_\varepsilon }} \vert dz} } \right)dxdy} \\
                  &+ \int_M {\left( {\int_{ - 1}^1 {\vert {{\nabla _H}\tilde B} \vert dz} } \right)\left( {\int_{ - 1}^1 {\vert {{\partial _z}{{\tilde U}_\varepsilon }} \vert\vert {{\Delta _H}{{\tilde U}_\varepsilon }} \vert dz} } \right)dxdy} \\
                  &+ \int_M {\left( {\int_{ - 1}^1 {(\vert {{\nabla _H}{{\tilde U}_\varepsilon }} \vert + \vert {{\nabla _H}{\partial _z}{{\tilde U}_\varepsilon }} \vert)dz} } \right)\left( {\int_{ - 1}^1 {\vert {{\partial _z}\tilde B} \vert\vert {{\partial _z}{{\tilde U}_\varepsilon }} \vert dz} } \right)dxdy} \\
                  &+ 2\int_M {\left( {\int_{ - 1}^1 {(\vert {\tilde B} \vert + \vert {{\partial _z}\tilde B} \vert)dz} } \right)\left( {\int_{ - 1}^1 {\vert {{\nabla _H}{\partial _z}{{\tilde U}_\varepsilon }} \vert\vert {{\partial _z}{{\tilde U}_\varepsilon }} \vert dz} } \right)dxdy} \\
                  \le& C\left[ {\Vert {\tilde B} \Vert_2^{\frac{1}{2}}\left( {\Vert {\tilde B} \Vert_2^{\frac{1}{2}} + \Vert {{\nabla _H}\tilde B} \Vert_2^{\frac{1}{2}}} \right) + \Vert {{\partial _z}\tilde B} \Vert_2^{\frac{1}{2}}\left( {\Vert {{\partial _z}\tilde B} \Vert_2^{\frac{1}{2}} + \Vert {{\nabla _H}{\partial _z}\tilde B} \Vert_2^{\frac{1}{2}}} \right)} \right]\\
                  &\times \Vert {{\nabla _H}{{\tilde U}_\varepsilon }} \Vert_2^{\frac{1}{2}}\left( {\Vert {{\nabla _H}{{\tilde U}_\varepsilon }} \Vert_2^{\frac{1}{2}} + \Vert {\nabla _H^2{{\tilde U}_\varepsilon }} \Vert_2^{\frac{1}{2}}} \right){\Vert {{\Delta _H}{{\tilde U}_\varepsilon }} \Vert_2}\\
                  &+ C\left( {{{\Vert {{\nabla _H}\tilde B} \Vert}_2} + \Vert {{\nabla _H}\tilde B} \Vert_2^{\frac{1}{2}}\Vert {\nabla _H^2\tilde B} \Vert_2^{\frac{1}{2}}} \right)\Vert {{\partial _z}{{\tilde U}_\varepsilon }} \Vert_2^{\frac{1}{2}}\Vert {{\nabla _H}{\partial _z}{{\tilde U}_\varepsilon }} \Vert_2^{\frac{1}{2}}{\Vert {{\Delta _H}{{\tilde U}_\varepsilon }} \Vert_2}\\
                  &+ C{\Vert {{\partial _z}{{\tilde U}_\varepsilon }} \Vert_2}{\Vert {{\Delta _H}{{\tilde U}_\varepsilon }} \Vert_2}\left( {{{\Vert {{\nabla _H}\tilde B} \Vert}_2} + \Vert {{\nabla _H}\tilde B} \Vert_2^{\frac{1}{2}}\Vert {\nabla _H^2\tilde B} \Vert_2^{\frac{1}{2}}} \right)\\
                  &+ C\left( {{{\Vert {{\nabla _H}{{\tilde U}_\varepsilon }} \Vert}_2} + {{\Vert {{\nabla _H}{\partial _z}{{\tilde U}_\varepsilon }} \Vert}_2}} \right)\Vert {{\partial _z}\tilde B} \Vert_2^{\frac{1}{2}}\left( {\Vert {{\partial _z}\tilde B} \Vert_2^{\frac{1}{2}} + \Vert {{\nabla _H}{\partial _z}\tilde B} \Vert_2^{\frac{1}{2}}} \right)\\
                  &\times \Vert {{\partial _z}{{\tilde U}_\varepsilon }} \Vert_2^{\frac{1}{2}}\left( {\Vert {{\partial _z}{{\tilde U}_\varepsilon }} \Vert_2^{\frac{1}{2}} + \Vert {{\nabla _H}{\partial _z}{{\tilde U}_\varepsilon }} \Vert_2^{\frac{1}{2}}} \right)\\
        \end{aligned}
\end{equation}
\begin{equation*}
    \begin{aligned}
                  &+ C\left[ {\Vert {\tilde B} \Vert_2^{\frac{1}{2}}\left( {\Vert {\tilde B} \Vert_2^{\frac{1}{2}} + \Vert {{\nabla _H}\tilde B} \Vert_2^{\frac{1}{2}}} \right) + \Vert {{\partial _z}\tilde B} \Vert_2^{\frac{1}{2}}\left( {\Vert {{\partial _z}\tilde B} \Vert_2^{\frac{1}{2}} + \Vert {{\nabla _H}{\partial _z}\tilde B} \Vert_2^{\frac{1}{2}}} \right)} \right]\\
                  &\times \Vert {{\partial _z}{{\tilde U}_\varepsilon }} \Vert_2^{\frac{1}{2}}\left( {\Vert{{\partial _z}{{\tilde U}_\varepsilon }} \Vert_2^{\frac{1}{2}} + \Vert {{\nabla _H}{\partial _z}{{\tilde U}_\varepsilon }} \Vert_2^{\frac{1}{2}}} \right){\Vert {{\nabla _H}{\partial _z}{{\tilde U}_\varepsilon }} \Vert_2}\\
                  \le& C\left( {1 + \Vert {\tilde B} \Vert_2^2 + \Vert {{\nabla _H}\tilde B} \Vert_2^2 + \Vert {\nabla \tilde B} \Vert_2^2 + \Vert {\nabla {\nabla _H}\tilde B} \Vert_2^2} \right)\Vert {\nabla {{\tilde U}_\varepsilon }} \Vert_2^2\\
                  &+ C\left( {\Vert {\tilde B} \Vert_2^4 + \Vert {\tilde B} \Vert_2^2\Vert {{\nabla _H}\tilde B} \Vert_2^2 + \Vert {\nabla \tilde B} \Vert_2^4 + \Vert {\nabla \tilde B} \Vert_2^2\Vert {\nabla {\nabla _H}\tilde B} \Vert_2^2} \right)\Vert {\nabla {{\tilde U}_\varepsilon }} \Vert_2^2 + \frac{1}{{102}}\Vert {\nabla {\nabla _H}{{\tilde U}_\varepsilon }} \Vert_2^2,
    \end{aligned}
\end{equation*}

\begin{equation} \label{4.5}
        \begin{aligned}
                 {J_{12}} =& \int_\Omega  {({V_\varepsilon } \cdot \nabla )\tilde A \cdot \Delta {{\tilde U}_\varepsilon }d\Omega} \\
                 =& \int_\Omega  {\left[ {({{\tilde V}_\varepsilon } \cdot {\nabla _H})\tilde A - ({\partial _z}\tilde A)\int_0^z {({\nabla _H} \cdot {{\tilde V}_\varepsilon })d\xi } } \right] \cdot {\Delta _H}{{\tilde U}_\varepsilon }d\Omega} \\
                 &+ \int_\Omega  {[({\nabla _H} \cdot {{\tilde V}_\varepsilon }){\partial _z}\tilde A - ({\partial _z}{{\tilde V}_\varepsilon } \cdot {\nabla _H})\tilde A] \cdot {\partial _z}{{\tilde U}_\varepsilon }d\Omega} \\
                 &+ \int_\Omega  {\left[ {({\partial _{zz}}\tilde A)\int_0^z {({\nabla _H} \cdot {{\tilde V}_\varepsilon })d\xi }  - ({{\tilde V}_\varepsilon } \cdot {\nabla _H}){\partial _z}\tilde A} \right] \cdot {\partial _z}{{\tilde U}_\varepsilon }d\Omega}\\
                 \le& \int_M {\left( {\int_{ - 1}^1 {(\vert {{{\tilde V}_\varepsilon }} \vert + \vert {{\partial _z}{{\tilde V}_\varepsilon }} \vert)dz} } \right)\left( {\int_{ - 1}^1 {\vert {{\nabla _H}\tilde A} \vert\vert {{\Delta _H}{{\tilde U}_\varepsilon }} \vert dz} } \right)dxdy} \\
                 &+ \int_M {\left( {\int_{ - 1}^1 {\vert {{\nabla _H}{{\tilde V}_\varepsilon }} \vert dz} } \right)\left( {\int_{ - 1}^1 {\vert {{\partial _z}\tilde A} \vert\vert {{\Delta _H}{{\tilde U}_\varepsilon }} \vert dz} } \right)dxdy} \\
                 &+ \int_M {\left( {\int_{ - 1}^1 {(\vert {{\nabla _H}{{\tilde V}_\varepsilon }} \vert + \vert {{\nabla _H}{\partial _z}{{\tilde V}_\varepsilon }} \vert)dz} } \right)\left( {\int_{ - 1}^1 {\vert {{\partial _z}\tilde A} \vert\vert {{\partial _z}{{\tilde U}_\varepsilon }} \vert dz} } \right)dxdy} \\
                 &+ \int_M {\left( {\int_{ - 1}^1 {(\vert {{\nabla _H}\tilde A} \vert + \vert {{\nabla _H}{\partial _z}\tilde A} \vert)dz} } \right)\left( {\int_{ - 1}^1 {\vert {{\partial _z}{{\tilde V}_\varepsilon }} \vert\vert {{\partial _z}{{\tilde U}_\varepsilon }} \vert dz} } \right)dxdy} \\
                 &+ \int_M {\left( {\int_{ - 1}^1 {\vert {{\nabla _H}{{\tilde V}_\varepsilon }} \vert dz} } \right)\left( {\int_{ - 1}^1 {\vert {{\partial _{zz}}\tilde A} \vert\vert {{\partial _z}{{\tilde U}_\varepsilon }} \vert dz} } \right)dxdy} \\
                 &+ \int_M {\left( {\int_{ - 1}^1 {(\vert {{{\tilde V}_\varepsilon }} \vert + \vert {{\partial _z}{{\tilde V}_\varepsilon }} \vert)dz} } \right)\left( {\int_{ - 1}^1 {\vert {{\partial _z}{{\tilde U}_\varepsilon }} \vert\vert {{\nabla _H}{\partial _z}\tilde A} \vert dz} } \right)dxdy} \\
                 \le& C\left[ {\Vert {{{\tilde V}_\varepsilon }} \Vert_2^{\frac{1}{2}}\left( {\Vert {{{\tilde V}_\varepsilon }} \Vert_2^{\frac{1}{2}} + \Vert {{\nabla _H}{{\tilde V}_\varepsilon }} \Vert_2^{\frac{1}{2}}} \right) + \Vert {{\partial _z}{{\tilde V}_\varepsilon }} \Vert_2^{\frac{1}{2}}\left( {\Vert {{\partial _z}{{\tilde V}_\varepsilon }} \Vert_2^{\frac{1}{2}} + \Vert {{\nabla _H}{\partial _z}{{\tilde V}_\varepsilon }} \Vert_2^{\frac{1}{2}}} \right)} \right]\\
                 &\times \Vert {{\nabla _H}\tilde A} \Vert_2^{\frac{1}{2}}\left( {\Vert {{\nabla _H}\tilde A} \Vert_2^{\frac{1}{2}} + \Vert {\nabla _H^2\tilde A} \Vert_2^{\frac{1}{2}}} \right){\Vert {{\Delta _H}{{\tilde U}_\varepsilon }} \Vert_2}\\
                 &+ C\Vert {{\nabla _H}{{\tilde V}_\varepsilon }} \Vert_2^{\frac{1}{2}}\left( {\Vert {{\nabla _H}{{\tilde V}_\varepsilon }} \Vert_2^{\frac{1}{2}} + \Vert {\nabla _H^2{{\tilde V}_\varepsilon }} \Vert_2^{\frac{1}{2}}} \right)\Vert {{\partial _z}\tilde A} \Vert_2^{\frac{1}{2}}\Vert {{\nabla _H}{\partial _z}\tilde A} \Vert_2^{\frac{1}{2}}{\Vert {{\Delta _H}{{\tilde U}_\varepsilon }} \Vert_2}\\
                 &+ C{\Vert {{\partial _z}\tilde A} \Vert_2}{\Vert {{\Delta _H}{{\tilde U}_\varepsilon }} \Vert_2}\Vert {{\nabla _H}{{\tilde V}_\varepsilon }} \Vert_2^{\frac{1}{2}}\left( {\Vert {{\nabla _H}{{\tilde V}_\varepsilon }} \Vert_2^{\frac{1}{2}} + \Vert {\nabla _H^2{{\tilde V}_\varepsilon }} \Vert_2^{\frac{1}{2}}} \right)\\
                 &+ C\left( {{{\Vert {{\nabla _H}{{\tilde V}_\varepsilon }} \Vert}_2} + {{\Vert {{\nabla _H}{\partial _z}{{\tilde V}_\varepsilon }} \Vert}_2}} \right)\Vert {{\partial _z}\tilde A} \Vert_2^{\frac{1}{2}}\left( {\Vert {{\partial _z}\tilde A} \Vert_2^{\frac{1}{2}} + \Vert {{\nabla _H}{\partial _z}\tilde A} \Vert_2^{\frac{1}{2}}} \right)\\
                 &\times \Vert {{\partial _z}{{\tilde U}_\varepsilon }} \Vert_2^{\frac{1}{2}}\left( {\Vert {{\partial _z}{{\tilde U}_\varepsilon }} \Vert_2^{\frac{1}{2}} + \Vert {{\nabla _H}{\partial _z}{{\tilde U}_\varepsilon }} \Vert_2^{\frac{1}{2}}} \right)\\
                 &+ C\left( {{{\Vert {{\nabla _H}\tilde A} \Vert}_2} + {{\Vert {{\nabla _H}{\partial _z}\tilde A} \Vert}_2}} \right)\Vert {{\partial _z}{{\tilde V}_\varepsilon }} \Vert_2^{\frac{1}{2}}\left( {\Vert {{\partial _z}{{\tilde V}_\varepsilon }} \Vert_2^{\frac{1}{2}} + \Vert {{\nabla _H}{\partial _z}{{\tilde V}_\varepsilon }} \Vert_2^{\frac{1}{2}}} \right)\\
        \end{aligned}
\end{equation}
\begin{equation*}
    \begin{aligned}
                 &\times \Vert {{\partial _z}{{\tilde U}_\varepsilon }} \Vert_2^{\frac{1}{2}}\left( {\Vert {{\partial _z}{{\tilde U}_\varepsilon }} \Vert_2^{\frac{1}{2}} + \Vert {{\nabla _H}{\partial _z}{{\tilde U}_\varepsilon }} \Vert_2^{\frac{1}{2}}} \right)\\
                 &+ C\Vert {{\nabla _H}{{\tilde V}_\varepsilon }} \Vert_2^{\frac{1}{2}}\left( {\Vert {{\nabla _H}{{\tilde V}_\varepsilon }} \Vert_2^{\frac{1}{2}} + \Vert {\nabla _H^2{{\tilde V}_\varepsilon }} \Vert_2^{\frac{1}{2}}} \right)\Vert {{\partial _z}{{\tilde U}_\varepsilon }} \Vert_2^{\frac{1}{2}}\Vert {{\nabla _H}{\partial _z}{{\tilde U}_\varepsilon }} \Vert_2^{\frac{1}{2}}{\Vert {{\partial _{zz}}\tilde A} \Vert_2}\\
                 &+ C{\Vert {{\partial _z}{{\tilde U}_\varepsilon }} \Vert_2}{\Vert {{\partial _{zz}}\tilde A} \Vert_2}\Vert {{\nabla _H}{{\tilde V}_\varepsilon }} \Vert_2^{\frac{1}{2}}\left( {\Vert {{\nabla _H}{{\tilde V}_\varepsilon }} \Vert_2^{\frac{1}{2}} + \Vert {\nabla _H^2{{\tilde V}_\varepsilon }} \Vert_2^{\frac{1}{2}}} \right)\\
                 &+ C\left[ {\Vert {{{\tilde V}_\varepsilon }} \Vert_2^{\frac{1}{2}}\left( {\Vert {{{\tilde V}_\varepsilon }} \Vert_2^{\frac{1}{2}} + \Vert {{\nabla _H}{{\tilde V}_\varepsilon }} \Vert_2^{\frac{1}{2}}} \right) + \Vert {{\partial _z}{{\tilde V}_\varepsilon }} \Vert_2^{\frac{1}{2}}\left( {\Vert {{\partial _z}{{\tilde V}_\varepsilon }} \Vert_2^{\frac{1}{2}} + \Vert {{\nabla _H}{\partial _z}{{\tilde V}_\varepsilon }} \Vert_2^{\frac{1}{2}}} \right)} \right]\\
                 &\times \Vert {{\partial _z}{{\tilde U}_\varepsilon }} \Vert_2^{\frac{1}{2}}\left( {\Vert {{\partial _z}{{\tilde U}_\varepsilon }} \Vert_2^{\frac{1}{2}} + \Vert {{\nabla _H}{\partial _z}{{\tilde U}_\varepsilon }} \Vert_2^{\frac{1}{2}}} \right){\Vert {{\nabla _H}{\partial _z}\tilde A} \Vert_2}\\
                 \le& C\left( {1 + \Vert {{\nabla _H}\tilde A} \Vert_2^2{\rm{ + }}\Vert {\nabla \tilde A} \Vert_2^2 + \Vert {\nabla {\nabla _H}\tilde A} \Vert_2^2 + \Vert {{\nabla ^2}\tilde A} \Vert_2^2} \right)\left( {\Vert {\nabla {{\tilde U}_\varepsilon }} \Vert_2^2 + \Vert {\nabla {{\tilde V}_\varepsilon }} \Vert_2^2} \right)\\
                 &+ C\left( {\Vert {\nabla \tilde A} \Vert_2^4 + \Vert {\nabla \tilde A} \Vert_2^2\Vert {\nabla {\nabla _H}\tilde A} \Vert_2^2} \right)\left( {\Vert {\nabla {{\tilde U}_\varepsilon }} \Vert_2^2 + \Vert {\nabla {{\tilde V}_\varepsilon }} \Vert_2^2} \right)\\
                 &+ C\Vert {{{\tilde V}_\varepsilon }} \Vert_2^2\left( {1 + \Vert {\nabla {\nabla _H}\tilde A} \Vert_2^2} \right) + \frac{1}{{102}}\left( {\Vert {\nabla {\nabla _H}{{\tilde U}_\varepsilon }} \Vert_2^2 + \Vert {\nabla {\nabla _H}{{\tilde V}_\varepsilon }} \Vert_2^2} \right),
    \end{aligned}
\end{equation*}
and
\begin{equation} \label{4.6}
        \begin{aligned}
                {J_{13}} =& \int_\Omega  {({V_\varepsilon } \cdot \nabla ){{\tilde U}_\varepsilon } \cdot \Delta {{\tilde U}_\varepsilon }d\Omega} \\
                =& \int_\Omega  {\left[ {({{\tilde V}_\varepsilon } \cdot {\nabla _H}){{\tilde U}_\varepsilon } - ({\partial _z}{{\tilde U}_\varepsilon })\int_0^z {({\nabla _H} \cdot {{\tilde V}_\varepsilon })d\xi } } \right] \cdot {\Delta _H}{{\tilde U}_\varepsilon }d\Omega} \\
                &+ \int_\Omega  {[({\nabla _H} \cdot {{\tilde V}_\varepsilon }){\partial _z}{{\tilde U}_\varepsilon } - ({\partial _z}{{\tilde V}_\varepsilon } \cdot {\nabla _H}){{\tilde U}_\varepsilon }] \cdot {\partial _z}{{\tilde U}_\varepsilon }d\Omega} \\
                \le& \int_M {\left( {\int_{ - 1}^1 {(\vert {{{\tilde V}_\varepsilon }} \vert + \vert {{\partial _z}{{\tilde V}_\varepsilon }} \vert)dz} } \right)\left( {\int_{ - 1}^1 {\vert {{\nabla _H}{{\tilde U}_\varepsilon }} \vert\vert {{\Delta _H}{{\tilde U}_\varepsilon }} \vert dz} } \right)dxdy} \\
                &+ \int_M {\left( {\int_{ - 1}^1 {\vert {{\nabla _H}{{\tilde V}_\varepsilon }} \vert dz} } \right)\left( {\int_{ - 1}^1 {\vert {{\partial _z}{{\tilde U}_\varepsilon }} \vert\vert {{\Delta _H}{{\tilde U}_\varepsilon }} \vert dz} } \right)dxdy} \\
                &+ \int_M {\left( {\int_{ - 1}^1 {(\vert {{\nabla _H}{{\tilde V}_\varepsilon }} \vert + \vert {{\nabla _H}{\partial _z}{{\tilde V}_\varepsilon }} \vert)dz} } \right)\left( {\int_{ - 1}^1 {\vert {{\partial _z}{{\tilde U}_\varepsilon }} \vert^2 dz} } \right)dxdy} \\
                &+ \int_M {\left( {\int_{ - 1}^1 {(\vert {{\nabla _H}{{\tilde U}_\varepsilon }} \vert + \vert {{\nabla _H}{\partial _z}{{\tilde U}_\varepsilon }} \vert)dz} } \right)\left( {\int_{ - 1}^1 {\vert {{\partial _z}{{\tilde U}_\varepsilon }} \vert\vert {{\partial _z}{{\tilde V}_\varepsilon }} \vert dz} } \right)dxdy} \\
                \le& C\left[ {\Vert {{{\tilde V}_\varepsilon }} \Vert_2^{\frac{1}{2}}\left( {\Vert {{{\tilde V}_\varepsilon }} \Vert_2^{\frac{1}{2}} + \Vert {{\nabla _H}{{\tilde V}_\varepsilon }} \Vert_2^{\frac{1}{2}}} \right) + \Vert {{\partial _z}{{\tilde V}_\varepsilon }} \Vert_2^{\frac{1}{2}}\left( {\Vert {{\partial _z}{{\tilde V}_\varepsilon }} \Vert_2^{\frac{1}{2}} + \Vert {{\nabla _H}{\partial _z}{{\tilde V}_\varepsilon }} \Vert_2^{\frac{1}{2}}} \right)} \right]\\
                &\times \Vert {{\nabla _H}{{\tilde U}_\varepsilon }} \Vert_2^{\frac{1}{2}}\left( {\Vert {{\nabla _H}{{\tilde U}_\varepsilon }} \Vert_2^{\frac{1}{2}} + \Vert {\nabla _H^2{{\tilde U}_\varepsilon }} \Vert_2^{\frac{1}{2}}} \right){\Vert {{\Delta _H}{{\tilde U}_\varepsilon }} \Vert_2}\\
                &+ C\Vert {{\nabla _H}{{\tilde V}_\varepsilon }} \Vert_2^{\frac{1}{2}}\left( {\Vert {{\nabla _H}{{\tilde V}_\varepsilon }} \Vert_2^{\frac{1}{2}} + \Vert {\nabla _H^2{{\tilde V}_\varepsilon }} \Vert_2^{\frac{1}{2}}} \right)\Vert {{\partial _z}{{\tilde U}_\varepsilon }} \Vert_2^{\frac{1}{2}}\Vert {{\nabla _H}{\partial _z}{{\tilde U}_\varepsilon }} \Vert_2^{\frac{1}{2}}{\Vert {{\Delta _H}{{\tilde U}_\varepsilon }} \Vert_2}\\
                &+ C{\Vert {{\partial _z}{{\tilde U}_\varepsilon }} \Vert_2}{\Vert {{\Delta _H}{{\tilde U}_\varepsilon }} \Vert_2}\Vert {{\nabla _H}{{\tilde V}_\varepsilon }} \Vert_2^{\frac{1}{2}}\left( {\Vert {{\nabla _H}{{\tilde V}_\varepsilon }} \Vert_2^{\frac{1}{2}} + \Vert {\nabla _H^2{{\tilde V}_\varepsilon }} \Vert_2^{\frac{1}{2}}} \right)\\
                &+ C\left( {{{\Vert {{\nabla _H}{{\tilde V}_\varepsilon }} \Vert}_2} + {{\Vert {{\nabla _H}{\partial _z}{{\tilde V}_\varepsilon }} \Vert}_2}} \right){\Vert {{\partial _z}{{\tilde U}_\varepsilon }} \Vert_2}\left( {{{\Vert {{\partial _z}{{\tilde U}_\varepsilon }} \Vert}_2} + {{\Vert {{\nabla _H}{\partial _z}{{\tilde U}_\varepsilon }} \Vert}_2}} \right)\\
                &+ C\left( {{{\Vert {{\nabla _H}{{\tilde U}_\varepsilon }} \Vert}_2} + {{\Vert {{\nabla _H}{\partial _z}{{\tilde U}_\varepsilon }} \Vert}_2}} \right)\Vert {{\partial _z}{{\tilde U}_\varepsilon }} \Vert_2^{\frac{1}{2}}\left( {\Vert {{\partial _z}{{\tilde U}_\varepsilon }} \Vert_2^{\frac{1}{2}} + \Vert {{\nabla _H}{\partial _z}{{\tilde U}_\varepsilon }} \Vert_2^{\frac{1}{2}}} \right)\\
                &\times \Vert {{\partial _z}{{\tilde V}_\varepsilon }} \Vert_2^{\frac{1}{2}}\left( {\Vert {{\partial _z}{{\tilde V}_\varepsilon }} \Vert_2^{\frac{1}{2}} + \Vert {{\nabla _H}{\partial _z}{{\tilde V}_\varepsilon }} \Vert_2^{\frac{1}{2}}} \right)\\
 \end{aligned}
\end{equation}
\begin{equation*}
    \begin{aligned}
                \le& C\left( {1 + \Vert {{\nabla _H}{{\tilde U}_\varepsilon }} \Vert_2^2 + \Vert {\nabla {\nabla _H}{{\tilde U}_\varepsilon }} \Vert_2^2 + \Vert {\nabla {\nabla _H}{{\tilde V}_\varepsilon }} \Vert_2^2} \right)\left( {\Vert {\nabla {{\tilde U}_\varepsilon }} \Vert_2^2 + \Vert {\nabla {{\tilde V}_\varepsilon }} \Vert_2^2} \right)\\
                &+ C\Vert {\nabla {{\tilde U}_\varepsilon }} \Vert_2^2\left( {\Vert {{\nabla _H}{{\tilde V}_\varepsilon }} \Vert_2^2 + \Vert {{{\tilde V}_\varepsilon }} \Vert_2^4 + \Vert {{{\tilde V}_\varepsilon }} \Vert_2^2\Vert {{\nabla _H}{{\tilde V}_\varepsilon }} \Vert_2^2} \right) + \frac{1}{{102}}\left( {\Vert {\nabla {\nabla _H}{{\tilde U}_\varepsilon }} \Vert_2^2 + \Vert {\nabla {\nabla _H}{{\tilde V}_\varepsilon }} \Vert_2^2} \right),
    \end{aligned}
\end{equation*}
respectively, where we have applied the boundary condition \eqref{2.4} and symmetry condition \eqref{2.6}. Utilizing the divergence-free conditions, Lemma \ref{lem3.1} and Young's inequality again, we have from the third integral factor ${J_3}$ that
\begin{equation} \label{4.7}
        \begin{aligned}
                  {J_3} =& \int_\Omega  {[(A \cdot \nabla ){{\tilde V}_\varepsilon } + ({U_\varepsilon } \cdot \nabla )\tilde B + ({U_\varepsilon } \cdot \nabla ){{\tilde V}_\varepsilon }] \cdot \Delta {{\tilde V}_\varepsilon }d\Omega}\\
                  \le& C\left( {1 + \Vert {\tilde A} \Vert_2^2 + \Vert {{\nabla _H}\tilde A} \Vert_2^2 + \Vert {\nabla \tilde A} \Vert_2^2 + \Vert {\nabla {\nabla _H}\tilde A} \Vert_2^2} \right)\Vert {\nabla {{\tilde V}_\varepsilon }} \Vert_2^2\\
                  &+ C\left( {\Vert {\tilde A} \Vert_2^4 + \Vert {\tilde A} \Vert_2^2\Vert {{\nabla _H}\tilde A} \Vert_2^2 + \Vert {\nabla \tilde A} \Vert_2^4 + \Vert {\nabla \tilde A} \Vert_2^2\Vert {\nabla {\nabla _H}\tilde A} \Vert_2^2} \right)\Vert {\nabla {{\tilde V}_\varepsilon }} \Vert_2^2\\
                  &+ C\left( {1 + \Vert {{\nabla _H}\tilde B} \Vert_2^2 + \Vert {\nabla \tilde B} \Vert_2^2 + \Vert {\nabla {\nabla _H}\tilde B} \Vert_2^2 + \Vert {{\nabla ^2}\tilde B} \Vert_2^2} \right)\left( {\Vert {\nabla {{\tilde U}_\varepsilon }} \Vert_2^2 + \Vert {\nabla {{\tilde V}_\varepsilon }} \Vert_2^2} \right)\\
                  &+ C\left( {\Vert {\nabla \tilde B} \Vert_2^4 + \Vert {\nabla \tilde B} \Vert_2^2\Vert {\nabla {\nabla _H}\tilde B} \Vert_2^2} \right)\left( {\Vert {\nabla {{\tilde U}_\varepsilon }} \Vert_2^2 + \Vert {\nabla {{\tilde V}_\varepsilon }} \Vert_2^2} \right)\\
                  &+ C\left( {1 + \Vert {{\nabla _H}{{\tilde V}_\varepsilon }} \Vert_2^2 + \Vert {\nabla {\nabla _H}{{\tilde V}_\varepsilon }} \Vert_2^2 + \Vert {\nabla {\nabla _H}{{\tilde U}_\varepsilon }} \Vert_2^2} \right)\left( {\Vert {\nabla {{\tilde U}_\varepsilon }} \Vert_2^2 + \Vert {\nabla {{\tilde V}_\varepsilon }} \Vert_2^2} \right)\\
                  &+ C\Vert {\nabla {{\tilde V}_\varepsilon }} \Vert_2^2\left( {\Vert {{\nabla _H}{{\tilde U}_\varepsilon }} \Vert_2^2 + \Vert {{{\tilde U}_\varepsilon }} \Vert_2^4 + \Vert {{{\tilde U}_\varepsilon }} \Vert_2^2\Vert {{\nabla _H}{{\tilde U}_\varepsilon }} \Vert_2^2} \right)\\
                  &+ C\Vert {{{\tilde U}_\varepsilon }} \Vert_2^2\left( {1 + \Vert {\nabla {\nabla _H}\tilde B} \Vert_2^2} \right) + \frac{1}{{102}}\left( {3\Vert {\nabla {\nabla _H}{{\tilde V}_\varepsilon }} \Vert_2^2 + 2\Vert {\nabla {\nabla _H}{{\tilde U}_\varepsilon}} \Vert_2^2} \right).
 \end{aligned}
\end{equation}
To obtain the upper bound for the second integral factor ${J_2}$, we can split it into four parts, namely, ${J_{21}}$, ${J_{22}}$, ${J_{23}}$ and ${J_{24}}$, and then employ Lemma \ref{lem3.1} as well as Young's inequality to reach
\begin{equation} \label{4.8}
    \begin{aligned}
                   {J_{21}} =& {\varepsilon ^2}\int_\Omega  {(B \cdot \nabla {A_3}) \Delta {U_{3,\varepsilon }}d\Omega} \\
                   =& {\varepsilon ^2}\int_\Omega  {\left[ {({\nabla _H} \cdot \tilde A)\int_0^z {({\nabla _H} \cdot \tilde B)d\xi }  - B \cdot \int_0^z {{\nabla _H}({\nabla _H} \cdot \tilde A)d\xi } } \right] {\Delta _H}{U_{3,\varepsilon }}d\Omega} \\
                   &+ {\varepsilon ^2}\int_\Omega  {\left[ {({\nabla _H} \cdot \tilde A)({\nabla _H} \cdot \tilde B) - {\partial _z}\tilde B \cdot \int_0^z {{\nabla _H}({\nabla _H} \cdot \tilde A)d\xi } } \right]({\nabla _H} \cdot {{\tilde U}_\varepsilon })d\Omega} \\
                   &+ {\varepsilon ^2}\int_\Omega  {\left[ {({\nabla _H} \cdot {\partial _z}\tilde A)\int_0^z {({\nabla _H} \cdot \tilde B)d\xi }  - \tilde B \cdot {\nabla _H}({\nabla _H} \cdot \tilde A)} \right]({\nabla _H} \cdot {{\tilde U}_\varepsilon })d\Omega} \\
                   \le& C{\varepsilon ^2}\left( {\Vert {\tilde B} \Vert_2^4 + \Vert {\tilde B} \Vert_2^2\Vert {{\nabla _H}\tilde B} \Vert_2^2 + \Vert {\nabla {\nabla _H}\tilde A} \Vert_2^2 + \Vert {{\nabla ^2}{\nabla _H}\tilde A} \Vert_2^2 } \right)\\
                   &+ C{\varepsilon ^2}\left( {\Vert {{\nabla ^2}\tilde A} \Vert_2^2\Vert {{\nabla ^2}{\nabla _H}\tilde A} \Vert_2^2 + \Vert {{\nabla ^2}\tilde B} \Vert_2^2\Vert {{\nabla ^2}{\nabla _H}\tilde B} \Vert_2^2} \right)\\
                   &+ C\left( {{\varepsilon ^4}\Vert {{\nabla ^2}\tilde B} \Vert_2^2\Vert {{\nabla ^2}{\nabla _H}\tilde B} \Vert_2^2 + {\varepsilon ^2}\Vert {{\nabla ^2}{\nabla _H}\tilde B} \Vert_2^2} \right)\Vert {\nabla {{\tilde U}_\varepsilon }} \Vert_2^2\\
                   &+ C\left[ {1 + {\varepsilon ^2} + \Vert {\tilde B} \Vert_2^2 + \Vert {{\nabla _H}\tilde B} \Vert_2^2 + (1 + {\varepsilon ^2})\Vert {\nabla \tilde B} \Vert_2^2 + (1 + {\varepsilon ^2})\Vert {\nabla {\nabla _H}\tilde B} \Vert_2^2} \right]\Vert {\nabla {{\tilde U}_\varepsilon }} \Vert_2^2\\
                   &+ C\left[ {\Vert {\tilde B} \Vert_2^4 + \Vert {\tilde B} \Vert_2^2\Vert {{\nabla _H}\tilde B} \Vert_2^2 + (1 + {\varepsilon ^4})\Vert {\nabla \tilde B} \Vert_2^4 + (1 + {\varepsilon ^4})\Vert {\nabla \tilde B} \Vert_2^2\Vert {\nabla {\nabla _H}\tilde B} \Vert_2^2} \right]\Vert {\nabla {{\tilde U}_\varepsilon }} \Vert_2^2\\
                   &+ \frac{1}{{102}}\left( {{\varepsilon ^2}\Vert {\nabla {\nabla _H}{U_{3,\varepsilon }}} \Vert_2^2 + \Vert {\nabla {\nabla _H}{{\tilde U}_\varepsilon }} \Vert_2^2} \right),
\end{aligned}
\end{equation}
\begin{equation} \label{4.9}
        \begin{aligned}
                 {J_{22}} =& {\varepsilon ^2}\int_\Omega  {(B \cdot \nabla {U_{3,\varepsilon }}) \cdot \Delta {U_{3,\varepsilon }}d\Omega} \\
                 =& {\varepsilon ^2}\int_\Omega  {\left[ {\tilde B \cdot {\nabla _H}{U_{3,\varepsilon }} + ({\nabla _H} \cdot {{\tilde U}_\varepsilon })\int_0^z {({\nabla _H} \cdot \tilde B)d\xi } } \right] {\Delta _H}{U_{3,\varepsilon }}d\Omega} \\
                 &+ {\varepsilon ^2}\int_\Omega  {\left[ {{\partial _z}\tilde B \cdot \int_0^z {{\nabla _H}({\nabla _H} \cdot {{\tilde U}_\varepsilon })d\xi }  - 2\tilde B \cdot {\nabla _H}{\partial _z}{U_{3,\varepsilon }}} \right]{\partial _z}{U_{3,\varepsilon }}d\Omega} \\
                 \le& C\left[ {(1 + {\varepsilon ^4})\Vert {\nabla \tilde B} \Vert_2^4 + {\varepsilon ^2}\Vert {{\nabla ^2}{\nabla _H}\tilde B} \Vert_2^2 + {\varepsilon ^4}\Vert {{\nabla ^2}\tilde B} \Vert_2^2\Vert {{\nabla ^2}{\nabla _H}\tilde B} \Vert_2^2} \right]\Vert {\nabla {{\tilde U}_\varepsilon }} \Vert_2^2\\
                 &+ C{\varepsilon ^2}\Vert {\nabla {U_{3,\varepsilon }}} \Vert_2^2\left[ {\Vert {\tilde B} \Vert_2^2 + \Vert {{\nabla _H}\tilde B} \Vert_2^2 + (1 + {\varepsilon ^2})\Vert {\nabla \tilde B} \Vert_2^2 + (1 + {\varepsilon ^2})\Vert {\nabla {\nabla _H}\tilde B} \Vert_2^2} \right]\\
                 &+ C{\varepsilon ^2}\Vert {\nabla {U_{3,\varepsilon }}} \Vert_2^2\left[ {(1 + {\varepsilon ^4})\Vert {\nabla \tilde B} \Vert_2^4} \right] + C{\varepsilon ^2}\Vert {\nabla {U_{3,\varepsilon }}} \Vert_2^2\Vert {\nabla {\nabla _H}{{\tilde U}_\varepsilon }} \Vert_2^2\\
                 &+ C{\varepsilon ^2}\Vert {\nabla {U_{3,\varepsilon }}} \Vert_2^2\left[ {\Vert {\tilde B} \Vert_2^4 + \Vert {\tilde B} \Vert_2^2\Vert {{\nabla _H}\tilde B} \Vert_2^2 + (1 + {\varepsilon ^4})\Vert {\nabla \tilde B} \Vert_2^2\Vert {\nabla {\nabla _H}\tilde B} \Vert_2^2} \right]\\
                 &+ C{\varepsilon ^2}\left( {\Vert {\tilde B} \Vert_2^2 + \Vert {{\nabla _H}\tilde B} \Vert_2^2 + \Vert {\nabla \tilde B} \Vert_2^2 + \Vert {\nabla {\nabla _H}\tilde B} \Vert_2^2} \right)\\
                 &+ \frac{1}{{102}}\left( {{\varepsilon ^2}\Vert {\nabla {\nabla _H}{U_{3,\varepsilon }}} \Vert_2^2 + \Vert {\nabla {\nabla _H}{{\tilde U}_\varepsilon }} \Vert_2^2} \right) + C\left[ {(1 + {\varepsilon ^2})\Vert {\nabla \tilde B} \Vert_2^2} \right]\Vert {\nabla {{\tilde U}_\varepsilon }} \Vert_2^2,
         \end{aligned}
\end{equation}

\begin{equation} \label{4.10}
        \begin{aligned}
                   {J_{23}} =& {\varepsilon ^2}\int_\Omega  {({V_\varepsilon } \cdot \nabla {A_3}) \Delta {U_{3,\varepsilon }}d\Omega} \\
                   =& {\varepsilon ^2}\int_\Omega  {\left[ {({\nabla _H} \cdot \tilde A)\int_0^z {({\nabla _H} \cdot {{\tilde V}_\varepsilon })d\xi }  - {{\tilde V}_\varepsilon } \cdot \int_0^z {{\nabla _H}({\nabla _H} \cdot \tilde A)d\xi } } \right] {\Delta _H}{U_{3,\varepsilon }}d\Omega}\\
                   &+ {\varepsilon ^2}\int_\Omega  {\left[ {{\partial _z}{{\tilde V}_\varepsilon } \cdot \int_0^z {{\nabla _H}({\nabla _H} \cdot \tilde A)d\xi }  + ({\nabla _H} \cdot \tilde A){\partial _z}{V_{3,\varepsilon }}} \right] {\partial _z}{U_{3,\varepsilon }}d\Omega} \\
                   &+ {\varepsilon ^2}\int_\Omega  {\left[ {{{\tilde V}_\varepsilon } \cdot {\nabla _H}({\nabla _H} \cdot \tilde A) + ({\nabla _H} \cdot {\partial _z}\tilde A){V_{3,\varepsilon }}} \right] {\partial _z}{U_{3,\varepsilon }}d\Omega} \\
                   \le& C{\varepsilon ^2}\left( {\Vert {{\nabla ^2}{\nabla _H}\tilde A} \Vert_2^2 + \Vert {{\nabla ^2}\tilde A} \Vert_2^2\Vert {{\nabla ^2}{\nabla _H}\tilde A} \Vert_2^2} \right) + C\left( {\Vert {{{\tilde V}_\varepsilon }} \Vert_2^4 + \Vert {{{\tilde V}_\varepsilon }} \Vert_2^2\Vert {{\nabla _H}{{\tilde V}_\varepsilon }} \Vert_2^2} \right)\\
                   &+ C\Vert {\nabla {{\tilde V}_\varepsilon }} \Vert_2^2\left[ { (1 + {\varepsilon ^4})\Vert {\nabla \tilde A} \Vert_2^4 + {\varepsilon ^2}\Vert {{\nabla ^2}{\nabla _H}\tilde A} \Vert_2^2 + {\varepsilon ^4}\Vert {{\nabla ^2}\tilde A} \Vert_2^2\Vert {{\nabla ^2}{\nabla _H}\tilde A} \Vert_2^2} \right]\\
                   &+ C{\varepsilon ^2}\Vert {\nabla {U_{3,\varepsilon }}} \Vert_2^2\left( {\Vert {{{\tilde V}_\varepsilon }} \Vert_2^2 + \Vert {{\nabla _H}{{\tilde V}_\varepsilon }} \Vert_2^2 + \Vert {{{\tilde V}_\varepsilon }} \Vert_2^2\Vert {{\nabla _H}{{\tilde V}_\varepsilon }} \Vert_2^2 + \Vert {{{\tilde V}_\varepsilon }} \Vert_2^4} \right)\\
                   &+ C{\varepsilon ^2}\Vert {\nabla {U_{3,\varepsilon }}} \Vert_2^2\left( {\varepsilon ^2}\Vert {{\nabla ^2}{\nabla _H}\tilde A} \Vert_2^2\right) + C\Vert {\nabla {{\tilde V}_\varepsilon }} \Vert_2^2 {\Vert {{\nabla _H}\tilde A} \Vert_2^2}\\
                   &+ C{\varepsilon ^2}\left( {\Vert {\nabla {V_{3,\varepsilon }}} \Vert_2^2 + \Vert {\nabla {U_{3,\varepsilon }}} \Vert_2^2} \right)\left[ {\Vert {{\nabla _H}\tilde A} \Vert_2^2 + (1 + {\varepsilon ^2})\Vert {\nabla {\nabla _H}\tilde A} \Vert_2^2} \right]\\
                   &+ \frac{1}{{102}}\left( {{\varepsilon ^2}\Vert {\nabla {\nabla _H}{U_{3,\varepsilon }}} \Vert_2^2 + {\varepsilon ^2}\Vert {\nabla {\nabla _H}{V_{3,\varepsilon }}} \Vert_2^2 + \Vert {\nabla {\nabla _H}{{\tilde V}_\varepsilon }} \Vert_2^2} \right),
         \end{aligned}
\end{equation}
and
\begin{equation} \label{4.11}
    \begin{aligned}
                  {J_{24}} =& {\varepsilon ^2}\int_\Omega  {({V_\varepsilon } \cdot \nabla {U_{3,\varepsilon }}) \Delta {U_{3,\varepsilon }}d\Omega} \\
                  =& {\varepsilon ^2}\int_\Omega  {\left[ {{{\tilde V}_\varepsilon } \cdot {\nabla _H}{U_{3,\varepsilon }} - {\partial _z}{U_{3,\varepsilon }}\int_0^z {({\nabla _H} \cdot {{\tilde V}_\varepsilon })d\xi } } \right] {\Delta _H}{U_{3,\varepsilon }}d\Omega} \\
                  &+ {\varepsilon ^2}\int_\Omega  {\left[ {{\partial _z}{{\tilde V}_\varepsilon } \cdot \int_0^z {{\nabla _H}({\nabla _H} \cdot {{\tilde V}_\varepsilon })d\xi }  - 2{{\tilde V}_\varepsilon } \cdot {\nabla _H}{\partial _z}{U_{3,\varepsilon }}} \right]  {\partial _z}{U_{3,\varepsilon }}d\Omega} \\
                  \le& C{\varepsilon ^2}\Vert {\nabla {U_{3,\varepsilon }}} \Vert_2^2\left( {\Vert {{{\tilde V}_\varepsilon }} \Vert_2^2 + \Vert {{\nabla _H}{{\tilde V}_\varepsilon }} \Vert_2^2 + \Vert {{{\tilde V}_\varepsilon }} \Vert_2^4 + \Vert {{{\tilde V}_\varepsilon }} \Vert_2^2\Vert {{\nabla _H}{{\tilde V}_\varepsilon }} \Vert_2^2} \right)\\
    \end{aligned}
\end{equation}
\begin{equation*}
    \begin{aligned}
                 &+ C\left( {\Vert {\nabla {{\tilde V}_\varepsilon }} \Vert_2^2 + {\varepsilon ^2}\Vert {\nabla {U_{3,\varepsilon }}} \Vert_2^2} \right)\left( {\Vert {\nabla {\nabla _H}{{\tilde V}_\varepsilon }} \Vert_2^2 + {\varepsilon ^2}\Vert {\nabla {\nabla _H}{U_{3,\varepsilon }}} \Vert_2^2} \right)\\
                 &+ C\Vert {\nabla {{\tilde V}_\varepsilon }} \Vert_2^2\left( {{\varepsilon ^2}\Vert {{\nabla _H}{U_{3,\varepsilon }}} \Vert_2^2} \right) + \frac{1}{{102}}\left( {{\varepsilon ^2}\Vert {\nabla {\nabla _H}{U_{3,\varepsilon }}} \Vert_2^2 + \Vert {\nabla {\nabla _H}{{\tilde V}_\varepsilon }} \Vert_2^2} \right),
    \end{aligned}
\end{equation*}
where the incompressible conditions have been invoked. With the similar argument of the second integral factor ${J_2}$ on the right-hand side of \eqref{4.3}, each part of the fourth integral factor ${J_4}$ can be bounded as, respectively,
\begin{equation} \label{4.12}
    \begin{aligned}
                   {J_{41}} =& {\varepsilon ^2}\int_\Omega  {(A \cdot \nabla {B_3}) \Delta {V_{3,\varepsilon }}d\Omega} \\
                   \le& C{\varepsilon ^2}\left( {\Vert {\tilde A} \Vert_2^4 + \Vert {\tilde A} \Vert_2^2\Vert {{\nabla _H}\tilde A} \Vert_2^2 + \Vert {\nabla {\nabla _H}\tilde B} \Vert_2^2 + \Vert {{\nabla ^2}{\nabla _H}\tilde B} \Vert_2^2 } \right)\\
                   &+ C{\varepsilon ^2}\left( {\Vert {{\nabla ^2}\tilde B} \Vert_2^2\Vert {{\nabla ^2}{\nabla _H}\tilde B} \Vert_2^2 + \Vert {{\nabla ^2}\tilde A} \Vert_2^2\Vert {{\nabla ^2}{\nabla _H}\tilde A} \Vert_2^2} \right)\\
                   &+ C\left( {{\varepsilon ^4}\Vert {{\nabla ^2}\tilde A} \Vert_2^2\Vert {{\nabla ^2}{\nabla _H}\tilde A} \Vert_2^2 + {\varepsilon ^2}\Vert {{\nabla ^2}{\nabla _H}\tilde A} \Vert_2^2} \right)\Vert {\nabla {{\tilde V}_\varepsilon }} \Vert_2^2\\
                   &+ C\left[ {1 + {\varepsilon ^2} + \Vert {\tilde A} \Vert_2^2 + \Vert {{\nabla _H}\tilde A} \Vert_2^2 + (1 + {\varepsilon ^2})\Vert {\nabla \tilde A} \Vert_2^2 + (1 + {\varepsilon ^2})\Vert {\nabla {\nabla _H}\tilde A} \Vert_2^2} \right]\Vert {\nabla {{\tilde V}_\varepsilon }} \Vert_2^2\\
                   &+ C\left[ {\Vert {\tilde A} \Vert_2^4 + \Vert {\tilde A} \Vert_2^2\Vert {{\nabla _H}\tilde A} \Vert_2^2 + (1 + {\varepsilon ^4})\Vert {\nabla \tilde A} \Vert_2^4 + (1 + {\varepsilon ^4})\Vert {\nabla \tilde A} \Vert_2^2\Vert {\nabla {\nabla _H}\tilde A} \Vert_2^2} \right]\Vert {\nabla {{\tilde V}_\varepsilon }} \Vert_2^2\\
                   &+ \frac{1}{{102}}\left( {{\varepsilon ^2}\Vert {\nabla {\nabla _H}{V_{3,\varepsilon }}} \Vert_2^2 + \Vert {\nabla {\nabla _H}{{\tilde V}_\varepsilon }} \Vert_2^2} \right),
    \end{aligned}
\end{equation}

\begin{equation} \label{4.13}
        \begin{aligned}
                 {J_{42}} =& {\varepsilon ^2}\int_\Omega  {(A \cdot \nabla {V_{3,\varepsilon }}) \Delta {V_{3,\varepsilon }}d\Omega} \\
                 \le& C\left[ {(1 + {\varepsilon ^4})\Vert {\nabla \tilde A} \Vert_2^4 + {\varepsilon ^2}\Vert {{\nabla ^2}{\nabla _H}\tilde A} \Vert_2^2 + {\varepsilon ^4}\Vert {{\nabla ^2}\tilde A} \Vert_2^2\Vert {{\nabla ^2}{\nabla _H}\tilde A} \Vert_2^2} \right]\Vert {\nabla {{\tilde V}_\varepsilon }} \Vert_2^2\\
                 &+ C{\varepsilon ^2}\Vert {\nabla {V_{3,\varepsilon }}} \Vert_2^2\left[ {\Vert {\tilde A} \Vert_2^2 + \Vert {{\nabla _H}\tilde A} \Vert_2^2 + (1 + {\varepsilon ^2})\Vert {\nabla \tilde A} \Vert_2^2 + (1 + {\varepsilon ^2})\Vert {\nabla {\nabla _H}\tilde A} \Vert_2^2} \right]\\
                 &+ C{\varepsilon ^2}\Vert {\nabla {V_{3,\varepsilon }}} \Vert_2^2\left[ {(1 + {\varepsilon ^4})\Vert {\nabla \tilde A} \Vert_2^4} \right] + C{\varepsilon ^2}\Vert {\nabla {V_{3,\varepsilon }}} \Vert_2^2\Vert {\nabla {\nabla _H}{{\tilde V}_\varepsilon }} \Vert_2^2\\
                 &+ C{\varepsilon ^2}\Vert {\nabla {V_{3,\varepsilon }}} \Vert_2^2\left[ {\Vert {\tilde A} \Vert_2^4 + \Vert {\tilde A} \Vert_2^2\Vert {{\nabla _H}\tilde A} \Vert_2^2 + (1 + {\varepsilon ^4})\Vert {\nabla \tilde A} \Vert_2^2\Vert {\nabla {\nabla _H}\tilde A} \Vert_2^2} \right]\\
                 &+ C{\varepsilon ^2}\left( {\Vert {\tilde A} \Vert_2^2 + \Vert {{\nabla _H}\tilde A} \Vert_2^2 + \Vert {\nabla \tilde A} \Vert_2^2 + \Vert {\nabla {\nabla _H}\tilde A} \Vert_2^2} \right)\\
                 &+ \frac{1}{{102}}\left( {{\varepsilon ^2}\Vert {\nabla {\nabla _H}{V_{3,\varepsilon }}} \Vert_2^2 + \Vert {\nabla {\nabla _H}{{\tilde V}_\varepsilon }} \Vert_2^2} \right) + C\left[ {(1 + {\varepsilon ^2})\Vert {\nabla \tilde A} \Vert_2^2} \right]\Vert {\nabla {{\tilde V}_\varepsilon }} \Vert_2^2,
         \end{aligned}
\end{equation}

\begin{equation} \label{4.14}
        \begin{aligned}
                   {J_{43}} =& {\varepsilon ^2}\int_\Omega  {({U_\varepsilon } \cdot \nabla {B_3}) \Delta {V_{3,\varepsilon }}d\Omega} \\
                   \le& C{\varepsilon ^2}\left( {\Vert {{\nabla ^2}{\nabla _H}\tilde B} \Vert_2^2 + \Vert {{\nabla ^2}\tilde B} \Vert_2^2\Vert {{\nabla ^2}{\nabla _H}\tilde B} \Vert_2^2} \right) + C\left( {\Vert {{{\tilde U}_\varepsilon }} \Vert_2^4 + \Vert {{{\tilde U}_\varepsilon }} \Vert_2^2\Vert {{\nabla _H}{{\tilde U}_\varepsilon }} \Vert_2^2} \right)\\
                   &+ C\Vert {\nabla {{\tilde U}_\varepsilon }} \Vert_2^2\left[ { (1 + {\varepsilon ^4})\Vert {\nabla \tilde B} \Vert_2^4 + {\varepsilon ^2}\Vert {{\nabla ^2}{\nabla _H}\tilde B} \Vert_2^2 + {\varepsilon ^4}\Vert {{\nabla ^2}\tilde B} \Vert_2^2\Vert {{\nabla ^2}{\nabla _H}\tilde B} \Vert_2^2} \right]\\
                   &+ C{\varepsilon ^2}\Vert {\nabla {V_{3,\varepsilon }}} \Vert_2^2\left( {\Vert {{{\tilde U}_\varepsilon }} \Vert_2^2 + \Vert {{\nabla _H}{{\tilde U}_\varepsilon }} \Vert_2^2 + \Vert {{{\tilde U}_\varepsilon }} \Vert_2^2\Vert {{\nabla _H}{{\tilde U}_\varepsilon }} \Vert_2^2 + \Vert {{{\tilde U}_\varepsilon }} \Vert_2^4} \right)\\
                   &+ C {{\varepsilon ^2}\Vert {\nabla {V_{3,\varepsilon }}} \Vert_2^2} \left({\varepsilon ^2}\Vert {{\nabla ^2}{\nabla _H}\tilde B} \Vert_2^2\right) + C\Vert {\nabla {{\tilde U}_\varepsilon }} \Vert_2^2 {\Vert {{\nabla _H}\tilde B} \Vert_2^2}\\
                   &+ C{\varepsilon ^2}\left( {\Vert {\nabla {U_{3,\varepsilon }}} \Vert_2^2 + \Vert {\nabla {V_{3,\varepsilon }}} \Vert_2^2} \right)\left[ {\Vert {{\nabla _H}\tilde B} \Vert_2^2 + (1 + {\varepsilon ^2})\Vert {\nabla {\nabla _H}\tilde B} \Vert_2^2} \right]\\
                   &+ \frac{1}{{102}}\left( {{\varepsilon ^2}\Vert {\nabla {\nabla _H}{V_{3,\varepsilon }}} \Vert_2^2 + {\varepsilon ^2}\Vert {\nabla {\nabla _H}{U_{3,\varepsilon }}} \Vert_2^2 + \Vert {\nabla {\nabla _H}{{\tilde U}_\varepsilon }} \Vert_2^2} \right),
         \end{aligned}
\end{equation}
and
\begin{equation} \label{4.15}
    \begin{aligned}
                  {J_{44}} =& {\varepsilon ^2}\int_\Omega  {({U_\varepsilon } \cdot \nabla {V_{3,\varepsilon }}) \Delta {V_{3,\varepsilon }}d\Omega} \\
                  \le& C{\varepsilon ^2}\Vert {\nabla {V_{3,\varepsilon }}} \Vert_2^2\left( {\Vert {{{\tilde U}_\varepsilon }} \Vert_2^2 + \Vert {{\nabla _H}{{\tilde U}_\varepsilon }} \Vert_2^2 + \Vert {{{\tilde U}_\varepsilon }} \Vert_2^4 + \Vert {{{\tilde U}_\varepsilon }} \Vert_2^2\Vert {{\nabla _H}{{\tilde U}_\varepsilon }} \Vert_2^2} \right)\\
                  &+ C\left( {\Vert {\nabla {{\tilde U}_\varepsilon }} \Vert_2^2 + {\varepsilon ^2}\Vert {\nabla {V_{3,\varepsilon }}} \Vert_2^2} \right)\left( {\Vert {\nabla {\nabla _H}{{\tilde U}_\varepsilon }} \Vert_2^2 + {\varepsilon ^2}\Vert {\nabla {\nabla _H}{V_{3,\varepsilon }}} \Vert_2^2} \right)\\
                  &+ C\Vert {\nabla {{\tilde U}_\varepsilon }} \Vert_2^2\left( {{\varepsilon ^2}\Vert {{\nabla _H}{V_{3,\varepsilon }}} \Vert_2^2} \right) + \frac{1}{{102}}\left( {{\varepsilon ^2}\Vert {\nabla {\nabla _H}{V_{3,\varepsilon }}} \Vert_2^2 + \Vert {\nabla {\nabla _H}{{\tilde U}_\varepsilon }} \Vert_2^2} \right).
    \end{aligned}
\end{equation}
With the assistance of the incompressibility conditions, the H\"older inequality and the Young inequality, we can estimate the integral factors ${J_5}$ and ${J_7}$ on the right-hand side of \eqref{4.3} as, respectively,
\begin{equation} \label{4.16}
    \begin{aligned}
                 {J_5} =& \int_\Omega  {[{\varepsilon ^2}({\partial _t}{A_3} - {\Delta _H}{A_3}) \Delta {U_{3,\varepsilon }}]d\Omega} \\
                 =& \int_\Omega  {[{\varepsilon ^2}({\partial _t}{A_3} - {\Delta _H}{A_3}) ({\Delta _H}{U_{3,\varepsilon }} + {\partial _{zz}}{U_{3,\varepsilon }})]d\Omega} \\
                 \le& {\varepsilon ^2}\int_\Omega  {\left( {\int_{ - 1}^1 {(\vert {\nabla _H^3\tilde A} \vert + \vert {{\nabla _H}{\partial _t}\tilde A} \vert)dz} } \right)\vert {{\Delta _H}{U_{3,\varepsilon }}} \vert d\Omega} \\
                 &+ {\varepsilon ^2}\int_\Omega  {(\vert {\nabla _H^3\tilde A} \vert + \vert {{\nabla _H}{\partial _t}\tilde A} \vert)\vert {{\partial _z}{U_{3,\varepsilon }}} \vert d\Omega} \\
                 \le& C{\varepsilon ^2}\int_\Omega  {\left( {\int_{ - 1}^1 {{(\vert {\nabla _H^3\tilde A} \vert + \vert {{\nabla _H}{\partial _t}\tilde A} \vert)^2}dz} } \right)d\Omega}  + \frac{{4{\varepsilon ^2}}}{{51}}\Vert {{\Delta _H}{U_{3,\varepsilon }}} \Vert_2^2\\
                 &+ C{\varepsilon ^2}\left( {\Vert {\nabla _H^3\tilde A} \Vert_2^2 + \Vert {{\nabla _H}{\partial _t}\tilde A} \Vert_2^2} \right) + C{\varepsilon ^2}\Vert {{\partial _z}{U_{3,\varepsilon }}} \Vert_2^2\\
                 \le& C{\varepsilon ^2}\left( {\Vert {{\nabla ^2}{\nabla _H}\tilde A} \Vert_2^2 + \Vert {\nabla {\partial _t}\tilde A} \Vert_2^2 + \Vert {\nabla {U_{3,\varepsilon }}} \Vert_2^2} \right) + \frac{{4{\varepsilon ^2}}}{{51}}\Vert {\nabla {\nabla _H}{U_{3,\varepsilon }}} \Vert_2^2,
    \end{aligned}
\end{equation}
and
\begin{equation} \label{4.17}
    \begin{aligned}
                  {J_7} =& \int_\Omega  {[{\varepsilon ^2}({\partial _t}{B_3} - {\Delta _H}{B_3}) \Delta {V_{3,\varepsilon }}]d\Omega} \\
                  \le& C{\varepsilon ^2}\left( {\Vert {{\nabla ^2}{\nabla _H}\tilde B} \Vert_2^2 + \Vert {\nabla {\partial _t}\tilde B} \Vert_2^2 + \Vert {\nabla {V_{3,\varepsilon }}} \Vert_2^2} \right) + \frac{{4{\varepsilon ^2}}}{{51}}\Vert {\nabla {\nabla _H}{V_{3,\varepsilon }}} \Vert_2^2.
    \end{aligned}
\end{equation}
For the integral factors ${J_6}$ and ${J_8}$, a direct application of H\"older's inequality and Young's inequality yields
\begin{equation} \label{4.18}
    \begin{aligned}
                   {J_6} =& \int_\Omega  {[{\varepsilon ^\alpha }({\nabla _H} \cdot {\partial _z}\tilde A) \Delta {U_{3,\varepsilon }} - {\varepsilon ^{\alpha  - 2}}({\partial _{zz}}\tilde A) \cdot \Delta {{\tilde U}_\varepsilon }]d\Omega} \\
                   =& \int_\Omega  {[{\varepsilon ^\alpha }({\nabla _H} \cdot {\partial _z}\tilde A) {\Delta _H}{U_{3,\varepsilon }} - {\varepsilon ^{\alpha  - 2}}({\partial _{zz}}\tilde A) \cdot {\Delta _H}{{\tilde U}_\varepsilon }]d\Omega} \\
                   &+ \int_\Omega  {[{\varepsilon ^\alpha }({\nabla _H} \cdot {\partial _z}\tilde A) {\partial _{zz}}{U_{3,\varepsilon }} - {\varepsilon ^{\alpha  - 2}}({\partial _{zz}}\tilde A) \cdot {\partial _{zz}}{{\tilde U}_\varepsilon }]d\Omega} \\
                   \le& {\varepsilon ^\alpha }{\Vert {{\nabla _H}{\partial _z}\tilde A} \Vert_2}{\Vert {{\Delta _H}{U_{3,\varepsilon }}} \Vert_2} + {\varepsilon ^{\alpha  - 2}}{\Vert {{\partial _{zz}}\tilde A} \Vert_2}{\Vert {{\Delta _H}{{\tilde U}_\varepsilon }} \Vert_2}\\
                   &+ {\varepsilon ^\alpha }{\Vert {{\nabla _H}{\partial _z}\tilde A} \Vert_2}{\Vert {{\partial _{zz}}{U_{3,\varepsilon }}} \Vert_2} + {\varepsilon ^{\alpha  - 2}}{\Vert {{\partial _{zz}}\tilde A} \Vert_2}{\Vert {{\partial _{zz}}{{\tilde U}_\varepsilon }} \Vert_2}\\
                   \le& C( {{\varepsilon ^{\alpha  - 2}} + {\varepsilon ^{2\alpha  - 4}}} )\Vert {{\nabla ^2}\tilde A} \Vert_2^2 + C( {{\varepsilon ^\alpha } + {\varepsilon ^{2\alpha  - 2}}} )\Vert {\nabla {\nabla _H}\tilde A} \Vert_2^2\\
                   &+ \frac{4}{{51}}\left( {\Vert {\nabla {\nabla _H}{{\tilde U}_\varepsilon }} \Vert_2^2 + {\varepsilon ^2}\Vert {\nabla {\nabla _H}{U_{3,\varepsilon }}} \Vert_2^2 + {\varepsilon ^\alpha }\Vert {\nabla {\partial _z}{U_{3,\varepsilon }}} \Vert_2^2 + {\varepsilon ^{\alpha  - 2}}\Vert {\nabla {\partial _z}{{\tilde U}_\varepsilon }} \Vert_2^2} \right),
    \end{aligned}
\end{equation}
and
\begin{equation} \label{4.19}
    \begin{aligned}
                 {J_8} =& \int_\Omega  {[{\varepsilon ^\alpha }({\nabla _H} \cdot {\partial _z}\tilde B) \Delta {V_{3,\varepsilon }} - {\varepsilon ^{\alpha  - 2}}({\partial _{zz}}\tilde B) \cdot \Delta {{\tilde V}_\varepsilon }]d\Omega} \\
                 \le& C({\varepsilon ^{\alpha  - 2}} + {\varepsilon ^{2\alpha  - 4}})\Vert {{\nabla ^2}\tilde B} \Vert_2^2 + C({\varepsilon ^\alpha } + {\varepsilon ^{2\alpha  - 2}})\Vert {\nabla {\nabla _H}\tilde B} \Vert_2^2\\
                 &+ \frac{4}{{51}}\left( {\Vert {\nabla {\nabla _H}{{\tilde V}_\varepsilon }} \Vert_2^2 + {\varepsilon ^2}\Vert {\nabla {\nabla _H}{V_{3,\varepsilon }}} \Vert_2^2 + {\varepsilon ^\alpha }\Vert {\nabla {\partial _z}{V_{3,\varepsilon }}} \Vert_2^2 + {\varepsilon ^{\alpha  - 2}}\Vert {\nabla {\partial _z}{{\tilde V}_\varepsilon }} \Vert_2^2} \right).
    \end{aligned}
\end{equation}
Therefore, collecting the estimations for all integral factors from \eqref{4.4} to \eqref{4.19}, we can obtain that
\begin{equation*}
       \begin{aligned}
                  \frac{1}{2}\frac{{dF(t)}}{{dt}} + \frac{5}{6}{G_1}(t) + \frac{5}{6}{G_2}(t) \le& {C_1}\left[ {{H_1}(t) + {H_2}(t) + {H_3}(t) + {H_4}(t) + {H_5}(t) + {H_6}(t)} \right]F(t)\\
                  &+ {C_1}\left[ {{H_7}(t) + {H_8}(t) + {H_9}(t) + {H_{10}}(t) + {H_{11}}(t) + {H_{12}}(t)} \right],
       \end{aligned}
\end{equation*}
where we define
\begin{align*}
                   F(t):=& {\Vert {\nabla ({{\tilde U}_\varepsilon },\varepsilon {U_{3,\varepsilon }},{{\tilde V}_\varepsilon },\varepsilon {V_{3,\varepsilon }})} \Vert_{2}^2},\\
                   G_1(t):=& {\Vert {\nabla {\nabla _H}{{\tilde U}_\varepsilon }} \Vert_{2}^2 + {\varepsilon ^{\alpha  - 2}}\Vert {\nabla {\partial _z}{{\tilde U}_\varepsilon }} \Vert_{2}^2 + {\varepsilon ^2}\Vert {\nabla {\nabla _H}{U_{3,\varepsilon }}} \Vert_{2}^2 + {\varepsilon ^\alpha }\Vert {\nabla {\partial _z}{U_{3,\varepsilon }}} \Vert_{2}^2},\\
                   G_2(t):=& {\Vert {\nabla {\nabla _H}{{\tilde V}_\varepsilon }} \Vert_{2}^2 + {\varepsilon ^{\alpha  - 2}}\Vert {\nabla {\partial _z}{{\tilde V}_\varepsilon }} \Vert_{2}^2 + {\varepsilon ^2}\Vert {\nabla {\nabla _H}{V_{3,\varepsilon }}} \Vert_{2}^2 + {\varepsilon ^\alpha }\Vert {\nabla {\partial _z}{V_{3,\varepsilon }}} \Vert_{2}^2},\\
                   H_1(t):=& {(1+{\varepsilon ^2}) + \Vert {\tilde A} \Vert_2^2 + \Vert {{\nabla _H}\tilde A} \Vert_2^2 + (1 + {\varepsilon ^2})\Vert {\nabla \tilde A} \Vert_2^2} + (1 + {\varepsilon ^2})\Vert {\nabla {\nabla _H}\tilde A} \Vert_2^2\\
                   &+ \Vert {\tilde B} \Vert_2^2 + \Vert {{\nabla _H}\tilde B} \Vert_2^2 + (1 + {\varepsilon ^2})\Vert {\nabla \tilde B} \Vert_2^2 + (1 + {\varepsilon ^2})\Vert {\nabla {\nabla _H}\tilde B} \Vert_2^2,\\
                   H_2(t):=& {{\varepsilon ^2}\Vert {{\nabla ^2}{\nabla _H}\tilde A}\Vert_2^2 + {\varepsilon ^4}\Vert {{\nabla ^2}\tilde A} \Vert_2^2\Vert {{\nabla ^2}{\nabla _H}\tilde A} \Vert_2^2 + (1 + {\varepsilon ^4})\Vert {\nabla \tilde A} \Vert_2^4}\\
                   &+ {\varepsilon ^2}\Vert {{\nabla ^2}{\nabla _H}\tilde B} \Vert_2^2 + {\varepsilon ^4}\Vert {{\nabla ^2}\tilde B} \Vert_2^2\Vert {{\nabla ^2}{\nabla _H}\tilde B} \Vert_2^2 + (1 + {\varepsilon ^4})\Vert {\nabla \tilde B} \Vert_2^4,\\
                   H_3(t):=& {\Vert {\tilde A} \Vert_2^4 + \Vert {{\nabla ^2}\tilde A} \Vert_2^2 + \Vert {\tilde A} \Vert_2^2\Vert {{\nabla _H}\tilde A} \Vert_2^2} + \Vert {\tilde B} \Vert_2^4 + \Vert {{\nabla ^2}\tilde B} \Vert_2^2 + \Vert {\tilde B} \Vert_2^2\Vert {{\nabla _H}\tilde B} \Vert_2^2,\\
                   H_4(t):=& (1 + {\varepsilon ^4})\Vert {\nabla \tilde A} \Vert_2^2\Vert {\nabla {\nabla _H}\tilde A} \Vert_2^2 + (1 + {\varepsilon ^4})\Vert {\nabla \tilde B} \Vert_2^2\Vert {\nabla {\nabla _H}\tilde B} \Vert_2^2,\\
                   H_5(t):=& {\Vert {{{\tilde U}_\varepsilon }} \Vert_2^2 + \Vert {{\nabla _H}{{\tilde U}_\varepsilon }} \Vert_2^2 + {\varepsilon ^2}\Vert {{\nabla _H}{U_{3,\varepsilon }}} \Vert_2^2 + \Vert {{{\tilde U}_\varepsilon }} \Vert_2^4 + \Vert {{{\tilde U}_\varepsilon }} \Vert_2^2\Vert {{\nabla _H}{{\tilde U}_\varepsilon }} \Vert_2^2}\\
                   &+ \Vert {{{\tilde V}_\varepsilon }} \Vert_2^2 + \Vert {{\nabla _H}{{\tilde V}_\varepsilon }} \Vert_2^2 + {\varepsilon ^2}\Vert {{\nabla _H}{V_{3,\varepsilon }}} \Vert_2^2 + \Vert {{{\tilde V}_\varepsilon }} \Vert_2^4 + \Vert {{{\tilde V}_\varepsilon }} \Vert_2^2\Vert {{\nabla _H}{{\tilde V}_\varepsilon }} \Vert_2^2,\\
                   H_6(t):=& {\Vert {\nabla {\nabla _H}{{\tilde U}_\varepsilon }} \Vert_2^2 + {\varepsilon ^2}\Vert {\nabla {\nabla _H}{U_{3,\varepsilon }}} \Vert_2^2 + \Vert {\nabla {\nabla _H}{{\tilde V}_\varepsilon }} \Vert_2^2 + {\varepsilon ^2}\Vert {\nabla {\nabla _H}{V_{3,\varepsilon }}} \Vert_2^2},\\
                   H_7(t):=& {\varepsilon ^2}\left( {\Vert {\tilde A} \Vert_2^2 + \Vert {{\nabla _H}\tilde A} \Vert_2^2 + \Vert {\nabla \tilde A} \Vert_2^2 + \Vert {\nabla {\nabla _H}\tilde A} \Vert_2^2 + \Vert {{\nabla ^2}{\nabla _H}\tilde A} \Vert_2^2} \right.\\
                   &\left. { + \Vert {\tilde B} \Vert_2^2 + \Vert {{\nabla _H}\tilde B} \Vert_2^2 + \Vert {\nabla \tilde B} \Vert_2^2 + \Vert {\nabla {\nabla _H}\tilde B} \Vert_2^2 + \Vert {{\nabla ^2}{\nabla _H}\tilde B} \Vert_2^2} \right),\\
                   H_8(t):=& {\varepsilon ^2}\left( {\Vert {\tilde A} \Vert_2^4 + \Vert {\tilde A} \Vert_2^2\Vert {{\nabla _H}\tilde A} \Vert_2^2 + \Vert {{\nabla ^2}\tilde A} \Vert_2^2\Vert {{\nabla ^2}{\nabla _H}\tilde A} \Vert_2^2 + \Vert {\nabla {\partial _t}\tilde A} \Vert_2^2} \right.\\
                   &\left. { + \Vert {\tilde B} \Vert_2^4 + \Vert {\tilde B} \Vert_2^2\Vert {{\nabla _H}\tilde B} \Vert_2^2 + \Vert {{\nabla ^2}\tilde B} \Vert_2^2\Vert {{\nabla ^2}{\nabla _H}\tilde B} \Vert_2^2 + \Vert {\nabla {\partial _t}\tilde B} \Vert_2^2} \right),\\
                   H_9(t):=& {\Vert {{{\tilde U}_\varepsilon }} \Vert_2^4 + \Vert {{{\tilde U}_\varepsilon }} \Vert_2^2\Vert {{\nabla _H}{{\tilde U}_\varepsilon }} \Vert_2^2 + \Vert {{{\tilde V}_\varepsilon }} \Vert_2^4 + \Vert {{{\tilde V}_\varepsilon }} \Vert_2^2\Vert {{\nabla _H}{{\tilde V}_\varepsilon }} \Vert_2^2},\\
                   H_{10}(t):=& ({\varepsilon ^{\alpha  - 2}} + {\varepsilon ^{2\alpha  - 4}})\left( {\Vert {{\nabla ^2}\tilde A} \Vert_2^2 + \Vert {{\nabla ^2}\tilde B} \Vert_2^2} \right),\\
                   H_{11}(t):=& ({\varepsilon ^\alpha } + {\varepsilon ^{2\alpha  - 2}})\left( {\Vert {\nabla {\nabla _H}\tilde A} \Vert_2^2 + \Vert {\nabla {\nabla _H}\tilde B} \Vert_2^2} \right),\\
                   H_{12}(t):=& {\Vert {{{\tilde U}_\varepsilon }} \Vert_2^2 + \Vert {{{\tilde U}_\varepsilon }} \Vert_2^2\Vert {\nabla {\nabla _H}\tilde B} \Vert_2^2 + \Vert {{{\tilde V}_\varepsilon }} \Vert_2^2 + \Vert {{{\tilde V}_\varepsilon }} \Vert_2^2\Vert {\nabla {\nabla _H}\tilde A} \Vert_2^2}.
\end{align*}
By virtue of the assumption stated in Proposition \ref{prop4.3}
\begin{equation*}
                \mathop {\sup}\limits_{0 \le s \le t} \left( {\Vert {\nabla ( {{{\tilde U}_\varepsilon },{{\tilde V}_\varepsilon }} )} \Vert_2^2 + {\varepsilon ^2}\Vert {\nabla {U_{3,\varepsilon }}} \Vert_2^2 + {\varepsilon ^2}\Vert {\nabla {V_{3,\varepsilon }}} \Vert_2^2} \right)(s) \le \delta_0^2,
\end{equation*}
we can choose a small positive constant ${\delta_0} = \sqrt {\frac{1}{{3{C_1}}}}$ such
that
\begin{equation*}
    \begin{aligned}
                \frac{{dF(t)}}{{dt}} + {G_1}(t) + {G_2}(t) \le& {C_1}\left[ {{H_1}(t) + {H_2}(t) + {H_3}(t) + {H_4}(t) + {H_5}(t)} \right]F(t)\\
                &+ {C_1}\left[ {{H_7}(t) + {H_8}(t) + {H_9}(t) + {H_{10}}(t) + {H_{11}}(t) + {H_{12}}(t)} \right],
       \end{aligned}
\end{equation*}
which together with the Grönwall inequality, Proposition \ref{prop4.1} and Proposition \ref{prop4.2} yields that
\begin{equation*}
    \begin{aligned}
                  &F(t) + \int_0^t {{G_1}(s)ds}  + \int_0^t {{G_2}(s)ds} \\
                 \le& \exp \left( {{C_2}\int_0^t {{H_1}(s)ds}  + {C_2}\int_0^t {{H_2}(s)ds}  + {C_2}\int_0^t {{H_3}(s)ds}  + {C_2}\int_0^t {{H_4}(s)ds}}\right.\\
                 &\left.{+ {C_2}\int_0^t {{H_5}(s)ds}} \right) \times \left( {{C_2}\int_0^t {{H_7}(s)ds}  + {C_2}\int_0^t {{H_8}(s)ds} + {C_2}\int_0^t {{H_9}(s)ds}} \right.\\
                 &\left. {   + {C_2}\int_0^t {{H_{10}}(s)ds} + {C_2}\int_0^t {{H_{11}}(s)ds}  + {C_2}\int_0^t {{H_{12}}(s)ds} } \right)\\
                 \le& {C_2}{\varepsilon ^\gamma }(t + 1)\exp \left\{ {{C_2}(t + 1)\left[ {{N_4}(t) + N_4^2(t) + {N_5}(t) + N_5^2(t) + 1} \right]} \right\}\\
                 &\times \left[ {{N_4}(t) + N_4^2(t) + {N_5}(t) + N_5^2(t) + {N_4}(t){N_5}(t)} \right],
    \end{aligned}
\end{equation*}
where $\gamma = \min \{ 2,\alpha - 2\}$ with $\alpha \in (2,\infty )$ and we have invoked the fact that ${({{\tilde U}_\varepsilon },{U_{3,\varepsilon }},{{\tilde V}_\varepsilon },{V_{3,\varepsilon }})} \vert_{t = 0} = 0$. The proof Proposition \ref{prop4.3} is completed.
\end{proof}

In the following proposition, we can find a small positive number ${\varepsilon (\mathcal{T})}$ depending only on $\mathcal{T}$ to eliminate the effect of the smallness condition in Proposition \ref{prop4.3}, and thus we obtain the $H^1$-estimate on the solution $({{\tilde U}_\varepsilon },{U_{3,\varepsilon }},{{\tilde V}_\varepsilon },{V_{3,\varepsilon }})$ of system \eqref{4.2}.

\begin{proposition}\label{prop4.4}
Let ${\mathcal{T}_\varepsilon ^ *}$ denote the maximal time of the existence of the strong solution $({{\tilde A}_\varepsilon },{A_{3,\varepsilon }},{{\tilde B}_\varepsilon },{B_{3,\varepsilon }})$ to system \eqref{2.3} supplemented with \eqref{2.4}-\eqref{2.6}. Then, for any finite time $\mathcal{T} > 0$, there exists a small constant $\varepsilon(\mathcal{T}) = {\left( {\frac{{5\delta_0^2}}{{8{N_6}(\mathcal{T})}}} \right)^{1/\gamma }} > 0$ such that $\mathcal{T}_\varepsilon ^ * > \mathcal{T}$, provided that $\varepsilon \in (0,\varepsilon (\mathcal{T}))$.
Moreover, the following energy estimate holds for system \eqref{4.2}, that is,
\begin{equation*}
    \begin{aligned}
                    \mathop {\sup }\limits_{0 \le t \le \mathcal{T}}& \left( {\Vert {({{\tilde U}_\varepsilon },\varepsilon {U_{3,\varepsilon }},{{\tilde V}_\varepsilon },\varepsilon {V_{3,\varepsilon }})} \Vert_{{H^1}(\Omega )}^2} \right)(t)\\
                    &+ \int_0^{\mathcal{T}} {\left( {\Vert {{\nabla _H}{{\tilde U}_\varepsilon }} \Vert_{{H^1}(\Omega )}^2 + {\varepsilon ^{\alpha  - 2}}\Vert {{\partial _z}{U_\varepsilon }} \Vert_{{H^1}(\Omega )}^2 + \Vert {{\nabla _H}{{\tilde V}_\varepsilon }} \Vert_{{H^1}(\Omega )}^2 + {\varepsilon ^{\alpha - 2}}\Vert {{\partial _z}{V_\varepsilon }} \Vert_{{H^1}(\Omega )}^2} \right)dt} \\
                    &+ \int_0^{\mathcal{T}} {\left( {{\varepsilon ^2}\Vert {{\nabla _H}{U_{3,\varepsilon }}} \Vert_{{H^1}(\Omega )}^2 + {\varepsilon ^\alpha }\Vert {{\partial _z}{U_{3,\varepsilon }}} \Vert_{{H^1}(\Omega )}^2 + {\varepsilon ^2}\Vert {{\nabla _H}{V_{3,\varepsilon }}} \Vert_{{H^1}(\Omega )}^2 + {\varepsilon ^\alpha }\Vert {{\partial _z}{V_{3,\varepsilon }}} \Vert_{{H^1}(\Omega )}^2} \right)dt}\\
                   \le {\varepsilon ^\gamma }&\left[ {{N_5}({\mathcal{T}}) + {N_6}({\mathcal{T}})} \right],
    \end{aligned}
\end{equation*}
where ${N_5}(t)$ and ${N_6}(t)$, the non-negative continuous increasing functions defined on $[0,\mathcal{T}]$, are both independent of $\varepsilon$, and $\gamma = \min \{ 2,\alpha - 2\}$ with $\alpha \in (2,\infty )$.
\end{proposition}

\begin{proof}
For any finite time $\mathcal{T} > 0$, choosing ${\mathcal{T}_\varepsilon^\sigma } = \min \{ \mathcal{T}_\varepsilon ^ * ,\mathcal{T}\}$ and applying Proposition \ref{prop4.2}, we obtain
\begin{equation} \label{4.20}
    \begin{aligned}
                    \mathop {\sup }\limits_{0 \le t < {\mathcal{T}_\varepsilon^\sigma }} &\left( {\Vert {({{\tilde U}_\varepsilon },\varepsilon {U_{3,\varepsilon }},{{\tilde V}_\varepsilon },\varepsilon {V_{3,\varepsilon }})} \Vert_{2}^2} \right)(t)\\
                    &+ \int_0^{\mathcal{T}_\varepsilon^\sigma } {\left( {\Vert {{\nabla _H}{{\tilde U}_\varepsilon }} \Vert_{2}^2 + {\varepsilon ^{\alpha  - 2}}\Vert {{\partial _z}{U_\varepsilon }} \Vert_{2}^2 + \Vert {{\nabla _H}{{\tilde V}_\varepsilon }} \Vert_{2}^2 + {\varepsilon ^{\alpha  - 2}}\Vert {{\partial _z}{V_\varepsilon }} \Vert_{2}^2} \right)dt} \\
                    &+ \int_0^{\mathcal{T}_\varepsilon^\sigma } {\left( {{\varepsilon ^2}\Vert {{\nabla _H}{U_{3,\varepsilon }}} \Vert_{2}^2 + {\varepsilon ^\alpha }\Vert {{\partial _z}{U_{3,\varepsilon }}} \Vert_{2}^2 + {\varepsilon ^2}\Vert {{\nabla _H}{V_{3,\varepsilon }}} \Vert_{2}^2 + {\varepsilon ^\alpha }\Vert {{\partial _z}{V_{3,\varepsilon }}} \Vert_{2}^2} \right)dt}\\
                   \le {\varepsilon ^\gamma }{N_5}&(\mathcal{T}),
    \end{aligned}
\end{equation}
where
\begin{equation*}
    \begin{aligned}
                {N_5}(\mathcal{T}) =& C({\cal T} + 1)\exp \left\{ {C(\mathcal{T} + 1)\left[ {{N_4}(\mathcal{T}) + N_4^2(\mathcal{T})} \right]} \right\}\\
                &\times \left[ {{N_4}({\mathcal{T}}) + N_4^2(\mathcal{T}) + {\left( {\Vert {{{\tilde A}_0}} \Vert_2^2 + \Vert {{A_{3,0}}} \Vert_2^2 + \Vert {{{\tilde B}_0}} \Vert_2^2 + \Vert {{B_{3,0}}} \Vert_2^2} \right)^2}} \right].
    \end{aligned}
\end{equation*}
It should be noted that $\gamma  = \min \{ 2,\alpha  - 2\}$ with $\alpha \in (2,\infty )$ and some positive constant $C$ is independent of $\varepsilon$.

On the other hand, let ${\delta_0}$ be the small constant stated in Proposition \ref{prop4.3} and define
\begin{equation*}
                   t_\varepsilon ^\sigma : = \sup \left\{ {t \in (0,T_\varepsilon ^\sigma )\left| {\mathop {\sup }\limits_{0 \le s \le t} \left( {\Vert {\nabla ({{\tilde U}_\varepsilon },\varepsilon {U_{3,\varepsilon }},{{\tilde V}_\varepsilon },\varepsilon {V_{3,\varepsilon }})} \Vert_2^2} \right)(s) \le \delta _0^2} \right.} \right\}.
\end{equation*}
By Proposition \ref{prop4.3}, we have the following energy estimate
\begin{equation} \label{4.21}
    \begin{aligned}
                \mathop {\sup }\limits_{0 \le s \le t} &\left( {\Vert {\nabla ({{\tilde U}_\varepsilon },\varepsilon {U_{3,\varepsilon }},{{\tilde V}_\varepsilon },\varepsilon {V_{3,\varepsilon }})} \Vert_2^2} \right)(s)\\
                &+ \int_0^t {\left( {\Vert {\nabla {\nabla _H}{{\tilde U}_\varepsilon }} \Vert_2^2 + {\varepsilon ^{\alpha  - 2}}\Vert {\nabla {\partial _z}{{\tilde U}_\varepsilon }} \Vert_2^2 + {\varepsilon ^2}\Vert {\nabla {\nabla _H}{U_{3,\varepsilon }}} \Vert_2^2} \right)ds} \\
                &+ \int_0^t {\left( {\Vert {\nabla {\nabla _H}{{\tilde V}_\varepsilon }} \Vert_2^2 + {\varepsilon ^{\alpha  - 2}}\Vert {\nabla {\partial _z}{{\tilde V}_\varepsilon }} \Vert_2^2 + {\varepsilon ^2}\Vert {\nabla {\nabla _H}{V_{3,\varepsilon }}} \Vert_2^2} \right)ds} \\
                &+ \int_0^{t } {\left( {{\varepsilon ^\alpha }\Vert {\nabla {\partial _z}{U_{3,\varepsilon }}} \Vert_2^2 + {\varepsilon ^\alpha }\Vert {\nabla {\partial _z}{V_{3,\varepsilon }}} \Vert_2^2} \right)ds}  \le {\varepsilon ^\gamma }{N_6}(\mathcal{T}),
    \end{aligned}
\end{equation}
for any $t \in [0,t_\varepsilon ^\sigma )$, where
\begin{equation*}
    \begin{aligned}
               {N_6}(\mathcal{T}) =& C(\mathcal{T} + 1)\exp \left\{ {C(\mathcal{T} + 1)\left[ {{N_4}(\mathcal{T}) + N_4^2(\mathcal{T}) + {N_5}(\mathcal{T}) + N_5^2(\mathcal{T}) + 1} \right]} \right\}\\
               &\times \left[ {{N_4}(\mathcal{T}) + N_4^2(\mathcal{T}) + {N_5}(\mathcal{T}) + N_5^2(\mathcal{T}) + {N_4}(\mathcal{T}){N_5}(\mathcal{T})} \right].
    \end{aligned}
\end{equation*}
On account of \eqref{4.21}, we can find a small positive constant $\varepsilon(\mathcal{T}) = {\left( {\frac{{5\delta_0^2}}{{8{N_6}(\mathcal{T})}}} \right)^{1/\gamma }}$ to obtain that
\begin{equation*}
    \begin{aligned}
                \mathop {\sup }\limits_{0 \le s \le t}& \left( {\Vert {\nabla ({{\tilde U}_\varepsilon },\varepsilon {U_{3,\varepsilon }},{{\tilde V}_\varepsilon },\varepsilon {V_{3,\varepsilon }})} \Vert_2^2} \right)(s)\\
                &+ \int_0^t {\left( {\Vert {\nabla {\nabla _H}{{\tilde U}_\varepsilon }} \Vert_2^2 + {\varepsilon ^{\alpha  - 2}}\Vert {\nabla {\partial _z}{{\tilde U}_\varepsilon }} \Vert_2^2 + {\varepsilon ^2}\Vert {\nabla {\nabla _H}{U_{3,\varepsilon }}} \Vert_2^2} \right)ds} \\
                &+ \int_0^t {\left( {\Vert {\nabla {\nabla _H}{{\tilde V}_\varepsilon }} \Vert_2^2 + {\varepsilon ^{\alpha  - 2}}\Vert {\nabla {\partial _z}{{\tilde V}_\varepsilon }} \Vert_2^2 + {\varepsilon ^2}\Vert {\nabla {\nabla _H}{V_{3,\varepsilon }}} \Vert_2^2} \right)ds} \\
                &+ \int_0^{t } {\left( {{\varepsilon ^\alpha }\Vert {\nabla {\partial _z}{U_{3,\varepsilon }}} \Vert_2^2 + {\varepsilon ^\alpha }\Vert {\nabla {\partial _z}{V_{3,\varepsilon }}} \Vert_2^2} \right)ds}  \le \frac{{5\delta _0^2}}{8},
    \end{aligned}
\end{equation*}
for any $t \in [0,t_\varepsilon ^\sigma )$, provided that $\varepsilon \in (0,\varepsilon (\mathcal{T}))$, and thus the above estimate implies
\begin{equation} \label{4.22}
               \mathop {\sup}\limits_{0 \le t < t_\varepsilon ^\sigma } \left( {\Vert {\nabla ({{\tilde U}_\varepsilon },\varepsilon {U_{3,\varepsilon }},{{\tilde V}_\varepsilon },\varepsilon {V_{3,\varepsilon }})} \Vert_2^2} \right)(t) \le \frac{{5\delta _0^2}}{8} < \delta _0^2.
\end{equation}
Owing to the definition of $t_\varepsilon^\sigma$ and \eqref{4.22}, we infer that $t_\varepsilon ^\sigma = \mathcal{T}_\varepsilon^\sigma $. Therefore, we conclude that the estimation \eqref{4.21} holds for any $t \in [0,t_\varepsilon ^\sigma )$ as long as $\varepsilon \in (0,\varepsilon (\mathcal{T}))$.

Now we claim that ${\mathcal{T}_\varepsilon ^ *} > \mathcal{T}$ for any $\varepsilon \in (0,\varepsilon (\mathcal{T}))$. Recalling Proposition \ref{prop4.1}, if ${\mathcal{T}_\varepsilon ^ *} \le \mathcal{T}$, then it is clear that
\begin{equation*}
               \mathop {\limsup }\limits_{t \to {({\mathcal{T}_\varepsilon ^ *})^ - }} \left( {\Vert {\nabla ({{\tilde U}_\varepsilon },\varepsilon {U_{3,\varepsilon }},{{\tilde V}_\varepsilon },\varepsilon {V_{3,\varepsilon }})} \Vert_2} \right) = \infty,
\end{equation*}
which means that the local-in-time strong solution $(\tilde A,{A_3},\tilde B,{B_3})$ to system \eqref{2.10} can be extended beyond the maximal existence time ${\mathcal{T}_\varepsilon ^ *}$, and hence this contradicts to \eqref{4.21}. This contradiction can lead to ${\mathcal{T}_\varepsilon ^ *} > \mathcal{T}$, and therefore $\mathcal{T}_\varepsilon^\sigma = \mathcal{T}$. Furthermore, \eqref{4.20} and \eqref{4.21} ensure that the energy estimate in Proposition \ref{prop4.4} holds for any $\varepsilon \in (0,\varepsilon (\mathcal{T}))$. The proof of Proposition \ref{prop4.4} is completed.
\end{proof}

\begin{proof2.7} For any finite time $\mathcal{T} > 0$, on account of Proposition \ref{prop4.4}, there exists a small constant $\varepsilon (\mathcal{T}) = {\left( {\frac{{5\delta _0^2}}{{8{N_6}(T)}}} \right)^{1/\gamma }} > 0$ such that ${\mathcal{T}_\varepsilon ^ *} > \mathcal{T}$, provided that $\varepsilon \in (0,\varepsilon (\mathcal{T}))$, which implies that there exists a unique local-in-time strong solution $({{\tilde A}_\varepsilon },{A_{3,\varepsilon }},{{\tilde B}_\varepsilon },{B_{3,\varepsilon }})$ on $[0,\mathcal{T}]$ to system \eqref{2.3} supplemented with \eqref{2.4}-\eqref{2.6}, for any $\varepsilon \in (0,\varepsilon (\mathcal{T}))$. Let ${N_7}(\mathcal{T}) = {N_5}(\mathcal{T}) + {N_6}(\mathcal{T})$ be a non-negative continuous increasing function defined on $[0,\mathcal{T}]$ independent of $\varepsilon$. Then the following estimate holds for system \eqref{4.2}, that is,
\begin{equation*}
    \begin{aligned}
                    \mathop {\sup }\limits_{0 \le t \le \mathcal{T}}& \left( {\Vert {({{\tilde U}_\varepsilon },\varepsilon {U_{3,\varepsilon }},{{\tilde V}_\varepsilon },\varepsilon {V_{3,\varepsilon }})} \Vert_{{H^1}(\Omega )}^2} \right)(t)\\
                    &+ \int_0^{\mathcal{T}} {\left( {\Vert {{\nabla _H}{{\tilde U}_\varepsilon }} \Vert_{{H^1}(\Omega )}^2 + {\varepsilon ^{\alpha  - 2}}\Vert {{\partial _z}{U_\varepsilon }} \Vert_{{H^1}(\Omega )}^2 + \Vert {{\nabla _H}{{\tilde V}_\varepsilon }} \Vert_{{H^1}(\Omega )}^2 + {\varepsilon ^{\alpha  - 2}}\Vert {{\partial _z}{V_\varepsilon }} \Vert_{{H^1}(\Omega )}^2} \right)dt} \\
                    &+ \int_0^{\mathcal{T}} {\left( {{\varepsilon ^2}\Vert {{\nabla _H}{U_{3,\varepsilon }}} \Vert_{{H^1}(\Omega )}^2 + {\varepsilon ^\alpha }\Vert {{\partial _z}{U_{3,\varepsilon }}} \Vert_{{H^1}(\Omega )}^2 + {\varepsilon ^2}\Vert {{\nabla _H}{V_{3,\varepsilon }}} \Vert_{{H^1}(\Omega )}^2 + {\varepsilon ^\alpha }\Vert {{\partial _z}{V_{3,\varepsilon }}} \Vert_{{H^1}(\Omega )}^2} \right)dt}\\
                   \le {\varepsilon ^\gamma }& {{N_7}({\mathcal{T}})},
    \end{aligned}
\end{equation*}
where $\gamma = \min \{ 2,\alpha - 2\}$ with $\alpha \in (2,\infty )$. As a consequence, it is obvious that strong convergence results stated in Theorem \ref{thm2.7} are the direct conclusion of the above estimation. \end{proof2.7}

\subsection*{Acknowledgments}

This work was supported by the Natural Science Foundation of China (No. 11771216), the Key Research and Development Program of Jiangsu Province (Social Development) (No. BE2019725), the Qing Lan Project of Jiangsu Province and Postgraduate Research and Practice Innovation Program of Jiangsu Province (No. KYCX21\_0931).

\subsection*{Declarations}

{\bf Conflict of interest} The author declares no potential conflict of interest.

\end{document}